\documentclass[a4paper, 10pt]{amsart}
\usepackage[top=80pt,bottom=65pt,left=70pt,right=70pt]{geometry}
\usepackage{graphicx, eepic}
\usepackage{times}

\normalfont
\usepackage{enumitem}
\usepackage{bbm}
\usepackage{xcolor}
\usepackage{verbatim}
\usepackage{kotex}
\usepackage{mathrsfs}
\usepackage{amscd}
\usepackage{hhline}
\usepackage{float}
\usepackage[all]{xy}
\usepackage[subnum]{cases}
\ifpdf
\usepackage{tikz}
\fi
\usetikzlibrary{arrows,matrix,positioning}

\usepackage{makecell}
\usepackage{adjustbox}

\usepackage{amsthm,amsmath,amsfonts,amssymb}
\usepackage{hyperref}
\hypersetup{colorlinks}
\hypersetup{
    bookmarks=true,         
    unicode=false,          
    pdftoolbar=true,        
    pdfmenubar=true,        
    pdffitwindow=false,     
    pdfstartview={FitH},    
    pdftitle={My title},    
    pdfauthor={Author},     
    pdfsubject={Subject},   
    pdfcreator={Creator},   
    pdfproducer={Producer}, 
    pdfkeywords={keyword1} {key2} {key3}, 
    pdfnewwindow=true,      
    colorlinks=true,       
    linkcolor=blue,          
    citecolor=red,        
    filecolor=violet,      
    urlcolor=violet       
}
\usepackage{multirow, url}
\usepackage{color}
\usepackage[all]{xy}
\theoremstyle{plain}
\newtheorem{theorem}{Theorem}[section]

\newtheorem{proposition}[theorem]{Proposition}
\newtheorem{lemma}[theorem]{Lemma}

\newtheorem{corollary}[theorem]{Corollary}

\theoremstyle{definition}

\newtheorem{definition}[theorem]{Definition}

\newtheorem{example}[theorem]{Example}

\theoremstyle{remark}
\newtheorem{remark}[theorem]{Remark}

\newcommand{\C}{\mathbb{C}}

\newcommand{\Q}{\mathbb{Q}}
\newcommand{\R}{\mathbb{R}}

\newcommand{\Z}{\mathbb{Z}}
\newcommand{\p}{\mathbb{C}P}

\newcommand{\pp}{\mathbb{P}}

\newcommand{\mcal}{\mathcal}

\def\mcal{\mathcal}
\def\frak{\mathfrak}

\newcommand{\ds}{\displaystyle}
\newcommand{\vs}{\vspace}
\newcommand{\hs}{\hspace}

\numberwithin{equation}{section} \numberwithin{table}{section}

\begin{document}                                                                          

\title{Classification of six dimensional monotone symplectic manifolds admitting semifree circle actions III}
\author{Yunhyung Cho}
\address{Department of Mathematics Education, Sungkyunkwan University, Seoul, Republic of Korea. }
\email{yunhyung@skku.edu}

\begin{abstract}
	In this paper, we complete the classification of six-dimensional closed monotone symplectic manifolds admitting semifree Hamiltonian $S^1$-actions. 
	We also show that every such manifold is $S^1$-equivariantly symplectomorphic to some K\"{a}hler Fano manifold with a certain holomorphic Hamiltonian circle action. 
\end{abstract}
\maketitle
\setcounter{tocdepth}{1} 
\tableofcontents

\section{Introduction}
\label{secIntroduction}

In a series of papers \cite{Cho1, Cho2}, the author studied the existence of K\"{a}hler structure on a six-dimensional closed symplectic manifold $(M,\omega)$ admitting a Hamiltonian circle action. 
More precisely, the author proved that any six-dimensional closed {\em monotone} {\em semifree}\footnote{We call an $S^1$-action {\em semifree} 
if the action is free outside the fixed point set.}$(M,\omega)$ Hamiltonian $S^1$-manifold $(M,\omega)$ 
is $S^1$-equivariantly symplectomorphic to some K\"{a}hler Fano manifold with a certain holomorphic Hamiltonian circle action in the case where 
\begin{itemize}
	\item (at least) one of extremal fixed components is an isolated point, or 
	\item any extremal fixed component is two dimensional.
\end{itemize}
Moreover, it turned out that there are 18 types and 21 types of such manifolds up to $S^1$-equivariant symplectomorphism, respectively.
In this paper (sequel to \cite{Cho1, Cho2}), we complete the classification and prove the following.

\begin{theorem}\label{theorem_main}
	Let $(M,\omega)$ be a six-dimensional closed monotone symplectic manifold equipped with a semifree Hamiltonian circle action. 
	Then $(M,\omega)$ is $S^1$-equivariantly symplectomorphic to some K\"{a}hler Fano manifold with a certain holomorphic Hamiltonian circle action. 
	Indeed, there are (18 + 21 + 56) types of such manifolds up to $S^1$-equivariant symplectomorphism.
\end{theorem}

We note that, without the monotonicity condition, one can easily construct infinitely many six-dimensional non-K\"{a}hler semifree Hamiltonian $S^1$-manifolds such as $KT \times S^2$ where $KT$ denotes the Kodaira-Thurston manifold \cite{Th}. On the other hand, there are infinitely many non-equivariant Hamiltonian $S^1$-actions on a fixed symplectic manifold, even in case of $\p^2$. Thus the conditions ``{\em monotonicity}'' and ``{\em semifreeness}'' are both essential to prove the finiteness of the number of such manifolds in Theorem \ref{theorem_main}.

\subsection{Summary of classification}
\label{ssecSummaryOfClassification}

	\begin{figure}[H]
		\scalebox{0.55}{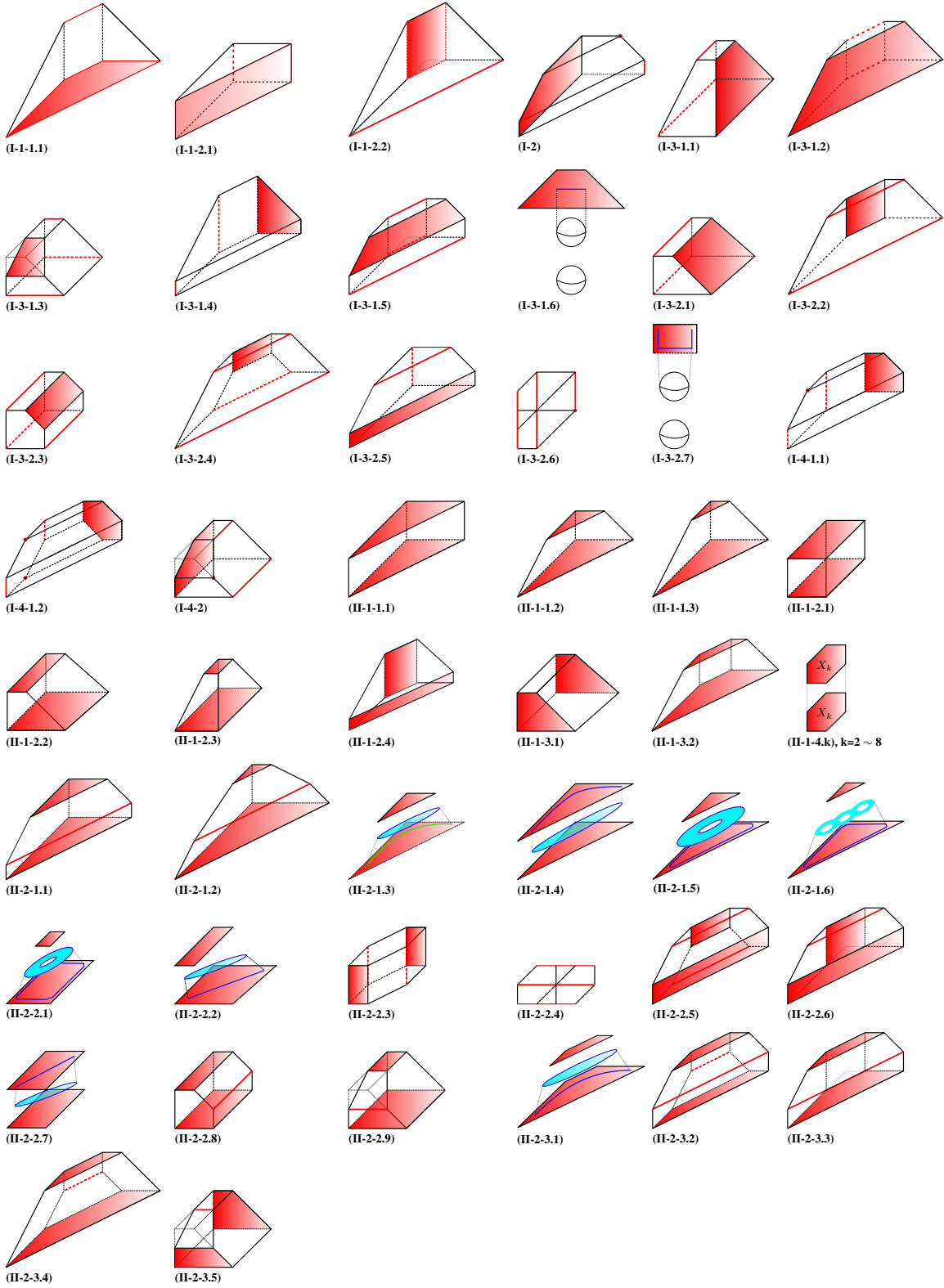}
		\caption{\label{figure_III_total} Semifree $S^1$-Fano 3-folds with $\dim Z_{\min} \geq 2$ and $\dim Z_{\max} = 4$}
	\end{figure}

Figure \ref{figure_III_total} illustrates all smooth Fano 3-folds admitting holomorphic semifree Hamiltonian $S^1$-action such that $\dim Z_{\min} \geq 2$ and $\dim Z_{\max} = 4$
where each of them is labeled and
the corresponding fixed point data can be found in Table \ref{table_list_1} and Table \ref{table_list_2}. Indeed, they are {\em all} six-dimensional closed monotone semifree
Hamiltonian $S^1$-manifold by Theorem \ref{theorem_main}.

\subsection{Sketch of the proof}
\label{ssecSketchOfTheProof}

	The proof of Theorem \ref{theorem_main} consists of three steps : 

	\begin{itemize}
		\item ({\bf Step 1}) Classify all {\em topological fixed point data}\footnote{A {\em topological fixed point data}, or TFD for short,
			is a topological analogue of a fixed point data in the sense that it records ``homology classes'', not embeddings themselves, of fixed components in reduced spaces.
			See Definition \ref{definition_topological_fixed_point_data}.}. 
			In this process, we obtain a complete list of all topological fixed point data as described in Table \ref{table_list_1} and \ref{table_list_2}.
		\item ({\bf Step 2}) Show that each topological fixed point data determines a unique {\em fixed point data}\footnote{See Definition \ref{definition_fixed_point_data}.}.
		\item ({\bf Step 3}) For each topological fixed point data listed in Table \ref{table_list_1} and \ref{table_list_2}, there exists a corresponding smooth Fano variety
			with a holomorphic semifree Hamiltonian $S^1$-action.
	\end{itemize}
	Then Theorem \ref{theorem_main} immediately follows from the Gonzalez's theorem \ref{theorem_G} which states that the fixed point data completely determines a whole manifold
	under the assumption that 
	\[
		\text{``every reduced space is symplectically rigid.''}
	\]
	It is known that the reduced space $M_0$ at level zero inherits the monotone symplectic form $\omega_0$ for the balanced moment map (Lemma \ref{lemma_balanced}). 
	That is, $(M_0, \omega_0)$ becomes a closed monotone symplectic four manifold. Therefore, by the classification of Ohta-Ono \cite{OO2}, $M_0$ is a del Pezzo surface whose 
	diffeomorphism type is one of the followings: 
	\[
		\p^2, S^2 \times S^2, X_k : \text{blow-up of $\p^2$ at $k$ generic points for $1 \leq k\leq 8$}.
	\] 
	Among those manifolds, it is known that $\p^2, S^2 \times S^2,$ and $X_k$ for $1 \leq k \leq 4$ are known to be {\em symplectically rigid} (see Definition \ref{ssecSymplecticRigidity}). 
	We crucially used this fact in the previous works \cite{Cho1, Cho2} where no $X_k$ (with $k > 4$) appears as a reduced space. 
	
	On the other hand, unlike the case of \cite{Cho1, Cho2}, $X_k$ ($k > 4$) appears as a reduced space so that the Gonzalez's theorem cannot be applied directly.
	(See {\bf (II-1-4.k), k=2 $\sim$ 8} in Figure \ref{figure_III_total}.) Fortunately, we can show that in each case {\bf (II-1-4.k)}, the manifold $M$ becomes the product space 
	$X_k \times S^2$ where the $S^1$ acts on the second factor. See Section \ref{secMainTheorem}: Proof of Theorem \ref{theorem_main}.
	
	This paper is organized as follows. In Section \ref{secNotationsAndPreliminaries}, we provide basic materials briefly and build up some notations.
	In Section \ref{secTopologicalFixedPointDataDimZMin2} and \ref{secTopologicalFixedPointDataDimZMax4}, we classify all topological fixed point data and obtain
	Table \ref{table_list_1} and \ref{table_list_2}. And we give the proof of Theorem \ref{theorem_main} in Section \ref{secMainTheorem}.
	
\subsection*{Acknowledgements} 
The author would like to thank Dmitri Panov for bringing the paper \cite{Z} to my attention.
The author would also like to thank Jinhyung Park for helpful comments. 
This work is supported by the National Research Foundation of Korea(NRF) grant funded by the Korea government(MSIP; Ministry of Science, ICT \& Future Planning) (NRF-2017R1C1B5018168).

\section{Notations and Preliminaries}
\label{secNotationsAndPreliminaries}

	In this section, we set up notations and recall some basic materials. For basic equivariant cohomology theory of Hamiltonian circle actions (such as 
	equivariant formality and the ABBV-localization theorem), we refer to \cite[Section 3]{Cho1} 
	and the references therein.

\subsection{Notations}
\label{ssecNotations}
	
	Let $(M,\omega)$ be a closed symplectic manifold admitting a Hamiltonian circle action and let $H : M \rightarrow \R$ be a moment map. We use the following notations.
	\begin{itemize}
		\item $M^{S^1}$ : fixed point set.
		\item $\mathrm{Crit}~H$ : the set of critical values of $H$. 
		\item $\mathrm{Crit}~ \mathring{H}$ : the set of non-extremal critical values of $H$.
		\item $Z_{\min}, Z_{\max}$ : the minimal and maximal fixed component with respect to $H$.
		\item $M_t := H^{-1}(t) / S^1$ : the reduced space at level $t \in [a,b]$
		\item $\omega_t$ : the reduced symplectic form on $M_t$.
		\item $P_c^{\pm}$  : the principal $S^1$-bundle $\pi_{c \pm \epsilon} : H^{-1}(c \pm \epsilon) \rightarrow M_{c \pm \epsilon}$ where $\epsilon > 0$ is sufficiently small.
		\item $e(P_c^{\pm}) \in H^2(M_{c \pm \epsilon}; \Q)$ : the Euler class of $P_c^{\pm}$.
		\item $Z_c$ : the fixed point set lying on the level set $H^{-1}(c)$. That is, $Z_c = M^{S^1} \cap H^{-1}(c)$.
		\item $\R[\lambda]$ : cohomology ring of $H^*(BS^1;\R)$, where $-\lambda$ is the Euler class of the universal Hopf bundle $ES^1 \rightarrow BS^1.$
	\end{itemize}

\subsection{Balanced moment map}
\label{ssecBalancedMomentMap}

	From now on, we assume that $(M,\omega)$ is a six-dimensional closed monotone symplectic manifold such that $c_1(TM) = [\omega]$ and it admits a 
	Hamiltonian $S^1$-action.
	
	\begin{lemma}\cite[Proposition 4.4 and Corollary 4.6]{Cho1}\label{lemma_balanced}
		There exists a unique moment map $H : M \rightarrow \R$, called the {\em balanced moment map}, which satisfies 
		\[
			H(Z) = -\Sigma(Z)
		\]
		for every fixed component $Z \subset M^{S^1}$. Here $\Sigma(Z)$ denotes the sum of all weights of the $S^1$-action at $Z$. 
	\end{lemma}
	
	Recall that an $S^1$-action on $M$ is called {\em semifree} if the action is free outside the fixed point set. It is equivalent to saying that every weight at any fixed point in $M^{S^1}$
	is either $\pm 1$ or $0$. Thus we have the following. 

	\begin{lemma}\label{lemma_possible_critical_values}\cite[Lemma 5.9]{Cho1} 
	Let $(M,\omega)$ be a six-dimensional closed monotone semifree $S^1$-manifold with the balanced moment map $H$.
	Then all possible critical values of $H$ are $\pm 3, \pm 2, \pm 1$, and $0$. Moreover, any connected component $Z$ of $M^{S^1}$ satisfies one of the followings : 
	\begin{table}[H]
		\begin{tabular}{|c|c|c|c|}
		\hline
		    $H(Z)$ & $\dim Z$ & $\mathrm{ind}(Z)$ & $\mathrm{Remark}$ \\ \hline 
		    $3$ &  $0$ & $6$ & $Z = Z_{\max} = \mathrm{point}$ \\ \hline
		    $2$ &  $2$ & $4$ & $Z = Z_{\max} \cong S^2$ \\ \hline
		    $1$ &  $4$ & $2$ & $Z = Z_{\max}$ \\ \hline
		    $1$ &  $0$ & $4$ & $Z = \mathrm{pt}$ \\ \hline
		    $0$ &  $2$ & $2$ & \\ \hline
		    $-1$ &  $0$ & $2$ & $Z = \mathrm{pt}$ \\ \hline
		    $-1$ &  $4$ & $0$ & $Z = Z_{\min}$ \\ \hline
		    $-2$ &  $2$ & $0$ & $Z = Z_{\min} \cong S^2$ \\ \hline
		    $-3$ &  $0$ & $0$ & $Z = Z_{\min} = \mathrm{point}$ \\ \hline
		\end{tabular}
		\vs{0.2cm}
		\caption{\label{table_fixed} List of possible fixed components}
		\end{table}
	\end{lemma}

	One important feature of the balanced moment map is that the reduced space $M_0$ at level zero inherits the monotonicity of $\omega$. 
	
	\begin{lemma}\cite[Proposition 4.8]{Cho1}\label{lemma_monotonicity}
		Let $H$ be the balanced moment map on $(M,\omega)$. Then the reduced space $(M_0, \omega_0)$ becomes a closed monotone symplectic four manifold such that 
		$c_1(TM_0) = [\omega_0]$. 
	\end{lemma}

	Lemma \ref{lemma_monotonicity}, together with the results of Ohta-Ono \cite{OO2}, implies that $M_0$ is diffeomorphic to a del Pezzo surface, that is, 
	$\p^2$, $S^2 \times S^2$, or $X_k$ for $1 \leq k \leq 8$ where $X_k$ denotes the blow-up of $\p^2$ at $k$ generic points. Moreover, by the uniqueness theorem 
	(of McDuff \cite{McD3}) of a symplectic structure 
	on a rational surface, it follows that $(M_0,\omega_0)$ is symplectomorphic to a K\"{a}hler Fano surface, which is a del Pezzo surface equipped with the 
	K\"{a}hler form representing the anti-canonical class of $M_0$. 
	
\subsection{Topology of reduced spaces}
\label{ssecTopologyOfReducedSpaces}
	
	Due to the result of Guillemin-Sternberg \cite{GS}, any two reduced spaces of a semifree Hamiltonian $S^1$-manifold are related by a sequence of blow-up and down, where the relation is called 
	{\em birational equivalence of reduced spaces}.
	 
	\begin{proposition}\cite{McD2}\cite{GS}\label{proposition_GS}
		Let $(M,\omega)$ be a closed semifree Hamiltonian $S^1$-manifold with a moment map $\mu : M \rightarrow \R$ and $c \in \R$ be a critical value of $\mu$. 
		If $Z_c := \mu^{-1}(c) \cap M^{S^1}$ consists of index-two (co-index two, resp.) fixed components, then $M_c = \mu^{-1}(c) / S^1$ is smooth and is diffeomorphic to 
		$M_{c-\epsilon}$
		($M_{c+\epsilon}$, resp.). Also, $M_{c+\epsilon}$ is the blow-up (blow-down, resp.) of $M_c$ along $Z_c$.	
	\end{proposition}
	
	In dimension six, any non-extremal fixed component is of index or co-index two. So, Lemma \ref{lemma_possible_critical_values} and Proposition \ref{proposition_GS} imply that
	\begin{enumerate}
		\item if $-1$ is a non-extremal critical value of $H$, then $M_{-1+\epsilon}$ is the $k$ points blow-up of $M_{-1-\epsilon}$ where $k = |Z_{-1}|$ is the number of fixed points 
		of index two,
		\item if $0$ is a non-extremal critical value of $H$, then $M_{0 - \epsilon}$ and $M_{0 + \epsilon}$ are diffeomorphic, 
		\item if $1$ is a a non-extremal critical value of $H$, then $M_{1-\epsilon}$ is the $k$-points blow-up of $M_{1+\epsilon}$ where $k = |Z_1|$.
	\end{enumerate}	

	Now, we investigate how a reduced symplectic form $\omega_t$ varies along $t$. 	
	The Duistermaat-Heckman theorem \cite{DH} states that if an interval $I$ consists of regular values of $H$, then 
	\begin{equation}\label{equation_DH}
		[\omega_s] - [\omega_t] = (t-s)e \quad \quad \text{for any $t, s \in I$}
	\end{equation}
	where $e$ is the Euler class of the principal $S^1$-bundle $\pi_t : H^{-1}(t) \rightarrow M_t$.  
	The following proposition and the Duistermaat-Heckman theorem then enable us to compute all cohomology classes of reduced symplectic forms. 

	\begin{lemma}\cite[Theorem 13.2]{GS}\label{lemma_Euler_class}
		Suppose that $Z_c = M^{S^1} \cap H^{-1}(c)$ consists of fixed components $Z_1, \cdots, Z_k$ each of which is of index two.   
		Let $e^{\pm}$ be the Euler classes of principal $S^1$-bundles $\pi_{c \pm \epsilon} : H^{-1}(c \pm \epsilon) \rightarrow M_{c \pm \epsilon}$.
		Then 
		\[
			e^+ = \phi^*(e^-) + E \in H^2(M_{c+\epsilon}; \Z)
		\] where $\phi : M_{c+\epsilon} \rightarrow M_{c-\epsilon}$ is the blow-down map and $E$ denotes the Poincar\'{e} dual of the exceptional divisor of $\phi$.
	\end{lemma}

\subsection{Symplectic rigidity}
\label{ssecSymplecticRigidity}

	A manifold $B$ is called {\em symplectically rigid} if it satisfies the following three conditions:
	\begin{itemize}
		\item (uniqueness) any two cohomologous symplectic forms are diffeomorphic, 
		\item (deformation implies isotopy) every path $\omega_t$ ($t \in [0,1]$) of symplectic forms such that $[\omega_0] = [\omega_1]$ 
		can be homotoped through families of symplectic forms 
		with the fixed endpoints $\omega_0$ and $\omega_1$ to an isotopy, that is, a path $\omega_t'$ such that $[\omega_t']$ is constant in $H^2(B)$. 
		\item For every symplectic form $\omega$ on $B$, the group $\text{Symp}(B,\omega)$ of symplectomorphisms that act trivially on $H_*(B;\Z)$ is path-connected.
	\end{itemize}
	See also \cite[Definition 2.13]{McD2} or \cite[Definition 1.4]{G}. 
	
	\begin{theorem}\cite[Theorem 1.2]{McD4}\label{theorem_uniqueness}
		Let $M$ be the blow-up of a rational or ruled symplectic four manifold. Then any two cohomologous and deformation equivalent\footnote{Two symplectic forms $\omega_0$ and 
		$\omega_1$ are said to be {\em deformation equivalent} if there exists a family of symplectic forms $\{ \omega_t  ~|~  0 \leq t \leq 1 \}$ connecting $\omega_0$ and $\omega_1$.
		 We also say that 
		$\omega_0$ and $\omega_1$ are {\em isotopic} if such a family can be chosen such that $[\omega_t]$ is a constant path in $H^2(M; \Z)$.}
		symplectic forms on $M$ are isotopic.
	\end{theorem}

	\begin{theorem}\label{theorem_symplectomorphism_group}
	For any of the following symplectic manifolds, the group of symplectomorphisms  which act trivially on homology is path-connected. 
		\begin{itemize}
			\item $\p^2$ with the Fubini-Study form. \cite[Remark in p.311]{Gr}
			\item $\p^1 \times \p^1$ with any symplectic form. \cite[Theorem 1.1]{AM}
			\item $\p^2 \# ~k~\overline{\p^2}$ with any blow-up symplectic form for $k \leq 4$. \cite[Theorem 1.4]{AM}, \cite{E}, \cite{LaP}, \cite{Pin} \cite{LLW}. 
		\end{itemize}		
	(See also \cite[Lemma 4.2]{G}.)
	\end{theorem}
	
	As a consequence of Theorem \ref{theorem_uniqueness}, Theorem \ref{theorem_symplectomorphism_group}, and the uniqueness of the symplectic form of rational surfaces \cite{McD3}, 
	we obtain the following.
	
	\begin{corollary}\label{corollary_symplectically_rigid}
		The following manifolds are symplectically rigid : 
		\[
			\p^2, \quad S^2 \times S^2, \quad X_k ~(1\leq k \leq 4)
		\]
	\end{corollary}

\subsection{Fixed point data}
\label{ssecFixedPointData}
	
	Recall the definition of the {\em fixed point data} of $(M,\omega,H)$ as follows.

	\begin{definition}\cite[Definition 1.2]{G}\label{definition_fixed_point_data} 
		A {\em fixed point data} (or {\em FD} shortly) of $(M,\omega, H)$, denoted by $\frak{F}(M, \omega, H)$, is a collection 
		\[
			 \frak{F} (M, \omega, H) := \left\{(M_{c}, \omega_c, Z_c^1, Z_c^2, \cdots,  Z_c^{k_c}, e(P_{c}^{\pm})) ~|~c \in \mathrm{Crit} ~H \right\}
		\]
		which consists of the information below.
		\begin{itemize}
			\item $(M_c, \omega_c)$\footnote{$M_c$ is smooth manifold under  the assumption that the action is semifree and the dimension of $M$ is six.
				See Proposition \ref{proposition_GS}.} is the symplectic reduction at level $c$.
			\item $k_c$ is the number of fixed components on the level $c$. 
			\item Each $Z_c^i$ is a connected fixed component and hence a symplectic submanifold of $(M_c, \omega_c)$ via the embedding
				\[
					Z_c^i \hookrightarrow H^{-1}(c) \rightarrow H^{-1}(c) / S^1 = M_c.
				\]
				(This information contains a normal bundle of $Z_c^i$ in $M_c$.)
			\item The Euler class $e(P_c^{\pm})$ of principal $S^1$-bundles $H^{-1}(c \pm \epsilon) \rightarrow M_{c \pm \epsilon}$.
		\end{itemize}		
	\end{definition}
	
	Also, we recall the definition of {\em topological fixed point data}, a topological analogue of a fixed point data. 
	
	\begin{definition}\label{definition_topological_fixed_point_data}\cite[Definition 5.7]{Cho1}
		The {\em topological fixed point data} (or {\em TFD} for short) of $(M,\omega, H)$, denoted by $\frak{F}_{\text{top}}(M, \omega, H)$, is defined as a collection 
		\[
			 \frak{F}_{\text{top}}(M, \omega, H) := \left\{(M_{c}, [\omega_c], \mathrm{PD}(Z_c^1), \mathrm{PD}(Z_c^2), \cdots, \mathrm{PD}(Z_c^{k_c}), e(P_c^{\pm}) ) ~|~c \in
			  \mathrm{Crit} ~H \right\}
		\]
		where 
		\begin{itemize}
			\item $(M_c, \omega_c)$ is the reduced symplectic manifold at level $c$, 
			\item $k_c$ is the number of fixed components at level $c$, 
			\item each $Z_c^i$ is a connected fixed component lying on the level $c$ and $\mathrm{PD}(Z_c^i) \in H^*(M_c)$ 
			denotes the Poincar\'{e} dual class of the image of the embedding
				\[
					Z_c^i \hookrightarrow H^{-1}(c) \rightarrow H^{-1}(c) / S^1 = M_c.
				\]
			\item the Euler class $e(P_c^{\pm})$ of principal $S^1$-bundles $H^{-1}(c \pm \epsilon) \rightarrow M_{e \pm \epsilon}$.
		\end{itemize}		
	\end{definition}
	
	Gonzalez proved that the fixed point data is a complete invariant of a six-dimensional 
	semifree Hamiltonian $S^1$-manifold under the assumption that every reduced space
	is symplectically rigid. 
	
	\begin{theorem}\cite[Theorem 1.5]{G}\label{theorem_G}
		Let $(M,\omega)$ be a six-dimensional closed semifree Hamiltonian $S^1$-manifold such that every critical level is simple\footnote{
		A critical level is called {\em simple} if every fixed component in the level set has a common Morse-Bott index.}.
		Suppose further that every reduced space is symplectically rigid.
		Then $(M,\omega)$ is determined by its fixed point data up to $S^1$-equivariant symplectomorphism.
	\end{theorem}
	
	It is notable that the classification of fixed point data is extremely difficult since we have to classify symplectic embeddings of each fixed component into a reduced space. 
	On the other hand, the classification of topological fixed point data is relatively easy as it is {\em computable} using various theorems such as the Duistermaat-Heckman theorem
	\eqref{equation_DH}.
	In the rest of the paper,
	we will classify all topological fixed point data and show that each data determines a fixed point data uniquely. So, the topological fixed point data
	becomes a complete invariant of $(M,\omega)$ in our cases.
	
\subsection{Exceptional classes}
\label{ssecExceptionalClasses}	
	
	We finally give the complete list of cohomology classes on $X_k$ (for $1 \leq k \leq 8$) which can be represented by a symplectic sphere with self-intersection $-1$. 
	The following lemma will be frequently used later.
	
	\begin{lemma}\label{lemma_list_exceptional}\cite[Section 2]{McD2}
		Let $X_k$ be the $k$-times simultaneous symplectic blow-up of $\p^2$ with the exceptional divisors $C_1, \cdots, C_k$.
		We denote by $E_i := \mathrm{PD}(C_i)  \in H^2(M_0; \Z)$ the dual classes, called {\em exceptional classes}.
		Then all possible exceptional classes are listed as follows (modulo permutations of indices) : 
			\[
				\begin{array}{l}
					E_1, u - E_{12},  \quad 2u - E_{12345}, \quad 3u - 2E_1 - E_{234567}, \quad 4u - 2E_{123} - E_{45678}  \\ \vs{0.1cm}
					5u - 2E_{123456}  - E_{78},  \quad 6u - 3E_1 - 2E_{2345678}  \\
				\end{array}
			\]
		Here, $u$ is the positive generator of $H^2(\p^2; \Z)$ and $E_{j \cdots n} := \sum_{i=j}^n E_i$. Furthermore, elements involving $E_i$ do not appear in $X_k$ with $k < i$. 
	\end{lemma}

\section{Topological fixed point data : $\dim Z_{\min} = 2$}
\label{secTopologicalFixedPointDataDimZMin2}

	Let $(M,\omega)$ be a six dimensional closed monotone symplectic manifold with $c_1(TM) = [\omega]$ admitting a semifree Hamiltonian circle action
	and $H : M \rightarrow \R$ be the balanced moment map for the action. 
	In this section we classify all possible topological fixed point data of $(M,\omega)$ when
	$\dim Z_{\min} = 2$ and $\dim Z_{\max} = 4$, or equivalently, $H(Z_{\min}) = -2$ and $H(Z_{\max}) = 1$, see Lemma \ref{lemma_possible_critical_values}.
	
	We first note that $Z_{\min} \cong S^2$ by Lemma \ref{lemma_possible_critical_values} so that the reduced space $M_{-2 + \epsilon}$ 
	near the minimum is an $S^2$-bundle over $S^2$. 
	We denote by $x$ and $y$ the dual classes of the fiber and the base of the bundle $M_{-2 + \epsilon} \rightarrow Z_{\min}$,
	respectively, such that 
	\[
		\begin{cases}
					\langle xy, [M_{-2 + \epsilon}] \rangle = 1, \quad \langle x^2, [M_{-2 + \epsilon}] \rangle = \langle y^2, [M_{-2 + \epsilon}] \rangle = 0 &
						\text{if $M_{-2 + \epsilon} \cong S^2 \times S^2$} \\
					\langle xy, [M_{-2 + \epsilon}] \rangle = 1, \quad \langle x^2, [M_{-2 + \epsilon}] \rangle = 0, \quad \langle y^2, [M_{-2 + \epsilon}] \rangle = -1 & 
						\text{if $M_{-2 + \epsilon} \cong E_{S^2}$} \\		
		\end{cases}
	\]
	where $E_{S^2}$ denotes the non-trivial $S^2$-bundle over $S^2$. (These notations are used in \cite{Li2, Li3} and \cite{Cho1, Cho2}.)
	With these notations, we have 
	\begin{equation}\label{equation_chern_class}
		c_1(T(S^2 \times S^2)) = 2x + 2y, \quad c_1(TE_{S^2}) = 3x+2y.
	\end{equation}
	
	Recall the following relation between the Euler classes of level sets (as principal $S^1$-bundles) near the extremal fixed components 
	$Z_{\min}$ and $Z_{\max}$ of the action and the first Chern numbers of their normal bundles.

	\begin{lemma}\cite[Lemma 6, 7]{Li2}\label{lemma_Euler_extremum}
		Let $b_{\min}$ be the first Chern number of the normal bundle of $Z_{\min}$ in $M$. 
		Then $M_{-2 + \epsilon}$ is a trivial $S^2$-bundle if and only if $b_{\min} = 2k$, and it is diffeomorphic to 
		$E_{S^2}$ if and only if $b_{\min} = 2k+1$ for some $k \in \Z$. In either case, we have 
		\[
			e(P_{-2}^+) = kx - y 
		\]
		where $e(P_t^{\pm})$ denote the Euler class of the principal $S^1$-bundle $\pi_{t \pm \epsilon} :  P_t^{\pm} = H^{-1}(t \pm \epsilon) \rightarrow 
		M_{t \pm \epsilon}$.
		In particular, we have
		\[
			\langle e(P_{-2}^+)^2, [M_{-2 + \epsilon}] \rangle = -b_{\min}.
		\]	
	\end{lemma}

	By Lemma \ref{lemma_Euler_extremum}, we obtain the following.

	\begin{corollary}\label{corollary_volume_extremum}\cite[Corollary 4.2]{Cho2}
		Let $(M,\omega)$ be a six-dimensional closed semifree Hamiltonian $S^1$-manifold. Suppose that $c_1(TM) = [\omega]$. 
		If the minimal fixed component $Z_{\min}$ is diffeomorphic to $S^2$, then 
		$
			b_{\min} \geq -1.
		$
	\end{corollary}

	Observe that the only possible non-extremal critical values are $\{-1, 0\}$
	and each non-extremal fixed component $Z$ is either 
	\[
		\begin{cases}
			\text{$Z$ = pt} \hspace{1cm} \text{if $H(Z) = - 1$, \quad or} \\
			\text{$\dim Z = 2$} \quad \text{if $H(Z) = 0$.}
		\end{cases}
	\]
	by Lemma  \ref{lemma_possible_critical_values}. 
	
	Let $m = |Z_{-1}|$ be the number of isolated fixed points (of index two). Then $M_0 \cong M_1$ is the $k$ points blow-up of $M_{-2 + \epsilon}$
	by Proposition \ref{proposition_GS}. Moreover, since $(M_0, \omega_0)$ is a monotone symplectic four manifold, $M_0$ is diffeomorphic to 
	a del Pezzo surface by \cite{OO2} and so either $\p^2, S^2 \times S^2, X_k$ for $k \leq 8$ where $X_k$ denotes the blow-up of $\p^2$ at $k$ generic points. Thus we obtain
	\begin{equation}\label{equation_number_of_points}
		0 \leq m \leq 7
	\end{equation}
	Denote by 
	\begin{itemize}
		\item $E_1, \cdots, E_m$ the dual classes of the exceptional divisors on $M_0$, 
		\item $\mathrm{PD}(Z_0) = ax + by + \sum_{i=1}^m c_i E_i$ the dual class of $Z_0$ in $M_0$ for some integers $a,b,c_1,\cdots,c_m$,
		\item $e(P_{-2 + \epsilon}) = kx - y$ (for $k \in \Z$) the Euler class of $P_{-2}^+$.
	\end{itemize}
	Then Corollary \ref{corollary_volume_extremum} implies that
	\begin{equation}\label{equation_I_k}
		\begin{cases}
			k \geq -1 & ~\text{if $b_{\min}$ is odd} \\
			k \geq 0 & ~\text{if $b_{\min}$ is even} 
		\end{cases}
	\end{equation}

	We collect equations in $\{a,b,c_1, \cdots, c_m, k\}$ as follows. Since $c_1(TM_0) = [\omega_0]$ by Lemma \ref{lemma_monotonicity},
	we see that
	\begin{equation}\label{equation_I_symplectic_form}
		[\omega_0] = \begin{cases}
			3x + 2y - \sum_{i=1}^m E_i & ~\text{if $b_{\min}$ is odd,} \\ 
			2x + 2y - \sum_{i=1}^m E_i & ~\text{if $b_{\min}$ is even.} \\ 
		\end{cases}
	\end{equation}
	So, we have
	\begin{equation}\label{equation_I_volume}
		\mathrm{Vol}(Z_0) = \langle [\omega_0], [Z_0] \rangle = \begin{cases}
			2a + b + \sum_{i=1}^m c_i \geq 1 & ~\text{if $b_{\min}$ is odd,} \\ 
			2a + 2b + \sum_{i=1}^m c_i \geq 1& ~\text{if $b_{\min}$ is even.} \\ 
		\end{cases}
	\end{equation}
	Moreover. since $e(P_0^+) = kx - y + \sum_{i=1}^m E_i + \mathrm{PD}(Z_0) = (a+k)x + (b-1)y + \sum_{i=1}^m (c_i + 1)E_i$ by Lemma \ref{lemma_Euler_class}, we obtain
	\begin{equation}\label{equation_I_symplectic_form_level_t}
		[\omega_t] = [\omega_0] - e(P_0^+)t = 
		\begin{cases}
			(3 - (a + k)t)x + (2 + (1-b)t)y - \ds \sum_{i=1}^m (1 + (c_i+1) t)E_i & ~\text{if $b_{\min}$ is odd,} \\
			(2 - (a + k)t)x + (2 + (1-b)t)y - \ds \sum_{i=1}^m (1 + (c_i+1) t)E_i & ~\text{if $b_{\min}$ is even.} 
		\end{cases}
	\end{equation}
	for $t \in [0,1]$.
	In particular, 
	\begin{equation}\label{equation_I_symplectic_form_level_1}
		[\omega_1] = [\omega_0] - e(P_0^+) = 
		\begin{cases}
			(3 - a - k)x + (3-b)y - \ds \sum_{i=1}^m (c_i + 2)E_i & ~\text{if $b_{\min}$ is odd,} \\
			(2 - a - k)x + (3-b)y - \ds \sum_{i=1}^m (c_i + 2)E_i & ~\text{if $b_{\min}$ is even.} 
		\end{cases}
	\end{equation}
	
	Now, we classify all possible TFD of $(M,\omega)$. 
	We begin with the following lemma which enables us to compute the Chern number using TFD. 
		
	\begin{lemma}\label{lemma_I_Chern_number}
		Suppose that $\dim Z_{\min} = 2$ and $\dim Z_{\max} = 4$ and let $m = |Z_{-1}|$. Then
		\begin{equation}\label{equation_I_Chern}
			\int_M c_1(TM)^3 = 24 + 4b_{\min} -m + \langle 3c_1(TM_0)^2 - 3c_1(TM_0)e(P_0^+) + e(P_0^+)^2, [M_0] \rangle.
		\end{equation}
	\end{lemma}
	
	\begin{proof}
		Applying the ABBV localization theorem, we have
		\[
			\begin{array}{ccl}\vs{0.3cm}
				\ds \int_M c_1^{S^1}(TM)^3 & = &  \ds  
							\int_{Z_{\min}} \frac{\left(c_1^{S^1}(TM)|_{Z_{\min}}\right)^3}{e_{Z_{\min}}^{S^1}} + m \frac{\overbrace{\lambda^3}^{Z_{-1} ~\text{term}}}
							{-\lambda^3}
							+ \int_{Z_0} \frac{\overbrace{\left(c_1^{S^1}(TM)|_{Z_0}\right)^3}^{= 0}}{e_{Z_0}^{S^1}}
							 + \int_{Z_{\max}} \frac{\left(c_1^{S^1}(TM)|_{Z_{\max}}\right)^3}
							 {e_{Z_{\max}}^{S^1}} \\ \vs{0.2cm}
							 & = & \ds \int_{Z_{\min}} \frac{(2\lambda + (2+b_{\min})u)^3}{\lambda^2 + b_{\min}u\lambda} - m + 
							 \int_{Z_{\max}} \frac{\left(c_1^{S^1}(TM)|_{Z_{\max}}\right)^3} {e_{Z_{\max}}^{S^1}} \\ \vs{0.2cm}
							& = &  (24 + 4b_{\min}) - m + \ds \int_{M_0} \frac{\left(c_1(TM_0) - e - \lambda \right)^3}{-e - \lambda}\quad \quad 
							(c_1(TZ_{\max}) = c_1(TM_0), ~e := e(P_0^+) = e(P_1^-)) \\ \vs{0.2cm}
							& = & 24 + 4b_{\min} - m + \ds \int_{M_0} \frac{\left(c_1(TM_0) - e - 
							\lambda \right)^3(e^2 - e\lambda + \lambda^2)}{-(e+\lambda)(e^2 - e\lambda + \lambda^2)}
							\\ \vs{0.2cm}
							& = & 24 + 4b_{\min} - m + \langle 3c_1(TM_0)^2 - 3c_1(TM_0)e) + e^2, [M_0] \rangle.
			\end{array}			
		\]
	\end{proof}

%
%
%

	We divide into four cases: $\mathrm{Crit} ~H = \underbrace{\{1,-2\}}_{\text{({\bf I-1})}}$ , $\underbrace{\{1,0, -2\}}_{\text{({\bf I-2})}}$, 
	$\underbrace{\{1,-1, -2\}}_{\text{({\bf I-3})}}$, and $\underbrace{\{1,0,-1,-2\}}_{\text{({\bf I-4})}}$.

	\begin{theorem}[Case {\bf I-1}]\label{theorem_I_1}
		Let $(M,\omega)$ be a six-dimensional closed monotone semifree Hamiltonian $S^1$-manifold such that 
		$\mathrm{Crit} H = \{ 1, -2\}$. 
		Then the list of all possible topological fixed point data is given in Table \ref{table_I_1}
		\begin{table}[h]
			\begin{tabular}{|c|c|c|c|c|c|c|}
				\hline
				    & $(M_0, [\omega_0])$ & $e(P_{-2}^+)$ & $Z_{-2}$ & $Z_1$ & $b_2(M)$ & $c_1^3(M)$ \\ \hline \hline
				    {\bf (I-1-1.1)} & $(E_{S^2}, 3x + 2y)$ & $-x-y$  & $S^2$ & 
				    	$E_{S^2}$
					     & $2$ & $54$ \\ \hline    
				    {\bf (I-1-2.1)} & $(S^2 \times S^2, 2x + 2y)$ & $-y$  & $S^2$ & 
				    	$S^2 \times S^2$
					     & $2$ & $54$ \\ \hline    
				    {\bf (I-1-2.2)} & $(S^2 \times S^2, 2x + 2y)$ & $x-y$  & $S^2$ & 
				    	$S^2 \times S^2$
					     & $2$ & $54$ \\ \hline    
			\end{tabular}		
			\vs{0.5cm}			
			\caption{\label{table_I_1} Topological fixed point data for $\mathrm{Crit} H = \{1, -2\}$.}
		\end{table}				   
	\end{theorem}

	\begin{proof}
		 We divide the proof into two cases: $M_0 \cong E_{S^2}$ {\bf (I-1-1)} and $M_0 \cong S^2 \times S^2$ {\bf (I-1-2)}.
		\vs{0.3cm}
		
	\noindent	
	{\bf (I-1-1)}: $M_0 \cong E_{S^2}$. \vs{0.3cm}
	
	\noindent
		In this case, we have $b_{\mathrm{min}} = 2k + 1$ for some $k \in \Z$ by Lemma \ref{lemma_Euler_extremum}.
		Recall that 
		\[
			[\omega_t] = (3 - kt)x + (2 + t)y, \quad t \in [0,1], 
		\]
		see \eqref{equation_I_symplectic_form_level_t}.
		If $k \geq 0$, then 
		\[
			\langle [\omega_t], y \rangle = 1 - (k+1)t
		\]
		so that a blow-down occurs at level $0< t = \frac{1}{k+1} \leq 1$ where the vanishing exceptional divisor corresponds to $y$. 
		Since $M_0 \cong M_t \cong M_1$ by an elementary Morse theory, this cannot be happened, i.e., $k \leq -1$ and hence $k=-1$ 
		by \eqref{equation_I_k}.
		See Table \ref{table_I_1}: {\bf (I-1-1.1)}. The Chern number is computed from Lemma \ref{lemma_I_Chern_number} as follows: 
		\[
			\int_M c_1(TM)^3 = 24 + (-4) + 0 + 3\cdot 8 - 3 (-3) + 1 = 54.
		\]
	\vs{0.3cm}
	
	\noindent
	{\bf (I-1-2)}: $M_0 \cong S^2 \times S^2$. \vs{0.3cm}

	\noindent
		In this case, $b_{\min}$ is even by Lemma \ref{lemma_Euler_extremum} so that $b_{\mathrm{min}} = 2k$ 
		for some $k \in \Z$. Since $[\omega_1] = (2 - k)x + 3y$ by \eqref{equation_I_symplectic_form_level_1}, 
		\[
			\int_{M_1} \omega_1^2 = 6(2-k) > 0. 
		\]
		Also since $k \geq 0$ by \eqref{equation_I_k}, we have $k=0$ or $1$. 
		See {\bf (I-1-2.1)} and {\bf (I-1-2.2)} in Table \ref{table_I_1} for $k=0$ and $k=1$, respectively.
		The Chern numbers for each cases immediately follow from Lemma \ref{lemma_I_Chern_number}.
	\end{proof}

	\begin{example}[Fano varieties of type {\bf (I-1)}]\label{example_I_1} We describe Fano varieties of type {\bf (I-1)} listed in Theorem \ref{theorem_I_1}
	as follows.
		
		\begin{itemize}

			\item {\bf (I-1-1.1)} \cite[No.33 in Section 12.3]{IP} : Let $M$ be the blow-up of $\p^3$ along a $T^3$-invariant line and the circle action generated 
			by $\xi = (0,0,-1) \in \frak{t}^* \cong \R^3$. 
			
				 \begin{figure}[H]
	           	 		\scalebox{0.8}{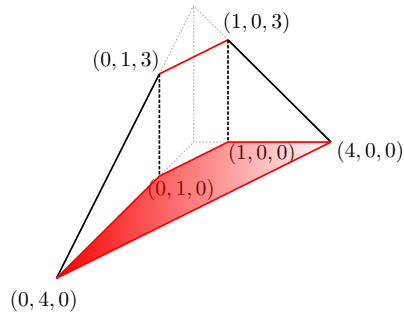}
		           	 	\caption{\label{figure_I_1_2_2} Blow-up of $\p^3$ along a $T^3$-invariant line}
		           	 \end{figure}
			\noindent
			As in Figure \ref{figure_I_1_2_2}, the volume of $Z_{\min} = 1$ (or equivalently $b_{\min} = -1$ by 
			Corollary \ref{corollary_volume_extremum}). Also, the volume of $Z_{\max}$ equals $15 (= 16 - 1)$ and so it coincides with 
			$\langle [\omega_1]^2, [Z_1] \rangle = \rangle (4x + 3y)^2, [Z_1] \rangle = 15$. Thus one can easily check that the topological fixed point data 
			of the action on $M$ coincides with {\bf (I-1-1.1)} in Table \ref{table_I_1}.
			See also \cite[Example 8.8]{Cho1} and \cite[Example 7.2]{Cho2}.\vs{0.3cm}
					           	 
	           	 \item {\bf (I-1-2.1)} \cite[No.34 in Section 12.3]{IP}  : Consider $M = \p^1 \times \p^2$ with the standard Hamiltonian $T^3$-action. Take a circle subgroup 
	           	 of $T^3$ generated by $\xi = (1,1,0)$. Then the fixed point set can be described as in Figure \ref{figure_I_1_1_1} (colored by red). In this case, the volume 
	           	 of $Z_{\min} = 2$ so that $b_{\min} = 0$ (i.e., $k=0$). Thus this example corresponds to {\bf (I-1-2.1)} in Table \ref{table_I_1}.
			See also \cite[Example 7.9]{Cho1}.	           	 
			
	           		 \begin{figure}[H]
	           	 		\scalebox{0.8}{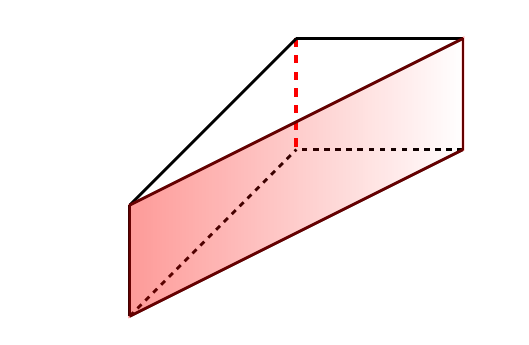}
		           	 	\caption{\label{figure_I_1_1_1} $\p^1 \times \p^2$}
		           	 \end{figure}

			\item {\bf (I-1-2.2)} \cite[No.33 in Section 12.3]{IP} : Let $M$ be the same as {\bf (I-1-1.1)}, that is, the blow-up of $\p^3$ along a $T^3$-invariant line. 
			In this case, we take another circle subgroup of $T^3$ generated by $\xi = (-1-1,0)$. Then the fixed point set can be described as in Figure \ref{figure_I_1_2_1}.
			Also, since the volume of $Z_{\min} = 4$, we see that $b_{\min} = 2$ by Corollary \ref{corollary_volume_extremum}, and hence this is the case where $k=1$. 
			
			See also {\bf (I-1-1.1)}, \cite[Example 8.8]{Cho1} and \cite[Example 7.2]{Cho2}.
	           		 \begin{figure}[H]
	           	 		\scalebox{0.8}{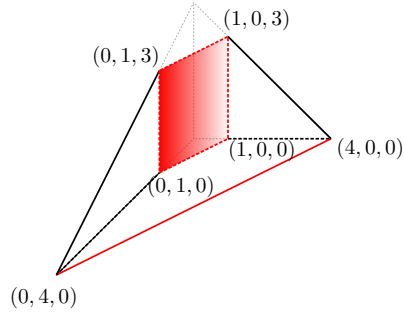}
		           	 	\caption{\label{figure_I_1_2_1} Blow-up of $\p^3$ along a $T^3$-invariant line}
		           	 \end{figure}
			
		\end{itemize}		
	\end{example}

	\begin{theorem}[Case {\bf I-2}]\label{theorem_I_2}
		Let $(M,\omega)$ be a six-dimensional closed monotone semifree Hamiltonian $S^1$-manifold such that $\mathrm{Crit} H = \{ 1, -1, -2\}$. 
		Then there is only one possible topological fixed point data as in Table \ref{table_I_2}
		\begin{table}[h]
			\begin{tabular}{|c|c|c|c|c|c|c|c|}
				\hline
				    & $(M_0, [\omega_0])$ & $e(P_{-2}^+)$ & $Z_{-2}$ & $Z_{-1}$ & $Z_1$ & $b_2(M)$ & $c_1^3(M)$ \\ \hline \hline
				    {\bf (I-2)} & $(E_{S^2}, 3x + 2y - E_1)$ & $-x-y$  & $S^2$ & $\mathrm{pt}$ & 
				    	$E_{S^2} \# ~\overline{\p^2}$
					     & $3$ & $46$ \\ \hline    
			\end{tabular}		
			\vs{0.5cm}			
			\caption{\label{table_I_2} Topological fixed point data for $\mathrm{Crit} H = \{1, -1, -2\}$.}
		\end{table}				   
	\end{theorem}

	\begin{proof}
		Similar to the proof of Theorem \ref{theorem_I_1}, we divide into two cases: $M_0 \cong E_{S^2}$ and $M_0 \cong S^2 \times S^2$. \vs{0.3cm}
		
	\noindent	
	{\bf Case (1)}: $M_0 \cong E_{S^2}$. \vs{0.3cm}
	
	\noindent
	Let $b_{\mathrm{min}} = 2k + 1$ for some $k \in \Z$ (by Lemma \ref{lemma_Euler_extremum}). If $m > 1$, then $x + y - E_1 - E_2$ becomes an exceptional class by Lemma \ref{lemma_list_exceptional}
	and the volume of the corresponding exceptional divisor $C$ on $M_1$ equals 
	\[
		\langle [\omega_1], [C] \rangle = \langle (x + y - E_1 - E_2) \cdot ((3-k)x + 3y - 2E_1 - 2E_2), [M_1] \rangle  = -1 - k \leq 0,
	\]
	see \eqref{equation_I_symplectic_form_level_t}.
	That is, the blow-down occurs at some level $t \leq 1$. (Indeed, the blow-down occurs at level $t = \frac{1}{2+k}$ by \eqref{equation_I_symplectic_form_level_t}.)
	This is impossible since $M_0 \cong M_1$ and so we have $m = 1$. (Note that $m > 0$ by our assumption.) Moreover, the symplectic volume of the exceptional class $y$ on $M_1$ 
	should be positive and is equal to 
	\[
		\langle y \cdot ((3-k)x + 3y - 2E_1 - 2E_2), [M_1] \rangle = -k >0. 
	\]
	Thus we get $k = -1$ by \eqref{equation_I_k}. See Table \ref{table_I_2}: {\bf (I-2)}.\vs{0.3cm}
	
	\noindent
	{\bf Case (2)}: $M_0 \cong S^2 \times S^2$. \vs{0.3cm}
	
	\noindent
	Let $b_{\mathrm{min}} = 2k$ (for some $k \geq 0$ by \eqref{equation_I_k}) by Lemma \ref{lemma_Euler_extremum}.  
	In this case, we observe that 
	\[
		\langle (y - E_1) \cdot [\omega_1], [M_1] \rangle = \langle (y - E_1) \cdot ((2 - k)x + 3y - \sum_{i=1}^m 2 E_i), [M_1] \rangle = -k \leq 0
	\]
	so that the exceptional divisor representing $y - E_1$ vanishes at $M_t$ where $t = \frac{1}{k+1} \leq 1$. So, no such manifold exists.

	\end{proof}

	\begin{example}[Fano variety of type {\bf (I-2)}]\label{example_I_2}\cite[No.26 in Section 12.4]{IP}
	Consider the standard $T^3$-action on $\p^3$ and let $M$ be the blow-up of $\p^3$ along a fixed point $p$ and a $T^3$-invariant sphere not containing $p$. The induced $T^3$-action
	on $M$ admits a moment map whose image is given in Figure \ref{figure_I_2}. Taking a circle subgroup of $T^3$ generated by $\xi = (-1,0,0)$, we obtain a semifree Hamiltonian $S^1$-action 
	on $m$ whose fixed point set is depicted by red faces as in Figure \ref{figure_I_2}. 
	
	The volume of $Z_{\min}$ equals the length of the red edge, 1, which also equals $b_{\min} + 2$  by Corollary \ref{corollary_volume_extremum}. So, one can see that the fixed point data
	coincides with {\bf (I-2)} in Table \ref{table_I_2}. See also \cite[Example 8.13]{Cho1}.
			           	 
		\begin{figure}[H]
	           	\scalebox{0.8}{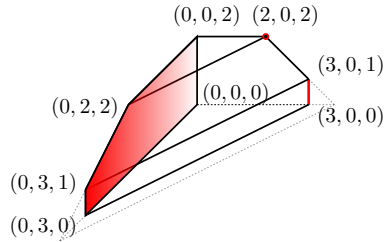}
			\caption{\label{figure_I_2} Blow-up of $\p^3$ along a point and a sphere}
		\end{figure}
			
	\end{example}

	\begin{theorem}[Case {\bf (I-3-1)}]\label{theorem_I_3_1}
		Let $(M,\omega)$ be a six-dimensional closed monotone semifree Hamiltonian $S^1$-manifold such that $\mathrm{Crit} H = \{ 1, 0, -2\}$. 
		If $M_0 \cong E_{S^2}$, then the list of all possible topological fixed point data is given in the Table \ref{table_I_3_1}.
		\begin{table}[h]		
		\begin{adjustbox}{max width=.9\textwidth}		
			\begin{tabular}{|c|c|c|c|c|c|c|c|}
				\hline
				    & $(M_0, [\omega_0])$ & $e(P_{-2}^+)$ & $Z_{-2}$ & $Z_0$ & $Z_1$ & $b_2(M)$ & $c_1^3(M)$ \\ \hline \hline
				    {\bf (I-3-1.1)} & $(E_{S^2}, 3x + 2y)$ & $-y$  & $S^2$ & \makecell{$Z_0 \cong S^2$ \\ $\mathrm{PD}(Z_0) = y$} & $E_{S^2}$ & $3$ & $52$ \\ \hline    
				    {\bf (I-3-1.2)} & $(E_{S^2}, 3x + 2y)$ & $-x-y$  & $S^2$ & \makecell{$Z_0 \cong S^2$ \\ $\mathrm{PD}(Z_0) = y$} 
				    															& $E_{S^2}$ & $3$ & $50$ \\ \hline    
				    {\bf (I-3-1.3)} & $(E_{S^2}, 3x + 2y)$ & $-y$  & $S^2$ &  
				    				\makecell{$Z_0 \cong S^2 ~\dot\cup ~S^2$ \\ $\mathrm{PD}(Z_0^1) = x+y$ \\ $\mathrm{PD}(Z_0^2) = y$} 				    
				    															& $E_{S^2}$ & $4$ & $44$ \\ \hline    
				    {\bf (I-3-1.4)} & $(E_{S^2}, 3x + 2y)$ & $-x-y$  & $S^2$ & \makecell{$Z_0 \cong S^2$ \\ $\mathrm{PD}(Z_0) = x + y$} 
				    								& $E_{S^2}$ & $3$ & $44$ \\ \hline    
				    {\bf (I-3-1.5)} & $(E_{S^2}, 3x + 2y)$ & $-x-y$  & $S^2$ & 
				    							\makecell{$Z_0 \cong S^2 ~\dot\cup ~S^2$ \\ $\mathrm{PD}(Z_0^1) = x+y$ \\ $\mathrm{PD}(Z_0^2) = y$} 
				    								& $E_{S^2}$ & $4$ & $40$ \\ \hline    
				    {\bf (I-3-1.6)} & $(E_{S^2}, 3x + 2y)$ & $-x-y$  & $S^2$ & \makecell{$Z_0 \cong S^2$ \\ $\mathrm{PD}(Z_0) = 2x + 2y$} & $E_{S^2}$ & $3$ & $36$ \\ \hline    
			\end{tabular}		
			\end{adjustbox}
			\vs{0.5cm}			
			\caption{\label{table_I_3_1} Topological fixed point data for $\mathrm{Crit} H = \{1, 0, -2\}$ with $M_0 \cong E_{S^2}$.}
		\end{table}				   
	\end{theorem}

	\begin{proof}
		Let $\mathrm{PD}(Z_0) = ax + by \in H^2(M_0; \Z)$. \vs{0.3cm}
		Observe that, since $[\omega_1] = (3 - a - k)x + (3-b)y$ by \eqref{equation_I_symplectic_form_level_1}, 
		\[
			\underbrace{a+k \leq 2 \text{\hs{0.1cm} and \hs{0.1cm}} b \leq 2}_{\langle [\omega_1]^2,  [M_1] \rangle > 0}, \quad 
			\underbrace{b \geq a + k + 1}_{\langle [\omega_1]\cdot y, [M_1] \rangle \geq 1}, \quad \underbrace{2a + b \geq 1}_{\mathrm{Vol}(Z_0) \geq 1}, \quad 
			\underbrace{k \geq -1}_{\text{by \eqref{equation_I_k}}}.
		\]
		From the inequalities above, we have $a + k \leq 1$, $0 \leq a \leq 2$, and $k \geq -1$. Thus we have 9 solutions: 
		\[
			a = 0 \quad \begin{cases}
				k = 1 & b = 2 \\
				k = 0 & b = 1, 2 \\
				k = -1 & b = 1,2 \\
			\end{cases},
			\quad \quad 
			a = 1 \quad \begin{cases}
				k = 0 & b = 2 \\
				k = -1 & b = 1, 2 \\
			\end{cases},
			\quad \quad 
			a = 2 \quad \begin{cases}
				k = -1 & b = 2.
			\end{cases}						
		\]
		On the other hand, for the case of $a = 0$ and $b=2$, we have 
		\begin{itemize}
			\item $\langle c_1(TM_0), [Z_0] \rangle = \mathrm{Vol}(Z_0) = 2$ and so there are at most two connected components in $Z_0$, and 
			\item $[Z_0] \cdot [Z_0] = -4$ (so that there are at least three sphere components by the adjunction formula)
		\end{itemize}
		which contradict to each other. Therefore, there are 6 possible cases: 
			\begin{table}[H]
				\begin{tabular}{|c|c|c|c|c|c|c|}
					\hline
					    & {\bf (I-3-1.1)} & {\bf (I-3-1.2)} & {\bf (I-3-1.3)} & {\bf (I-3-1.4)} & {\bf (I-3-1.5)} & {\bf (I-3-1.6)} \\ \hline \hline
					    $(a,b,k)$ & $(0,1,0)$ & $(0,1,-1)$  & $(1,2,0)$ & $(1,1,-1)$ & $(1,2,-1)$ & $(2,2,-1)$ \\ \hline
					    $[Z_0]\cdot[Z_0]$ & $-1$ & $-1$ & $0$ & $1$ & $0$ & $4$ \\ \hline
					    \makecell{$\langle c_1(TM_0), [Z_0] \rangle$ \\ $= ~\mathrm{Vol}(Z_0)$}  & $1$ & $1$ & $4$ & $3$ & $4$ & $6$ \\ \hline
				\end{tabular}		
			\end{table}
		\noindent
		Applying the adjunction formula and Lemma \ref{lemma_list_exceptional} to each cases, we obtain Table \ref{table_I_3_1}. 
		We only provide the proof for the case {\bf (I-3-1.3)} where the other five cases can be proved in a similar way and so we left it to the reader.
		
		In case of {\bf (I-3-1.3)}, there are at least two sphere components, namely $Z_0^1$ and $Z_0^2$, such that $\mathrm{Vol}(Z_0^1) + \mathrm{Vol}(Z_0^2) \leq 4$.
		Then,
		\begin{itemize}
			\item  $Z_0^1$ and $Z_0^2$ cannot have volume one simultaneously (otherwise $Z_0^1$ and $Z_0^2$ represent $-1$ spheres so that 
			$\mathrm{PD}(Z_0^1) = \mathrm{PD}(Z_0^2) = y$ by Lemma \ref{lemma_list_exceptional}
				and this violates the fact that $[Z_0^1] \cdot [Z_0^2] = 0$),
			\item  any $Z_0^i$ cannot have volume two. Otherwise, $[Z_0^i] \cdot [Z_0^i]  = 0$ by the adjunction formula and so one would obtain 
				either $\mathrm{PD}(Z_0^i) = c(x + 2y)$ for some integer $c$ (where the volume of $Z_0^i$ becomes $4c$) or 
				$\mathrm{PD}(Z_0^i) = x$ (so that $(\mathrm{PD}(Z_0) - \mathrm{PD}(Z_0^i)) \cdot \mathrm{PD}(Z_0^i) = 2y \cdot x \neq 0$).
		\end{itemize}
		Therefore, the only possibility is that  $\mathrm{Vol}(Z_0^1) = 1$ and  $\mathrm{Vol}(Z_0^2) = 3$
		and hence $\mathrm{PD}(Z_0^1) = y$ and $\mathrm{PD}(Z_0^2) = x + y$. (See Table \ref{table_I_3_1} {\bf (I-3-1.3)}.)

	\end{proof}

	\begin{example}[Fano variety of type {\bf (I-3-1)}]\label{example_I_3_1}  We describe Fano varieties of type {\bf (I-3-1)} in Theorem \ref{theorem_I_3_1} as follows. 
		
		\begin{itemize}
	           	 \item {\bf (I-3-1.1)} \cite[No.31 in Section 12.4]{IP}  : Let $M = \mathbb{P}(\mcal{O} \oplus \mcal{O}(1,1))$ equipped with a $T^3$-action whose moment polytope is given 
	           	 in Figure \ref{figure_I_3_1_1}. If we consider a circle subgroup of $T^3$ generated by $\xi = (1,0,1)$, then the $S^1$-action becomes semifree and has a fixed point set 
	           	 where its image is colored by red. The volume of $Z_{\min}$ equals three as in Figure \ref{figure_I_3_1_1} and so we have $b_{\min} = 1$ (i.e., $k = 0$). Also, $Z_0 \cong S^2$
	           	 has volume $1$, and therefore the fixed point data equals {\bf (I-3-1.1)} in Table \ref{table_I_3_1}.  See also \cite[Example 6.9]{Cho1}. 	           	 	
	           	 
	           		 \begin{figure}[H]
	           	 		\scalebox{0.8}{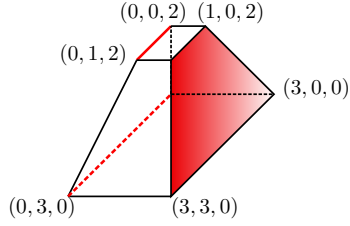}
		           	 	\caption{\label{figure_I_3_1_1} $M = \mathbb{P}(\mcal{O} \oplus \mcal{O}(1,1))$}
		           	 \end{figure}	 	
	           	 	
			\item {\bf (I-3-1.2)} \cite[No.30 in Section 12.4]{IP} : Consider $V_7$, the blow-up of $\p^3$ at a fixed point and let $M$ be the blow-up of $V_7$ along a $T^3$-invariant sphere 
			$C$ (depicted in Figure \ref{figure_I_3_1_2} (a)) passing through the exceptional divisor of $V_7 \rightarrow \p^3$. A moment polytope associated to the induced $T^3$-action on $M$ 
			is given by Figure \ref{figure_I_3_1_2} (b). 
			
			Take a circle subgroup generated by $\xi = (1,1,1)$. The fixed point set for the $S^1$-action corresponds to the faces painted by red. The volume of $Z_{\min}$ and $Z_0$ are both
			one (and hence $b_{\min} = -1$). So, the fixed point data for the $S^1$-action should coincides with {\bf (I-3-1.2)} in Table \ref{table_I_3_1}. See also 
			\cite[Example 8.13]{Cho1} and \cite[Example 8.5]{Cho2}.
	           	 
	           		 \begin{figure}[H]
	           	 		\scalebox{0.8}{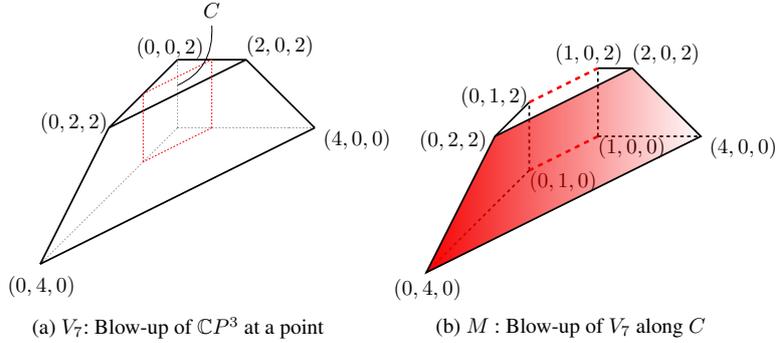}
		           	 	\caption{\label{figure_I_3_1_2} Blow-up of $V_7$ along a $T^2$-invariant sphere passing through the exceptional divisor}
		           	 \end{figure}					
		           
			\item {\bf (I-3-1.3)} \cite[No.11 in Section 12.5]{IP} : Consider the toric variety $\p^1 \times X_1$ and let $M$ be the blow-up of $\p^1 \times X_1$ along a 
			$T^3$-invariant sphere $t \times E$ where $t \in \p^1$ is a fixed point (for the $S^1$-action on $\p^1$) and $E$ is the exceptional curve in $X_1$. Then Figure \ref{figure_I_3_1_3}
			is a moment polytope for the induced $T^3$-action on $M$. 

	           		 \begin{figure}[H]
	           	 		\scalebox{1}{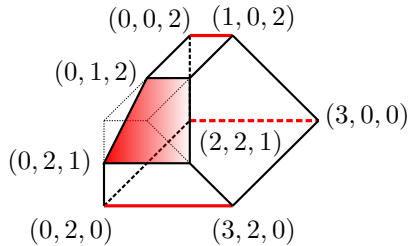}
		           	 	\caption{\label{figure_I_3_1_3} Blow-up of $\p^1 \times X_1$ along $t \times E$}
		           	 \end{figure}					
			
			Now, consider the circle subgroup of $T^3$ generated by $\xi = (0,1,1)$. Then the moment map image of the fixed point set for the $S^1$-action is red faces
			as in Figure \ref{figure_I_3_1_3} (where $Z_0^1$ and $Z_0^2$ correspond to $\overline{(0,0,2)~(1,0,2)}$ and $\overline{(0,2,0)~(3,2,0)}$, respectively).
			Also, we have $\mathrm{Vol}(Z_{\min}) = 3$ (i.e., $b_{\min} =1$ and $k=0$), $\mathrm{Vol}(Z_0^1) = 1$, and $\mathrm{Vol}(Z_0^2) = 3$ so that the fixed point data is exactly the same as 
			{\bf (I-3-1-2)} in Table \ref{table_I_3_1}. See also \cite[Example 7.12]{Cho1}. \vs{0.3cm}

	           	 \item {\bf (I-3-1.4)} \cite[No.25 in Section 12.4]{IP} : Consider $\p^3$ with the standard $T^3$-action and let $Y$ be the blow-up of $\p^3$ with two disjoint $T^3$-invariant spheres.
	           	 Then the induced $T^3$-action has a moment polytope illustrated in Figure \ref{figure_I_3_1_4}. 

	           		 \begin{figure}[H]
	           	 		\scalebox{0.8}{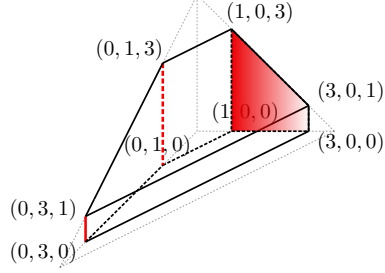}
		           	 	\caption{\label{figure_I_3_1_4} $M$ : blow-up of $\p^3$ along two disjoint curves}
		           	 \end{figure}

	           	 Take a circle subgroup generated by $\xi = (0,-1,0)$. Then the fixed point set for the $S^1$-action corresponds to red faces in Figure \ref{figure_I_3_1_4}. 
	           	 One can easily see that $\mathrm{Vol}(Z_{\min}) = 1$ (so that $b_{\min} = -1$ and $k = -1$) and also $\mathrm{Vol}(Z_0) = 3$. Thus the fixed point data 
	           	 coincides with {\bf (I-3-1.4)} in Table \ref{table_I_3_1}. 
	           	 See also \cite[Example 7.2]{Cho2}.\vs{0.3cm}

	           	 \item {\bf (I-3-1.5)} \cite[No.9 in Section 12.5]{IP} : Let $M$ be the blow-up of $\p^2$ along two disjoint $T^3$-invariant spheres given in {\bf (I-3-1.4)}. For the induced $T^3$-action, 
	           	 let $\widetilde{M}$ be the blow-up of $M$ along a $T^3$-invariant curve (labeled by $A$ in Figure \ref{figure_I_3_1_5})
	           	 lying on an exceptional divisor where the corresponding moment polytope is given in Figure \ref{figure_I_3_1_5}. 
	      
		     		 \begin{figure}[H]
	           	 		\scalebox{0.8}{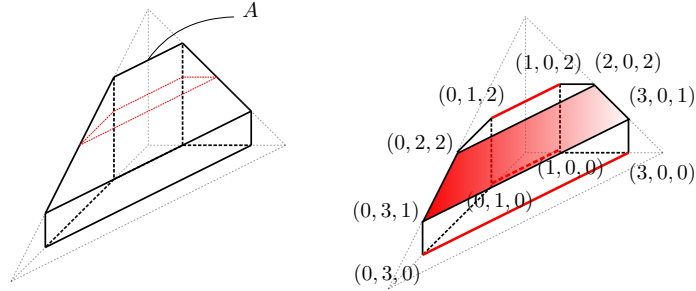}
		           	 	\caption{\label{figure_I_3_1_5} Blow-up of $M$ (in {\bf (I-3-1.4)}) along a $T^3$-invariant exceptional sphere}
		           	 \end{figure}
	      	            \vs{-0.5cm}	 
	           	 Take a circle subgroup of $T^3$ generated by $\xi = (1,1,1)$. Then the minimal fixed component $Z_{\min}$ has volume 1 so that $b_{\min} = -1$ and $k = -1$. Also the volume of 
	           	 $Z_0^1$ and $Z_0^2$ corresponding to $\overline{(0,1,0)~(1,0,0)}$ and $\overline{(3,0,0)~(0,3,0)}$ are respectively $1$ and $3$. Thus the fixed point data for the $S^1$-action
	           	 should be equal to {\bf (I-3-1.5)} in Table \ref{table_I_3_1}.
	           	 See \cite[Example 8.3]{Cho2}
	           	 \vs{0.3cm}
	      
	           	 \item {\bf (I-3-1.6)} \cite[No.18 in Section 12.4]{IP} : Consider the closed monotone semifree Hamiltonian $S^1$-manifold $(M,\omega)$ given in 
	           	 Example \ref{example_I_1}  {\bf (I-1-1.1)}. Then the maximal fixed component $Z_{\max}$ is the one-point blow-up of $\p^2$. 
	           	 
	           	 Let $C$ be a conic in $Z_{\max}$ disjoint from the exceptional divisor and $\widetilde{M}$ be the $S^1$-equivariant blow-up of $M$ along $C$. Then one can check that the induced 
	           	 $S^1$-action is semifree. Since the blow-up does not affect outside the neighborhood of the exceptional divisor, we have $\mathrm{Vol}(\widetilde{Z}_{\min}) = 1$ and so 
	           	 $b_{\min} = -1$ and $k = -1$. Also, since $\widetilde{Z}_0$ is a conic in $M_0 \cong \p^2$, the volume of $\widetilde{Z}_0$ equals six. Thus the fixed point data for the $S^1$-action
	           	 coincides with {\bf (I-3-1.6)} in Table \ref{table_I_3_1}. See figure \ref{figure_I_3_1_6}.
	           	 
	           		 \begin{figure}[H]
	           	 		\scalebox{0.8}{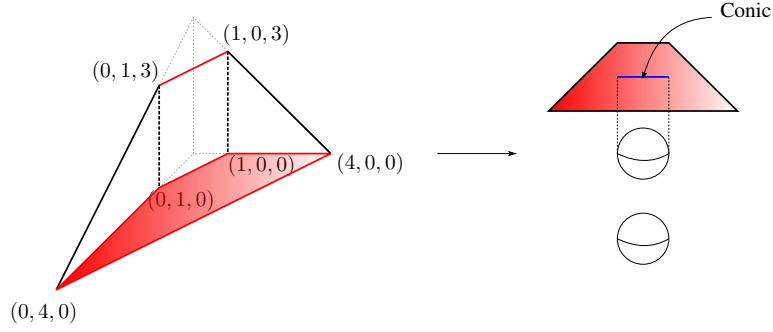}
		           	 	\caption{\label{figure_I_3_1_6} $M$ : blow-up of $\p^3$ along a disjoint union of a line and a conic}
		           	 \end{figure}

		\end{itemize}		
	\end{example}

	\begin{theorem}[Case {\bf (I-3-2)}]\label{theorem_I_3_2}
		Let $(M,\omega)$ be a six-dimensional closed monotone semifree Hamiltonian $S^1$-manifold such that $\mathrm{Crit} H = \{ 1, 0, -2\}$. 
		If $M_0 \cong S^2 \times S^2$, then the list of all possible topological fixed point data is given in the Table \ref{table_I_3_2}.
		\begin{table}[h]
			\begin{tabular}{|c|c|c|c|c|c|c|c|}
				\hline
				    & $(M_0, [\omega_0])$ & $e(P_{-2}^+)$ & $Z_{-2}$ & $Z_0$ & $Z_1$ & $b_2(M)$ & $c_1^3(M)$ \\ \hline \hline
				    {\bf (I-3-2.1)} & $(S^2 \times S^2, 2x + 2y)$ & $-y$  & $S^2$ & \makecell{$Z_0 \cong S^2$ \\ $\mathrm{PD}(Z_0) = y$} & $S^2 \times S^2$ & $3$ & $48$ \\ \hline    
				    {\bf (I-3-2.2)} & $(S^2 \times S^2, 2x + 2y)$ & $x-y$  & $S^2$ & \makecell{$Z_0 \cong S^2$ \\ $\mathrm{PD}(Z_0) = y$} 
				    															&$S^2 \times S^2$ & $3$ & $50$ \\ \hline    
				    {\bf (I-3-2.3)} & $(S^2 \times S^2, 2x + 2y)$ & $-y$  & $S^2$ &  
				    				\makecell{$Z_0 \cong S^2 ~\dot\cup ~S^2$ \\ $\mathrm{PD}(Z_0^1) = y$ \\ $\mathrm{PD}(Z_0^2) = y$} 				    
				    															&$S^2 \times S^2$ & $4$ & $42$ \\ \hline    
				    {\bf (I-3-2.4)} & $(S^2 \times S^2, 2x + 2y)$ & $x-y$  & $S^2$ & 
				    								\makecell{$Z_0 \cong S^2 ~\dot\cup ~S^2$ \\ $\mathrm{PD}(Z_0^1) = y$ \\ $\mathrm{PD}(Z_0^2) = y$} 
				    								& $S^2 \times S^2$ & $4$ & $46$ \\ \hline    
				    {\bf (I-3-2.5)} & $(S^2 \times S^2, 2x + 2y)$ & $-y$  & $S^2$ & 
				    							\makecell{$Z_0 \cong S^2$ \\ $\mathrm{PD}(Z_0) = x$}
				    								& $S^2 \times S^2$ & $3$ & $46$ \\ \hline    
				    {\bf (I-3-2.6)} & $(S^2 \times S^2, 2x + 2y)$ & $-y$  & $S^2$ & 
				    							\makecell{$Z_0 \cong S^2$ \\ $\mathrm{PD}(Z_0) = x + y$}
				    								& $S^2 \times S^2$ & $3$ & $42$ \\ \hline    
				    {\bf (I-3-2.7)} & $(S^2 \times S^2, 2x + 2y)$ & $-y$  & $S^2$ & 
				    							\makecell{$Z_0 \cong S^2$ \\ $\mathrm{PD}(Z_0) = x + 2y$}
				    								& $S^2 \times S^2$ & $3$ & $38$ \\ \hline    				    								
			\end{tabular}		
			\vs{0.5cm}			
			\caption{\label{table_I_3_2} Topological fixed point data for $\mathrm{Crit} H = \{1, 0, -2\}$ with $M_0 \cong S^2 \times S^2$.}
		\end{table}				   
	\end{theorem}

	\begin{proof}
		In this case, we have $b_{\min} = 2k$ for some $k \in \Z$ by Lemma \ref{lemma_Euler_extremum}. Let $\mathrm{PD}(Z_0) = ax + by$ for some $a,b \in \Z$
		Since $[\omega_1] = (2 - a - k)x + (3-b)y$ by \eqref{equation_I_symplectic_form_level_1}, we have 
		\[
			\underbrace{a+k \leq 1 \text{\hs{0.1cm} and \hs{0.1cm}} b \leq 2}_{\langle [\omega_1]^2,  [M_1] \rangle > 0}, \quad 
			\underbrace{2a + 2b \geq 2}_{\mathrm{Vol}(Z_0)}, \quad 
			\underbrace{k \geq 0}_{\text{by \eqref{equation_I_k}}}.
		\]
		From the inequalities above, we have 10 solutions: 
		\[
			a = -1 \quad \begin{cases}
				k = 0 & b = 2 \\
				k = 1 & b = 2 \\
				k = 2 & b = 2 \\
			\end{cases},
			\quad \quad 
			a = 0 \quad \begin{cases}
				k = 0 & b = 1,2 \\
				k = 1 & b = 1,2 \\
			\end{cases},
			\quad \quad 
			a = 1 \quad \begin{cases}
				k = 0 & b = 0,1,2.
			\end{cases}						
		\]		 
		On the other hand, in case of $(a,b) = (-1,2)$ (i.e., $\mathrm{PD}(Z_0) = -x + 2y$), we have 
		\begin{itemize}
			\item $[Z_0] \cdot [Z_0] = -4$,
			\item $\mathrm{Vol}(Z_0) = 2$
		\end{itemize}	
		so that it violates the adjunction formula. Thus we have 7 possibilities:
			\begin{table}[H]
				\begin{tabular}{|c|c|c|c|c|c|c|c|}
					\hline
					    & {\bf (I-3-2.1)} & {\bf (I-3-2.2)} & {\bf (I-3-2.3)} & {\bf (I-3-2.4)} & {\bf (I-3-2.5)} & {\bf (I-3-2.6)} & {\bf (I-3-2.7)}\\ \hline \hline
					    $(a,b,k)$ & $(0,1,0)$ & $(0,1,1)$  & $(0,2,0)$ & $(0,2,1)$ & $(1,0,0)$ & $(1,1,0)$ & (1,2,0) \\ \hline
					    $[Z_0]\cdot[Z_0]$ & $0$ & $0$ & $0$ & $0$ & $0$ & $2$  & $4$\\ \hline
					    \makecell{$\langle c_1(TM_0), [Z_0] \rangle$ \\ $= ~\mathrm{Vol}(Z_0)$}  & $2$ & $2$ & $4$ & $4$ & $2$ & $4$ & $6$\\ \hline
				\end{tabular}		
			\end{table}
		\noindent
		Similar to the proof of Theorem \ref{theorem_I_3_1}, one can check, by the adjunction formula and the fact that the volume of each fixed component of $Z_0$ is even,
		that each $(a,b,k)$ in the table above determines the corresponding topological fixed point data 
		in Table \ref{table_I_3_2} uniquely (where each Chern numbers are obtained from Lemma \ref{lemma_I_Chern_number}).
	\end{proof}
	
	\begin{example}[Fano varieties of type {\bf (I-3-2)}]\label{example_I_3_2}  We describe Fano varieties of type {\bf (I-3-2)} in Theorem \ref{theorem_I_3_2} as follows. 
		
		\begin{itemize}
	           	 \item {\bf (I-3-2.1)} \cite[No.28 in Section 12.4]{IP} : Let $M = \p^1 \times X_1$ equipped with the $T^3$-action induced from the standard toric action on 
	           	 $\p^1 \times \p^2$. The moment polytope is described in Figure \ref{figure_I_3_2_1}. If we take a circle subgroup of $T^3$ generated by $\xi = (1,0,1)$, then
	           	 the $S^1$-action is semifree and the fixed point components corresponds to the faces colored by red in Figure \ref{figure_I_3_2_1}. 
	           	 One can check at a glance from the figure that $\mathrm{Vol}(Z_{\min}) = \mathrm{Vol}(Z_0) = 2$ (and so $b_{\min} = 0$ and $k=0$) so that the fixed point data 
	           	 for the $S^1$-action coincides with {\bf (I-3-2.1)} in Table \ref{table_I_3_2}.
	           	 See also \cite[Example 6.5]{Cho2}.
	           	 
	           		 \begin{figure}[H]
	           	 		\scalebox{1}{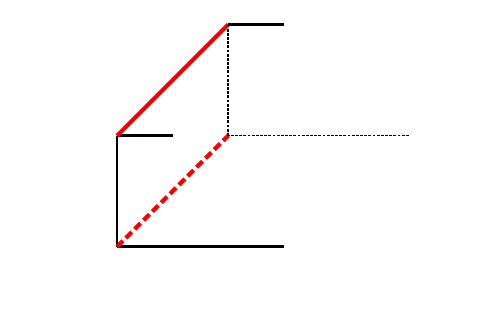}
		           	 	\caption{\label{figure_I_3_2_1} $M = \p^1 \times X_1$}
		           	 \end{figure}
	           	 	
	           	 \item {\bf (I-3-2.2)} \cite[No.30 in Section 12.4]{IP} : Consider $M$ in {\bf (I-3-1.2)}, the blow-up of $V_7$\footnote{Note that $V_7$ denotes the $T^3$-equivariant 
	           	 blow-up of $\p^3$ at a fixed point.} along a $T^3$-invariant sphere passing through the exceptional divisor. In this case, we take another circle subgroup of $T^3$
	           	 generated by $\xi = (-1,-1,0)$. Then we see from Figure \ref{figure_I_3_2_2} that $\mathrm{Vol}(Z_{\min}) = 4$ (and so $b_{\min} = 2$ and $k=1$) and 
	           	 $\mathrm{Vol}(Z_0) = 2$. So, the corresponding fixed point data should equal {\bf (I-3-2.2)}.
	           	 See also \cite[Example 8.13]{Cho1}, \cite[Example 8.5]{Cho2}, and {\bf (I-3-1.2)}.

	           		 \begin{figure}[H]
	           	 		\scalebox{0.8}{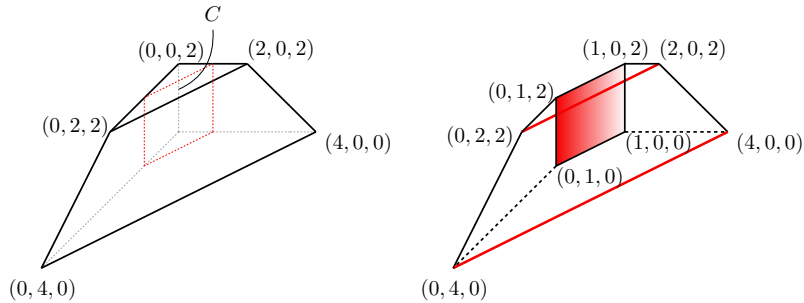}
		           	 	\caption{\label{figure_I_3_2_2} Blow up of $V_7$ along a $T^3$-invariant sphere passing through the exceptional divisor}
		           	 \end{figure}

	           	 \item {\bf (I-3-2.3)} \cite[No.10 in Section 12.5]{IP} : Let $M = \p^1 \times X_2$ with the $T^3$-action induced from the standard $T^3$-action on 
	           	 $\p^1 \times \p^2$. We then take a circle subgroup generated by $\xi = (1,0,1)$ where the moment map image of the fixed point set is depicted by red
	           	 in Figure \ref{figure_I_3_2_3}. One can immediately check that $\mathrm{Vol}(Z_{\min}) = 2$ (so that $b_{\min} = 0$ and $k=0$) and 
	           	 $\mathrm{Vol}(Z_0^1) = \mathrm{Vol}(Z_0^2) = 2$. Thus the fixed point data coincides with {\bf (I-3-2.2)} in Table \ref{table_I_3_2}.
	           	 See also \cite[Example 7.12]{Cho1} and \cite[Example 8.5]{Cho2}.
	           	 
	           		 \begin{figure}[H]
	           	 		\scalebox{1}{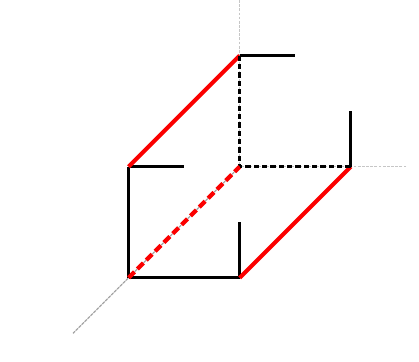}
		           	 	\caption{\label{figure_I_3_2_3} $\p^1 \times X_2$}
		           	 \end{figure}

	           	 \item {\bf (I-3-2.4)} \cite[No.12 in Section 12.5]{IP} : Let $M$ be the $T^3$-equivariant blow-up of $\p^3$ along a $T^3$-invariant line given in {\bf (I-1-1.1)}
	           	 (and {\bf (I-1-2.2)}). Let $\widetilde{M}$ denotes the $T^3$-equivariant blow-up of $M$ along two $T^3$-invariant exceptional lines (that is, lines lying on the exceptional 
	           	 divisor). Then the induced $T^3$-action on $\widetilde{M}$ has a moment map whose image is illustrated in Figure \ref{figure_I_3_2_4}.
	           	 	           	 
	           		 \begin{figure}[H]
	           	 		\scalebox{0.8}{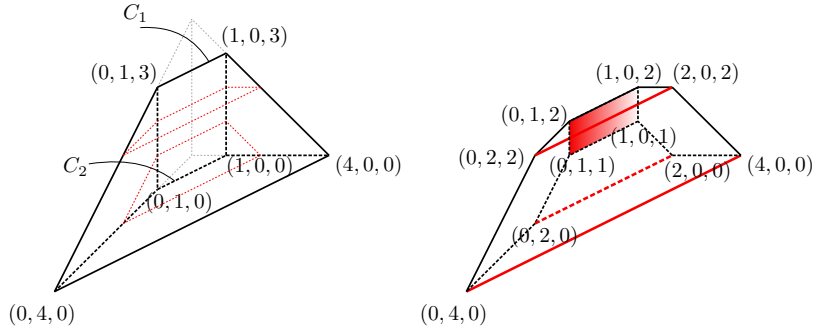}
		           	 	\caption{\label{figure_I_3_2_4} Blow-up of $Y$ along two exceptional lines}
		           	 \end{figure}
		           	 \noindent
	           	 Take a circle subgroup of $T^3$ generated by $\xi = (-1,-1,0)$. Then the action becomes semifree and we can immediately check that 
	           	 $\mathrm{Vol}(Z_{\min}) = 4$ (so that $b_{\min} = 2$ and $k=1$) and $\mathrm{Vol}(Z_0^1) = \mathrm{Vol}(Z_0^2) = 2$. So, the fixed point data of the action
	           	 is equal to {\bf (I-3-2.4)} in Table \ref{table_I_3_2}. See also \cite[Example 8.13]{Cho1} and \cite[Example 8.5]{Cho2}.\vs{0.3cm}

	           	 \item {\bf (I-3-2.5)} \cite[No.26 in Section 12.4]{IP} : Let $M$ be the $T^3$-equivariant blow-up of $\p^3$ along a disjoint union of a fixed point and a $T^3$-invariant line.
	           	 If we take a circle subgroup generated by $\xi = (-1,-1,0)$, then the action is semifree and the moment map image of the fixed point set can be described as in 
	           	 Figure \ref{figure_I_3_2_5}. The volume of the minimal fixed component $Z_{\min}$ is 2 so that $b_{\min} = 2k = 0$ and also $\mathrm{Vol}(Z_0) = 2$. 
	           	 Therefore the fixed point data for the $S^1$-action equals {\bf (I-3-2.5)} in Table \ref{table_I_3_2}.
	           	 See \cite[Example 8.13]{Cho1} and \cite[Example 8.5]{Cho2}
	           	 
	           		 \begin{figure}[H]
	           	 		\scalebox{0.8}{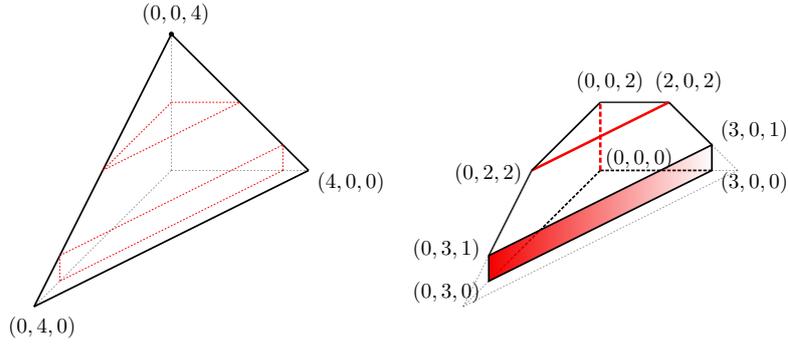}
		           	 	\caption{\label{figure_I_3_2_5} Blow-up of $\p^3$ along a disjoint union of a point and a line}
		           	 \end{figure}

	           	 \item {\bf (I-3-2.6)} \cite[No.24 in Section 12.4]{IP} : Consider the complete flag manifold $\mcal{F}(3)$ with the $T^2$-action induced from the standard $T^3$-action 
	           	 on $\C^3$ where the moment map image, together with the images of $T^2$-invariant spheres, is described on the left of Figure \ref{figure_I_3_2_6}. Let $M$ be the 
	           	 $T^2$-equivariant blow-up of $\mcal{F}(3)$ along a $T^2$-invariant curve $C$ where the corresponding edge is $\overline{(0,0)~(0,2)}$. Then the image of the 
	           	 moment map for the induced $T^2$-action on $M$ is given on the right of Figure \ref{figure_I_3_2_6}. 
	           	 (Note that the red edge $\overline{(1,0)~(1,3)}$ corresponds to the exceptional divisor of the blow-up.) 

	           		 \begin{figure}[H]
	           	 		\scalebox{1}{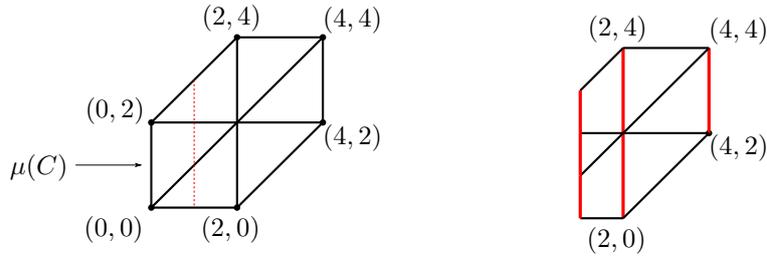}
		           	 	\caption{\label{figure_I_3_2_6} Blow up of $\mcal{F}(3)$ along $C$}
		           	 \end{figure}
	           	 
	           	 Take a circle subgroup of $T^2$ generated by $\xi = (-1,0)$. Then $\mathrm{Vol}(Z_{\min}) = 2$ (so that $b_{\min} = 2k = 0$) and $\mathrm{Vol}(Z_0) = 4$. 
	           	 Therefore the fixed point data for the $S^1$-action should be equal to {\bf (I-3-2.6)} in Table \ref{table_I_3_2}.
	           	 See \cite[Example 8.5]{Cho2}. 
			\vs{0.3cm}	           	 

	           	 \item {\bf (I-3-2.7)} \cite[No.21 in Section 12.4]{IP} : Consider $\p^1 \times \p^2$ with the standard $T^3$-action and consider the $S^1$ subgroup of $T^3$ 
	           	 generated by $\xi = (0,-1,0)$. Then the fixed point set for the $S^1$-action
	           	 can be described as in the first of Figure \ref{figure_I_3_2_7}. The maximal fixed component $Z_{\max}$ is diffeomorphic to  $S^2 \times S^2$ and we take 
	           	 a curve $C$ of bidegree $(1,2)$. 
	
           		 \begin{figure}[H]
	           	 		\scalebox{1}{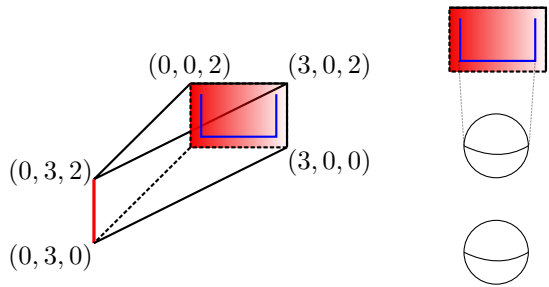}
		           	 	\caption{\label{figure_I_3_2_7} blow-up of $\p^1 \times \p^2$ along a curve of bidegree $(2,1)$}
	           	 \end{figure}
		            \noindent	 
	           	 Let $M$ be the $S^1$-equivariant blow-up of $\p^1 \times \p^2$ along the curve $C$. Since $C$ has volume six in $M_0$ and the blow-up operation does not 
	           	 affect the neighborhood of the minimal fixed component $Z_{\min}$, the volume of $Z_{\min}$ is still 2 and so $b_{\min} = 2k = 0$. This coincides with 
	           	 {\bf (I-3-2.7)} in Table \ref{table_I_3_2}. 
	           	 	           	 	
		\end{itemize}
	\end{example}

	\begin{theorem}[Case {\bf (I-4)}]\label{theorem_I_4}
		Let $(M,\omega)$ be a six-dimensional closed monotone semifree Hamiltonian $S^1$-manifold such that $\mathrm{Crit} H = \{ 1, 0, -1, -2\}$. 
		Then the list of all possible topological fixed point data is given in the Table \ref{table_I_4}
		\begin{table}[h]
			\begin{tabular}{|c|c|c|c|c|c|c|c|c|}
				\hline
				    & $(M_0, [\omega_0])$ & $e(P_{-2}^+)$ & $Z_{-2}$ & $Z_{-1}$ & $Z_0$ & $Z_1$ & $b_2(M)$ & $c_1^3(M)$ \\ \hline \hline
				    {\bf (I-4-1.1)} & \makecell{$(E_{S^2} \# \overline{\p^2},$ \\ $3x + 2y - E_1)$} & $-x-y$  & $S^2$ & $\mathrm{pt}$ & \makecell{ $Z_0 \cong S^2$  \\ $\mathrm{PD}(Z_0) = x + y - E_1$} & 
				    	$X_2$
					     & $4$ & $40$ \\ \hline    
				    {\bf (I-4-1.2)} & \makecell{$(E_{S^2} \# 2\overline{\p^2},$ \\ $3x + 2y - E_1 - E_2)$} & $-x-y$  & $S^2$ & $\mathrm{2 ~pts}$ & \makecell{ $Z_0 \cong S^2$  \\ 
				    $\mathrm{PD}(Z_0) = x + y - E_1 - E_2$} & 
				    	$X_3$
					     & $5$ & $36$ \\ \hline    					     
				    {\bf (I-4.2)} & \makecell{$(S^2 \times S^2 \# \overline{\p^2},$ \\  $2x + 2y - E_1)$} & $-y$  & $S^2$ & $\mathrm{pt}$ & \makecell{ $Z_0 \cong S^2$  \\ $\mathrm{PD}(Z_0) = y - E_1$} & 
				    	$X_2$
					     & $4$ & $44$ \\ \hline    					     					     
			\end{tabular}		
			\vs{0.5cm}			
			\caption{\label{table_I_4} Topological fixed point data for $\mathrm{Crit} H = \{1, 0, -1, -2\}$.}
		\end{table}				   
	\end{theorem}

	\begin{proof}
		 Let $m = |Z_{-1}|$ so that $M_0$ is the $m$-times blow-up of $M_{-2+\epsilon}$ with the exceptional divisors whose dual classes are respectively 
		 denoted by $E_1, \cdots, E_m$. 
		 Let
		 \[
		 	\mathrm{PD}(Z_0) = ax + by + \sum_{i=1}^m c_i E_i, \quad \quad a,b,c_1, \cdots, c_m \in \Z.
		 \]
		 We divide the proof into two cases for $M_{-2 + \epsilon}$: $E_{S^2}$ and $S^2 \times S^2$.\vs{0.3cm}
		 
  		 \noindent
    		 {\bf (I-4-1) :} $M_{-2 + \epsilon} \cong E_{S^2}$ \vs{0.3cm}
		
		\noindent 
		In this case, we have 
		\[
			b_{\min} = 2k + 1, \quad \quad [\omega_1] = (3 - a - k)x + (3 - b)y - \sum_{i=1}^m (c_i + 2) E_i.
		\]
		Then we have the following inequalities:
		\begin{enumerate}
			\item $2a + b + \sum_{i=1}^m c_i \geq 1$ \quad ($\because$ ~ $\mathrm{Vol}(Z_0) \geq 1$)
			\item $a + k \leq 2$ \quad ($\because$ ~ $\int_{M_t} \omega_t^2 > 0$)
			\item $b \leq 2$ \quad ($\because$ ~ $\int_{M_t} \omega_t^2 > 0$)
			\item $c_i \geq -1$ \quad ($\because$ ~ $E_i$ does not vanish on $M_t$)
			\item $b+c_i \leq 0$ \quad ($\because$ ~ $x - E_i$ does not vanish on $M_t$)
			\item $k+a-b \leq -1$ \quad ($\because$ ~ $y$ does not vanish on $M_t$)
			\item $k \geq -1$ \quad (by \eqref{equation_I_k})
		\end{enumerate}
		Combining these inequalities, we have 
		\begin{equation}\label{equation_I_4}
			a \leq b \leq 1
		\end{equation}
		where the first one comes from (6) and (7), and the latter one follows from (4),(5).
		
		We claim that $m \leq 2$. If $m=3$, then (1) and (5) imply that 
		\[
			2a - 2b \geq 2a - 2b + 3b + \sum_{i=1}^3 c_i \geq 1
		\]
		which contradicts that $a - b \leq 0$ in \eqref{equation_I_4}.
		Also if $m \geq 4$, a new exceptional divisor $C$ appears where $\mathrm{PD}(C) = 2x + y - E_{i_1} - E_{i_2} - E_{i_3} - E_{i_4}$
		for any distinct $i_1, i_2, i_3, i_4$ as in Lemma \ref{lemma_list_exceptional}.
		Then 
		\[
			\mathrm{Vol}(C) = \langle (2x + y - E_{i_1} - E_{i_2} - E_{i_3} - E_{i_4}) \cdot [\omega_1], [M_1] \rangle = -2 - a - b - k - \sum_{j=1}^4 c_{i_j} \geq 1
		\]
		which implies that 
		\[
			-3 \geq a + b + k + \sum_{j=1}^4 c_{i_j}. 
		\]		
		Summing those equations for all possible quadruples $(i_1, i_2, i_3, i_4)$, we obtain
		\[
			-3 {m \choose 4} \geq {m \choose 4}a + {m \choose 4}b + {m \choose 4}k + {{m-1} \choose 3} \sum_{i=1}^m c_i		
		\]
		which can be simplified as 
		\begin{equation}\label{equation_imsi}
			\begin{array}{ccl}\vs{0.1cm}
				-3m & \geq & ma + mb + mk + 4\sum_{i=1}^m c_i \\ \vs{0.1cm}
					& = &  4(2a + b + k + \sum_{i=1}^m c_i) + (m-8)a + (m-4)b + (m-4)k \\ \vs{0.1cm}
					& \geq & \ds (m-4)(b-a) + (2m - 12)a - (m-4) \\ \vs{0.1cm}
					& \geq & \ds (2m - 12)a - (m-4).
			\end{array}
		\end{equation}
		by (1) and (7). So, $-2m - 4 \geq (2m - 12)a$ and it is impossible when $4\leq m\leq 6$ since $a \geq 1$ by \eqref{equation_I_4}. 
		Therefore the only possible case is when $m=7$. However in this case, a new divisor appears, namely $D$, where 
		$\mathrm{PD}(D) = 6x + 3y - 2(E_1 + \cdots + E_7)$ by Lemma \ref{lemma_list_exceptional}.
		 Applying the inequality $\langle \mathrm{PD}(D) \cdot [\omega_1], [M_1] \rangle \geq 1$,(1),(7), and \eqref{equation_I_4},
		we get 
		\[
			-11 \geq 3(a+k+b) + 2\sum_{i=1}^7 c_i \geq -a + b + k \geq -1
		\]
		which leads to a contradiction.

		Consequently, we have $m = 1$ or $2$. If $m=1$, then by (1), (5), and \eqref{equation_I_4}, we have $a=1, b=1$ (and hence $k = -1$ by (6) and (7)). Also $c = -1$ by (4) and (5).
		So, 
		\[
			\mathrm{PD}(Z_0) = x + y - E_1, \quad Z_0 \cong S^2 ~\text{by the adjunction formula.}
		\]
		See Table \ref{table_I_4}: {\bf (I-4-1.1)}. 
		
		Finally, when $m=2$, we similarly obtain $a=b=1$ by (1),(5), and \eqref{equation_I_4}, and hence $c_1 = c_2 = -1$ 
		by (4),(5), and $k=-1$ by (6),(7). Thus
		\[
			\mathrm{PD}(Z_0) = x+ y - E_1 - E_2, \quad Z_0 \cong S^2
		\]
		by the adjunction formula again. See Table \ref{table_I_4}: {\bf (I-4-1.2)}.
		\vs{0.3cm}
		
		\noindent
		 {\bf (I-4-2) :} $M_{-2 + \epsilon} \cong S^2 \times S^2$ \vs{0.3cm}
		
		\noindent 
		In this case, we have 
		\[
			b_{\min} = 2k \quad (k \geq 0 ~\text{by \eqref{equation_I_k}}), \quad \quad [\omega_1] = (2 - a - k)x + (3 - b)y - \sum_{i=1}^m (c_i + 2) E_i.
		\]		
		We similarly have the following inequalities :
		\begin{enumerate}
			\item $2a + 2b + \sum_{i=1}^m c_i \geq 1$ \quad ($\mathrm{Vol}(Z_0) \geq 1$.)
			\item $1\geq a+k$ \quad ($\int_{M_t} [\omega_t]^2 > 0$.)
			\item $2 \geq b$ \quad ($\int_{M_t} [\omega_t]^2 > 0$.)
			\item $c_i  \geq -1$ \quad($E_i$ does not vanish on $M_t$)
			\item $b+c_i \leq 0$ \quad ($x - E_i$ does not vanish on $M_t$)
			\item $k+a+c_i \leq -1$ \quad ($y-E_i$ does not vanish on $M_t$)
			\item $k \geq 0$ 
		\end{enumerate}
		Note that by (4),(5),(6),(7), we have 
		\begin{equation}\label{equation_I_4_2}
			a \leq 0, \quad b \leq 1, \quad \underbrace{a+b \leq 0}_{\Leftrightarrow ~(a,b) \neq (0,1)} ~(\text{in case of}~ m \geq 2)
		\end{equation}
		otherwise it contradicts the equation (1),(4),(5).

		We claim that $m=1$. If $m=2$, then (1) and (5) imply that 
		\[
			1 \leq 2a + 2b + c_1 + c_2 \leq 2a 
		\]
		which contradicts that $a \leq 0$ in \eqref{equation_I_4_2}.
		
		Now we assume that $m \geq 3$ (and $m \leq 7$ by \eqref{equation_number_of_points}). 
		Then we have an exceptional divisor $C$ with 
		$\mathrm{PD}(C) = x + y - E_i - E_j - E_k$ for distinct $i,j,k$ by Lemma \ref{lemma_list_exceptional}. 
		The condition $\langle [\omega_1], C \rangle > 0$ implies that 
		\[
			k + a + b + c_i + c_j + c_k \leq -2
		\] for each triple $(i,j,k)$. By summing these inequalities, we obtain
		\[
			{{m} \choose {3}} k+ {{m} \choose {3}} a + {{m} \choose {3}} b + {{m-1} \choose {2}} \sum_{i=1}^m c_i \leq -2 {{m} \choose {3}}
		\] 
		or equivalently, $mk + ma + mb + 3\sum_{i=1}^m c_i \leq -2m$.
		Then, 
		\[
			\begin{array}{ccl}\vs{0.2cm}
				-2m & \geq & 3(2a + 2b + \sum_{i=1}^m c_i) + (m-6)a + (m-6)b +mk  \\ \vs{0.2cm}
					& \geq & 3 + (m-6)(a+b) 
			\end{array}
		\]
		which is possible unless $m=7$ (since $a + b \leq 0$ by \eqref{equation_I_4_2}). When $m=7$, we have an exceptional divisor $D$ with 
		$\mathrm{PD}(D) = 4x + 3y - E_1 - 2(E_2 + \cdots + E_7)$. Then 
		\[
			\langle \mathrm{PD}(D) \cdot [\omega_1], [M_1] \rangle = -8 - 3(a+k) - 4b + c_1 - 2\sum_{i=1}^7 c_i \geq 1.
		\]
		So, by (1), (6), and (7), we have 
		\[
			-9 \geq -c_1 + 3(a+k) + 4b + 2\sum_{i=1}^7 c_i \geq -c_1 - a + 2 \geq k + 1 + 2 \geq 3
		\] which leads to a contradiction.
		
		Therefore, we have $m=1$. By (1),(5), and \eqref{equation_I_4_2}, we obtain
		\[
			1 - c \leq 2a + 2b \leq 2a - 2c \leq -2c, \quad \quad (\Rightarrow ~c \leq -1)
		\]
		which implies that $c = -1$ by (4). Then (1) and \eqref{equation_I_4_2} induces $a = 0, b=1$ (and hence $k=0$ by (6) and (7)). So, 
		\[
			\mathrm{PD}(Z_0) = y - E_1, \quad Z_0 \cong S^2
		\]
		by the adjunction formula. See Table \ref{table_I_4}: {\bf (I-4-2)}. This completes the proof.
		
	\end{proof}
			
	\begin{example}[Fano varieties of type {\bf (I-4)}]\label{example_I_4}  We describe Fano varieties of type {\bf (I-4)} in Theorem \ref{theorem_I_4} as follows. 
		
		\begin{itemize}
	           	 \item {\bf (I-4-1.1)} \cite[No.9 in Section 12.5]{IP}: Let $M$ be the toric variety given in {\bf (I-3-1.4)}; it is the 
	           	 $T^3$-equivariant blow-up of $\p^3$ along two $T^3$-invariant spheres whose moment polytope is given on the left of Figure \ref{figure_I_4_1_1}.
	           	 Let $\widetilde{M}$ be the $T^3$-equivariant blow-up of $M$ along a $T^3$-equivariant exceptional sphere. Then the corresponding moment polytope
	           	 can be described as on the right of Figure \ref{figure_I_4_1_1}. 

	           		 \begin{figure}[H]
	           	 		\scalebox{0.8}{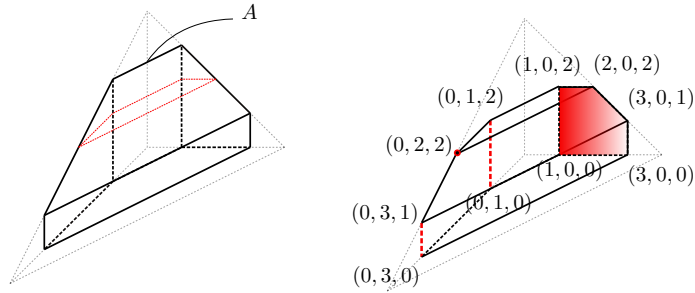}
		           	 	\caption{\label{figure_I_4_1_1} Blow-up of $M$ (in {\bf (I-3-1.4)}) along an exceptional curve}
		           	 \end{figure}
	           	 \noindent
	           	 Take a circle subgroup of $T^3$ generated by $\xi = (0,-1,0)$. The red faces correspond to the fixed components of the action and we have 
	           	 $\mathrm{Vol}(Z_{\min}) = 1$ (and so $b_{\min} = -1$ and $k = -1$). Also, there is a one isolated fixed point of index two and $Z_0$ is a sphere 
	           	 having volume 2. Therefore, the fixed point data should be equal to {\bf (I-4-1.1)} in Table \ref{table_I_4}.
	           	 See \cite[Example 8.3]{Cho2}
			\vs{0.3cm}	           	 
	           	 	
	           	 \item {\bf (I-4-1.2)} \cite[No.2 in Section 12.6]{IP} : Let $M$ be the same as in {\bf (I-3-1.4)}. Now, we let $\widetilde{M}$ be the $T^3$-equivariant 
	           	 blow-up of $M$ along two $T^3$-invariant exceptional spheres lying on the same exceptional component where the moment polytope for the induced $T^3$-action
	           	 is described in Figure \ref{figure_I_4_1_2}.
	           	 
	           		 \begin{figure}[H]
	           	 		\scalebox{0.8}{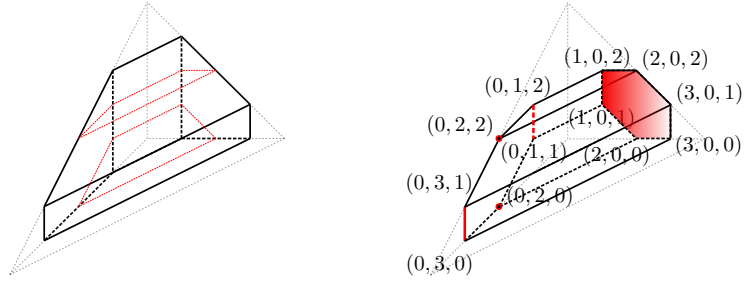}
		           	 	\caption{\label{figure_I_4_1_2} Blow-up of $M$ along two invariant spheres on the same exceptional components}
		           	 \end{figure}
			\noindent
			Take a circle subgroup generated by $\xi = (0,-1,0)$. One can check that the action is semifree and the fixed components corresponds to the red faces 
			in Figure \ref{figure_I_4_1_2}. We can check at a glance that the volume of $Z_{\min}$ and $Z_0$ are both 1 and there are two isolated fixed points 
			of index two. Thus the fixed point data for the $S^1$-action coincides with {\bf (I-4-1.2)} in Table \ref{table_I_4}. 
			See \cite[Example 8.3]{Cho2}. \vs{0.3cm}
	           	 	
	           	 \item {\bf (I-4-2)} \cite[No.11 in Section 12.5]{IP} :  $M$ be given in {\bf (I-3-1.3)} in Example \ref{example_I_3_1}, the $T^3$-equivariant blow-up of $\p^1 \times X_1$
	           	 along $t \times E$ where $t \in \p^1$ is the fixed point for the $S^1$-action and $E$ is the exceptional curve in $X_1$. In this case, we take another circle subgroup
	           	 generated by $\xi = (-1,0,0)$. Then the action is semifree and the fixed point set can be illustrated by the red faces in Figure \ref{figure_I_4_2}. One can immediately
	           	 check that $\mathrm{Vol}(Z_{\min}) = 2$ (so that $b_{\min} = 2k = 0$)
	           	  and $\mathrm{Vol}(Z_0) = 1$. Also there is exactly one isolated fixed point of index two. Thus the fixed point data for the action
	           	 should be equal to {\bf (I-4-2)} in Table \ref{table_I_4}. 
			See also \cite[Example 7.12]{Cho1} and {\bf (I-3-1.3)}. \vs{0.3cm}
	           	 
	           		 \begin{figure}[H]
	           	 		\scalebox{0.8}{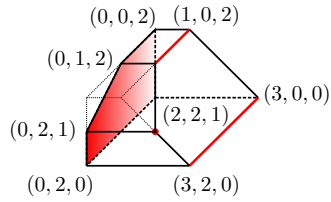}
		           	 	\caption{\label{figure_I_4_2} Blow-up of $\p^1 \times X_1$ along $t \times E$}
		           	 \end{figure}

		\end{itemize}
	
	\end{example}			
	
\section{Topological fixed point data : $\dim Z_{\min} = 4$}
\label{secTopologicalFixedPointDataDimZMax4}

	In this section we classify all topological fixed point data for a six dimensional closed monotone semifree Hamiltonian $S^1$-manifold in case of 
	$\dim Z_{\min} = \dim Z_{\max} = 4$. 
	Since the only possible interior critical value is $0$, there is no isolated fixed point and $Z_{\min} \cong Z_0 \cong Z_{\max}$.
	Also, all possible candidates of $M_0$ are $\C P^2, S^2 \times S^2,$ or $X_k$ (blow-up of $\p^2$ at $k$  generic points) for $1 \leq k \leq 8$
	by Lemma \ref{lemma_monotonicity}. \vs{0.1cm}

	We begin with the following lemma about Chern numbers of Fano varieties of type {\bf (II)}.
		
	\begin{lemma}\label{lemma_II_Chern_number}
		Suppose that $\dim Z_{\min} = 4$ and $\dim Z_{\max} = 4$. Then
		\[
			\int_M c_1(TM)^3 =  \langle 2e^2 + 6c_1^2 - 3c_1[Z_0] + 2e[Z_0] + [Z_0]^2, [M_0] \rangle
		\]
		where $e := e(P_{-1 + \epsilon}^+)$, $c_1 := c_1(TM_0)$, and $[Z_0] := \mathrm{PD}(Z_0)$. 
	\end{lemma}
	
	\begin{proof}
		It is straightforward from the ABBV localization theorem:
		\[
			\begin{array}{ccl}\vs{0.3cm}
				\ds \int_M c_1^{S^1}(TM)^3 & = &  \ds  
							\int_{Z_{\min}} \frac{\left(c_1^{S^1}(TM)|_{Z_{\min}}\right)^3}{e_{Z_{\min}}^{S^1}} 
							+ \int_{Z_0} \frac{\overbrace{\left(c_1^{S^1}(TM)|_{Z_0}\right)^3}^{= 0}}{e_{Z_0}^{S^1}}
							 + \int_{Z_{\max}} \frac{\left(c_1^{S^1}(TM)|_{Z_{\max}}\right)^3}
							 {e_{Z_{\max}}^{S^1}} \\ \vs{0.2cm}
							 & = & \ds \int_{Z_{\min}} \frac{(\lambda + e + c_1)^3}{\lambda + e}  + 
							 \int_{Z_{\max}} \frac{(-\lambda - (e + [Z_0]) + c_1)^3} {-\lambda - (e + [Z_0]) }  \\ \vs{0.2cm}
							& = &  \langle 2e^2 + 6c_1^2 - 3c_1[Z_0] + 2e[Z_0] + [Z_0]^2, [M_0] \rangle
			\end{array}			
		\]
	\end{proof}

	We divide the classification into two cases: $\mathrm{Crit} \mathring{H} = \emptyset$ {\bf (II-1)} and $\mathrm{Crit} \mathring{H} = \{0\}$ {\bf (II-2)}.
	Note that, since the moment map can be reversed by taking $-\omega$ instead of $\omega$, we may assume that
	\begin{equation}\label{equation_II_1_assumption}
		\mathrm{Vol}(Z_{\min}) \geq \mathrm{Vol}(Z_{\max}).
	\end{equation}
	We first consider the case that $\mathrm{Crit} \mathring{H} = \emptyset$. 		
	
	\begin{theorem}[Case {\bf (II-1)}]\label{theorem_II_1}
		Let $(M,\omega)$ be a six-dimensional closed monotone semifree Hamiltonian $S^1$-manifold such that $\mathrm{Crit} H = \{ 1,-1\}$. 
		Then the list of all possible topological fixed point data is given in the Table \ref{table_II_1}
		\begin{table}[h]
			\begin{tabular}{|c|c|c|c|c|c|c|c|c|}
				\hline
				    & $(M_0, [\omega_0])$ & $e(P_{-1}^+)$ & $Z_{-1}$  & $Z_0$ &  $Z_1$ & $b_2(M)$ & $c_1^3(M)$ \\ \hline \hline
				    {\bf (II-1-1.1)} & $(\p^2, 3u)$ & $0$  & $\p^2$  & & 
				    	$\p^2$
					     & $2$ & $54$ \\ \hline    
				    {\bf (II-1-1.2)} & $(\p^2, 3u)$ & $u$  & $\p^2$  && 
				    	$\p^2$
					     & $2$ & $56$ \\ \hline    
				    {\bf (II-1-1.3)} & $(\p^2, 3u)$ & $2u$  & $\p^2$  && 
				    	$\p^2$
					     & $2$ & $62$ \\ \hline    
					     					     					     
				    {\bf (II-1-2.1)} & $(S^2 \times S^2, 2x + 2y)$ & $0$  & $S^2 \times S^2$  && 
				    	$S^2 \times S^2$
					     & $3$ & $48$ \\ \hline    					     
				    {\bf (II-1-2.2)} & $(S^2 \times S^2, 2x + 2y)$ & $x$  & $S^2 \times S^2$ && 
				    	$S^2 \times S^2$
					     & $3$ & $48$ \\ \hline    					     
				    {\bf (II-1-2.3)} & $(S^2 \times S^2, 2x + 2y)$ & $x+y$  & $S^2 \times S^2$ && 
				    	$S^2 \times S^2$
					     & $3$ & $52$ \\ \hline    					     
				    {\bf (II-1-2.4)} & $(S^2 \times S^2, 2x + 2y)$ & $x-y$  & $S^2 \times S^2$ && 
				    	$S^2 \times S^2$
					     & $3$ & $44$ \\ \hline    					     
					     					     					     					     
				    {\bf (II-1-3.1)} & $(E_{S^2}, 3x + 2y)$ & $0$  & $E_{S^2}$ && 
				    	$E_{S^2}$
					     & $3$ & $48$ \\ \hline    					     					     
				    {\bf (II-1-3.2)} & $(E_{S^2}, 3x + 2y)$ & $x + y$  & $E_{S^2}$ && 
				    	$E_{S^2}$
					     & $3$ & $50$ \\ \hline    					     					     
					     					     
				    \makecell{{\bf (II-1-4.k)} \\ {\bf k = 2$\sim$8}} & $(X_k, 3u - \sum_{i=1}^k E_i)$ & $0$  & $X_k$ && 
				    	$X_k$
					     & $k+2$ & $54-6k$ \\ \hline    					     					     					     
			\end{tabular}		
			\vs{0.5cm}			
			\caption{\label{table_II_1} Topological fixed point data for $\mathrm{Crit} H = \{1,-1\}$.}
		\end{table}				   
	\end{theorem}
	
	\begin{proof}
		Let $e := e(P_{-1 + \epsilon}^+) = e(P_{1 - \epsilon}^-)$. \vs{0.3cm}
	
	\noindent
	{\bf (II-1-1)} $M_0 \cong \p^2$. \vs{0.3cm}
	
	Note that $[\omega_t] = 3u - te$ for $t \in [-1, 1]$. Since $\int_{M_0} \omega_t^2 > 0$, we get $e = 0, u,$ or $2u$ (by our assumption \eqref{equation_II_1_assumption}).
	Each cases can be described as
	\begin{table}[H]
		\begin{tabular}{|c|c|c|c|}
			\hline
				& {\bf (II-1-1.1)} & {\bf (II-1-1.2)} & {\bf (II-1-1.3)} \\ \hline \hline
				$e$ & $0$ & $u$  & $2u$  \\ \hline
				$\mathrm{Vol}(Z_{\min})$ & $9$ & $16$ & $25$ \\ \hline
				$\mathrm{Vol}(Z_{\max})$ & $9$ & $4$ & $1$ \\ \hline
		\end{tabular}		
	\end{table}
	\noindent	
	and those are listed in Table \ref{table_II_1}.
	\vs{0.3cm}

	\noindent
	{\bf (II-1-2)} $M_0 \cong S^2 \times S^2$. \vs{0.3cm}
	
	Letting $e = ax + by$ for $a,b \in \Z$, 
	we have $[\omega_t] = (2-at)x + (2-bt)y$ for $t \in [-1,1]$. Again by $\int_{M_t} \omega_t^2 > 0$, we obtain $2 \pm a \geq 1$ and $2 \pm b \geq 1$. 
	So, up to permutation on $\{x, y\}$ and by the condition \eqref{equation_II_1_assumption}, 
	we have $(a,b) = (0,0), (1,0), (1,1),$ and $(1,-1)$. So, we obtain
	\begin{table}[H]
		\begin{tabular}{|c|c|c|c|c|}
			\hline
				& {\bf (II-1-2.1)} & {\bf (II-1-2.2)} & {\bf (II-1-2.3)} & {\bf (II-1-2.4)} \\ \hline \hline
				$e$ & $0$ & $x$  & $x+y$ & $x - y$  \\ \hline
				$\mathrm{Vol}(Z_{\min})$ & $8$ & $12$ & $18$ & $6$ \\ \hline
				$\mathrm{Vol}(Z_{\max})$ & $8$ & $4$ & $2$ & $6$ \\ \hline
		\end{tabular}		
	\end{table}
	\vs{0.3cm}

	\noindent
	{\bf (II-1-3)} $M_0 \cong E_{S^2}$. \vs{0.3cm}
	
	Let $u = x + y$ and $E_1 = y$ so that $\langle u^2, [M_0] \rangle = 1$, $\langle u\cdot E_1, [M_0] \rangle = 0$, and $\langle E_1^2, [M_0] \rangle = -1$.
	Set $e = au + bE_1$ ($a,b \in \Z$). 
	Since $\int_{M_t} \omega_t^2 > 0$ for every $t \in [-1,1]$ where $[\omega_t] = (3-at)u - (1 + bt)E_1$ for $t \in [-1,1]$, we have
	\begin{enumerate}
		\item $3-a > 1+b \geq 1$ \quad ($\langle [\omega_1]^2, [M_1] \rangle >0$ and $\langle [\omega_1] \cdot E_1, [M_1] \rangle > 0$),  
		\item $3+a > 1-b \geq 1$ \quad ($\langle [\omega_{-1}]^2, [M_{-1}] \rangle >0$ and $\langle [\omega_{-1}] \cdot E_1, [M_{-1}] \rangle > 0$) . 
	\end{enumerate}
	From (1) and (2), we obtain $b=0$ and hence $(a,b) = (1,0)$, $(0,0)$ (where the case $(a,b) = (-1,0)$ can be recovered by \eqref{equation_II_1_assumption}). Therefore,
	\begin{table}[H]
		\begin{tabular}{|c|c|c|}
			\hline
				& {\bf (II-1-3.1)} & {\bf (II-1-3.2)} \\ \hline \hline
				$e$ & $0$ & $x + y$   \\ \hline
				$\mathrm{Vol}(Z_{\min})$ & $8$ & $15$  \\ \hline
				$\mathrm{Vol}(Z_{\max})$ & $8$ & $3$ \\ \hline
		\end{tabular}		
	\end{table}
	\vs{0.3cm}

	\noindent
	{\bf (II-1-4)} $M_0 \cong X_k$ (for $2 \leq k \leq 8$). \vs{0.3cm}
	
	In this case, we have $[\omega_0] = 3u - \sum_{i=1}^{k} E_i$. Let $\mathrm{PD}(Z_0) = au + \sum_{i=1}^k b_i E_i$ for some $a, b_1, \cdots, b_k \in \Z$.
	Then, as $\int_{M_t} \omega_t^2 > 0$ for every $t \in [-1,1]$ where $[\omega_t] = (3-at)u - \sum_{i=1}^k (1 + b_it)E_i$ for $t \in [-1,1]$, we get
	\begin{enumerate}
		\item $3-a > 1+b_i \geq 1$, \quad ($\langle [\omega_1]^2, [M_1] \rangle >0$ and $\langle [\omega_1] \cdot E_i, [M_1] \rangle > 0$)  
		\item $3+a > 1-b_i \geq 1$, \quad ($\langle [\omega_{-1}]^2, [M_{-1}] \rangle >0$ and $\langle [\omega_{-1}] \cdot E_i, [M_{-1}] \rangle > 0$)  
	\end{enumerate}
	for every $i=1,\cdots,k$. From (1) and (2), we see that every $b_i$ vanishes and $-1 \leq a \leq 1$. On the other hand, since $u - E_i - E_j$ is an exceptional class for any $i \neq j$, we also have
	\[
		\langle [\omega_1] \cdot (u - E_i - E_j), [M_1] \rangle = 1 - a \geq 1, \quad \langle [\omega_{-1}] \cdot (u - E_i - E_j), [M_{-1}] \rangle = 1 + a \geq 1, 
	\]
	i.e., $a = 0$ (and hence $e = 0$). We denote by {\bf (II-1-4.k)} for each case : $M_0 \cong X_k$ where $2 \leq k \leq 8$. See Table \ref{table_II_1}.

	The Chern number computation for each case can be easily obtained from Lemma \ref{lemma_II_Chern_number}.

	\end{proof}

	\begin{example}[Fano varieties of type {\bf (II-1)}]\label{example_II_1}  We describe Fano varieties of type {\bf (II-1)} in Theorem \ref{theorem_II_1} as follows. \vs{0.3cm}
		
		\begin{itemize}
	           	 \item {\bf (II-1-1)} \cite[No. 34, 35, 36 in Section 12.3]{IP} 
			For {\bf (II-1-1.1)}, let $M = \p^1 \times \p^2$ with the standard $T^3$-action where the moment polytope is described in the first of Figure \ref{figure_II_1_1}. 
			Then the circle subgroup generated by $\xi = (0,0,1)$ acts on the first factor of $\p^1 \times \p^2$ and is semifree. The fixed components for the action 
			correspond to the red faces. It immediately follows from the figure that $\mathrm{Vol}(Z_{\min} = \mathrm{Vol}(Z_{\max} = 16$. Thus the fixed point data 
			coincides with {\bf (II-1-1.1)} in Table \ref{table_II_1}. See also {\bf (I-1-2.1)} and \cite[Example 7.9]{Cho1}. \vs{0.3cm}

			For {\bf (II-1-1.2)}, Let $M = V_7$, the $T^3$-equivariant blow-up of $\p^3$ at a fixed point. The induced $T^3$-action has a moment polytope as in the middle 
			of Figure \ref{figure_II_1_1}. If we take a circle subgroup generated by $\xi = (0,0,1)$, then the action becomes semifree and there are two fixed components
			(colored by red in the figure) such that $\mathrm{Vol}(Z_{\min}) = 16$ and $\mathrm{Vol}(Z_{\max}) = 4$. So, this corresponds to {\bf (II-1-1.2)} in Table \ref{table_II_1}.
			See also \cite[Example 8.5]{Cho1} and \cite[Example 6.3]{Cho2}.
			\vs{0.3cm}

			For {\bf (II-1-1.3)}, we consider the polytope given on the right of Figure \ref{figure_II_1_1}. The polytope corresponds to the toric variety 
			$M = \mathbb{P}(\mcal{O} \oplus \mcal{O}(2))$. Take a circle subgroup generated by $\xi = (0,0,1)$. Then two fixed components appears as in the figure such that 
			$\mathrm{Vol}(Z_{\min}) = 25$ and$\mathrm{Vol}(Z_{\max}) = 1$. Thus the fixed point data for the action shuold equal {\bf (II-1-1.3)} in Table \ref{table_II_1}.
			See also \cite[Example 7.9]{Cho1}.
	           	 
	           		 \begin{figure}[H]
	           	 		\scalebox{0.8}{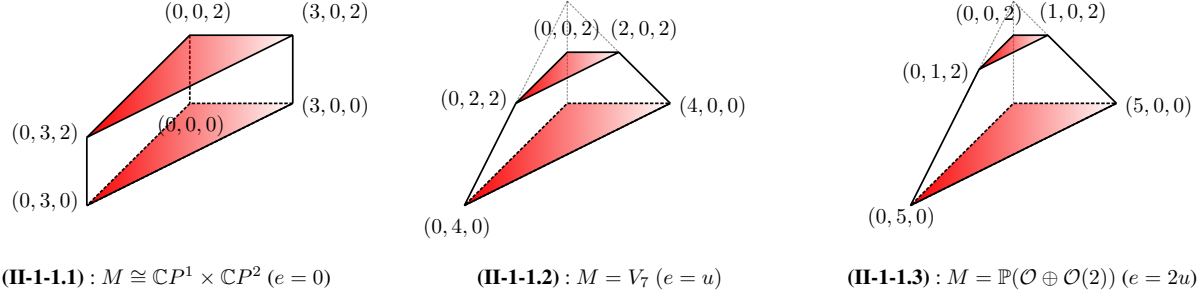}
		           	 	\caption{\label{figure_II_1_1} Fano varieties of type {\bf (II-1-1)}}
		           	 \end{figure}
		           	 
	           	 \item {\bf (II-1-2)} \cite[No. 27,28,31,25 in Section 12.4]{IP}  
	           	 For {\bf (II-1-2.1)}, we consider $M = \p^1 \times \p^1 \times \p^1$ with the standard $T^3$-action whose moment polytope is given in the first of Figure 
	           	 \ref{figure_II_1_2} and a circle subgroup generated by $\xi = (0,0,1)$. Then the action is semifree and there are two copies of $\p^1 \times \p^1$ as 
	           	 the fixed components whose volumes are both $8$. Thus the fixed point data is equal to {\bf (II-1-2.1)} in Table \ref{table_II_1}. 
	           	 See also \cite[Example 6.6]{Cho1} and \cite[Example 6.3]{Cho2}. \vs{0.3cm}

	           		 \begin{figure}[H]
	           	 		\scalebox{0.8}{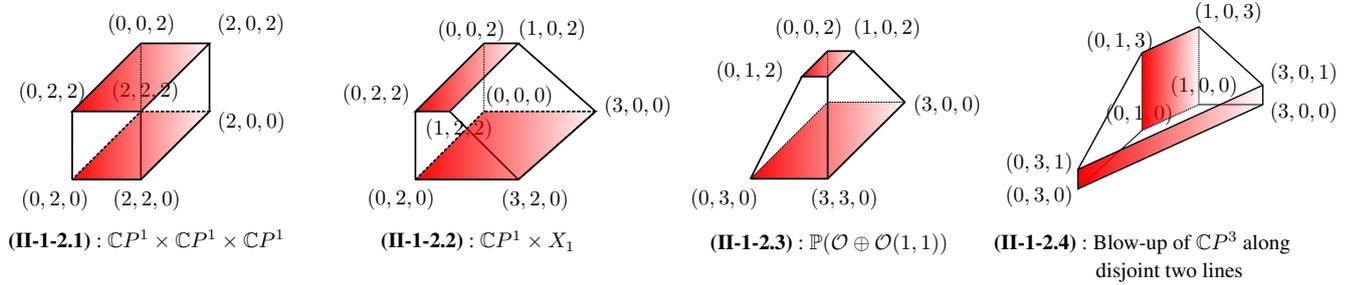}
		           	 	\caption{\label{figure_II_1_2} Fano varieties of type {\bf (II-1-2)}}
		           	 \end{figure}
	           	 
	           	 In case of {\bf (II-1-2.2)}, let $M = \p^1 \times X_1$ equipped with the standard $T^3$-action where the $T^2$-action on $X_1$ is induced from the standard 
	           	 $T^2$-action on $\p^2$. The moment polytope for the action is given in the second of Figure \ref{figure_II_1_2}. Take a circle subgroup of $T^3$ generated 
	           	 by $\xi = (0,0,1)$ so that the action is semifree and there are two fixed  components corresponding to red faces of the figure. One can immediately check that 
	           	 $\mathrm{Vol}(Z_{\min}) = 12$ and $\mathrm{Vol}(Z_{\max}) = 4$. So, the corresponding fixed point data should be {\bf (II-1-2.2)} in Table \ref{table_II_1}. 
	           	 See also \cite[Example 6.5]{Cho2} and {\bf (I-3-2.1)}. \vs{0.3cm}
	           	 
	           	  For {\bf (II-1-2.3)}, consider the polytope given in the third of Figure \ref{figure_II_1_2} where the corresponding toric variety is 
	           	  $M = \mathbb{P}(\mcal{O} \oplus \mcal{O}(1,1))$. Then the circle subgroup of $T^3$ generated by $\xi = (0,0,1)$ acts on $M$ semifreely where the fixed point set 
	           	  consists of two extremal components such that $\mathrm{Vol}(Z_{\min}) = 18$ and $\mathrm{Vol}(Z_{\max}) = 2$. 
	           	  Therefore the fixed point data should be equal to {\bf (II-1-2.3)} in Table \ref{table_II_1}. 
	           	  See also \cite[Example 6.9]{Cho1} and {\bf (I-3-1.1)}. \vs{0.3cm}

			  Finally in case of {\bf (II-1-2.4)}, let $M$ be the $T^3$-equivariant blow-up of $\p^3$ along two $T^3$-invariant lines. Then the corresponding moment polytope is 
			  depicted in the last of Figure \ref{figure_II_1_2}. Take a circle subgroup of $T^3$ generated by $\xi = (1,1,0)$ so that the fixed point set for the $S^1$-action consists
			  of two extremal components. One can check that both of two fixed components have volume 6, and hence the fixed point data is exactly {\bf (II-1-2.4)} in Table 
			  \ref{table_II_1}. See also	 \cite[Example 7.2]{Cho2} and {\bf (I-3-1.4)}.\vs{0.3cm}

	           	 \item {\bf (II-1-3)} \cite[No. 28, 30 in Section 12.4]{IP}  For {\bf (II-1-3.1)}, we let $M = \p^1 \times X_1$ with the standard $T^3$-action as in the case of 
	           	 {\bf (II-1-2.2)}. We take a circle subgroup of $T^3$ generated by $\xi = (0,1,0)$. Then the $S^1$-action is semifree and the fixed components are two copies 
	           	 of $X_1$ where the corresponding faces in the moment polytope are colored by red in Figure \ref{figure_II_1_3}. It immediately follows that the volume of the both
	           	 fixed components are 8, and so the fixed point data equals {\bf (II-1-3.1)} in Table \ref{table_II_1}. 
	           	 See also \cite[Example 6.5]{Cho2}, {\bf (I-3-2.1)}, and {\bf (II-1-2.2)}. \vs{0.3cm}
	           	 
	           	 In case that {\bf (II-1-3.2)}, we consider $V_7$, the $T^3$-equivariant blow-up of $\p^3$ at a fixed point, and let $M$ be the $T^3$-equivariant blow-up of 
	           	 $V_7$ along a $T^3$-invariant line passing through the exceptional divisor of $V_7 \rightarrow \p^3$. Then the moment polytope can be illustrated as on the right 
	           	 of Figure \ref{figure_II_1_3}. Take a circle subgroup of $T^3$ generated by $\xi = (0,0,1)$. One can easily see that $\mathrm{Vol}(Z_{\min}) = 15$ and 
	           	 $\mathrm{Vol}(Z_{\max}) = 3$, and therefore the fixed point data is exactly {\bf (II-1-3.2)} in Table \ref{table_II_1}.
	           	 See also \cite[Example 8.13]{Cho1}, \cite[Example 8.5]{Cho2}, {\bf (I-3-1.2)}, and {\bf (I-3-2.2)}. 
	           	 
	           		 \begin{figure}[H]
	           	 		\scalebox{0.8}{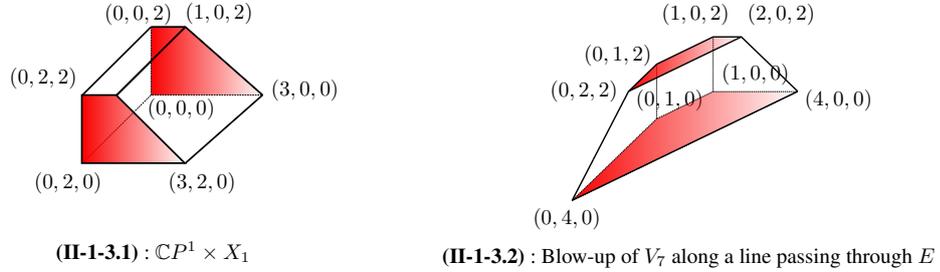}
		           	 	\caption{\label{figure_II_1_3} Fano varieties of type {\bf (II-1-3)}}
		           	 \end{figure}
	           	 
	           	 \item {\bf (II-1-4)} \cite[No. 10 in Section 12.5, No. 3,4,5,6,7,8 in Section 12.6]{IP} : For each $2 \leq k \leq 8$, let $M = X_k \times \p^1$ 
	           	 where the $S^1$-action is given on the second factor. Then the Euler class of each level set of a moment map is trivial. 
	           	 In particular two fixed components $X_k \times t_1$ and $X_k \times t_2$ have the volume $9-k$. Therefore, the fixed point data coincides with
	           	 {\bf (II-1-4.k)} in Table \ref{table_II_1}.
		\end{itemize}
	\end{example}

	Before to proceed the classification for remaining cases, we prove that $M_0$ cannot be $X_k$ for $k > 1$ whenever $Z_0$ is non-empty as follows.

	\begin{proposition}\label{proposition_II_2}
		Let $(M,\omega)$ be a six-dimensional closed monotone semifree Hamiltonian $S^1$-manifold such that $\mathrm{Crit} H = \{ 1,0,-1\}$. Then 
		$M_0 \not \cong X_k$ for any $k \geq 2$. 
	\end{proposition}
	
	\begin{proof}
		Suppose that $M_0 \cong X_k$ for some $k \geq 2$ with $[\omega_0] = 3u - \sum_{i=1}^{k} E_i$ where $u \in H^2(X_k; \Z)$ such that 
		$\langle u^2, [X_k] \rangle = 1$ and $u \cdot E_i  = 0$ for every $i=1,\cdots,k$.
		
		Denote by $\mathrm{PD}(Z_0) = xu + \sum_{i=1}^k y_iE_i$ and $e = au + \sum_{i=1}^k b_iE_i$ for some $a, b_1, \cdots, b_k, x, y_1, \cdots, y_k \in \Z$.
		By the Duistermaat-Heckman theorem \eqref{equation_DH}, we get
		\[
			[\omega_{-1}] = (3+a)u - \sum (1 - b_i)E_i \quad \text{and} \quad [\omega_{1}] = (3-a-x)u - \sum (1 + b_i + y_i)E_i.
		\]			
		Thus we have the following inequalities.
		\begin{enumerate}
			\item $3x + \sum_{i=1}^k y_i \geq 1$ \quad ($\mathrm{Vol}(Z_0) \geq 1$):
			\item $b_i \leq 0$ \quad ($\langle [\omega_{-1}] \cdot E_i, [M_{-1}] \rangle \geq 1$),
			\item $a + b_i + b_j \geq 0$ \quad ($\langle [\omega_{-1}] \cdot (u - E_i - E_j), [M_{-1}] \rangle \geq 1$),
			\item $b_i + y_i \geq 0$ \quad ($\langle [\omega_{1}] \cdot E_i, [M_{1}] \rangle \geq 1$),
			\item $a+x+b_i+b_j+y_i+y_j \leq 0$ \quad ($\langle [\omega_1] \cdot (u-E_i-E_j), [M_1] \rangle \geq 1$).
		\end{enumerate}
		Note that $\underbrace{a \geq 0}_{=: (6)}$ from (2) and (3), and $\underbrace{a + x \leq 0}_{=: (7)}$ by (4) and (5).
		By summing all equations of the form (5), we obtain
		\[
			{k \choose 2} (a + x) + (k-1) \left( \sum_{i=1}^k b_i + \sum_{i=1}^k y_i \right) \leq 0 \quad \Rightarrow \quad \underbrace{\frac{k}{2}a + \frac{k}{2}x + 
			\sum_{i=1}^k b_i + \sum_{i=1}^k y_i \leq 0.}_{=: (5')}. 
		\]	
		Similarly, from (3), we have
		\[
			\underbrace{\frac{k}{2}a + \sum_{i=1}^k b_i \geq 0.}_{=: (3')}.
		\]
		So, by (3') and (5'), we get $\underbrace{\ds \frac{k}{2}x + \sum_{i=1}^k y_i \leq 0}_{=: (8)}$. 
		Since $x \leq 0$ by (6) and (7), we have $k > 6$ (otherwise (1) and (8) contradict to each other).
		
		For $k \geq 7$, there are exceptional divisors of the form $3u - 2E_1 - E_{234567}$ by Lemma \ref{lemma_list_exceptional} where 
		$E_{234567} := E_{2} + \cdots + E_{7}$. Then 
		\begin{itemize}
			\item $\langle [\omega_{-1}] \cdot (3u - 2E_1 - E_{234567}), [M_{-1}] \rangle \geq 1$ \quad $\Rightarrow$ 
			\quad $3a + 2b_1 + b_{234567} \geq 0$.
			\item $\langle [\omega_1] \cdot (3u - 2E_1 - E_{234567}), [M_1] \rangle \geq 1$ \quad $\Rightarrow$ 
			\quad $3a + 3x + 2b_1 + 2y_1 + b_{234567} + y_{234567} \leq 0$
		\end{itemize}
		where $b_{234567} =  b_{2} + \cdots + b_{7}$ and $y_{234567} = y_{2} + \cdots + y_{7}$, respectively.
		Then we obtain 
		\begin{equation}\label{equation_k_7}
			3x  + 2y_1 + y_{234567} \leq 0 
		\end{equation}
		Since the inequality \eqref{equation_k_7} holds up to permutation on $\{1,\cdots,8\}$, we may assume that $y_1$ is  the maximal among all $y_i$'s. Then, 
		\[
			1 \leq 3x  + (y_1 + y_8) + y_{234567} \leq 3x  + 2y_1 + y_{234567} \leq 0 
		\]
		by (1) and this leads  to a contradiction. Therefore, no such manifold exists. 
	\end{proof}
				
	By Proposition \ref{proposition_II_2}, we only to consider the three cases $M_0 \cong \p^2$ {\bf (II-2-1)}, $S^2 \times S^2$ {\bf (II-2-2)}, and $X_1$ {\bf (II-2-3)}. 			

	\begin{theorem}[Case {\bf (II-2-1)}]\label{theorem_II_2_1}
		Let $(M,\omega)$ be a six-dimensional closed monotone semifree Hamiltonian $S^1$-manifold such that $\mathrm{Crit} H = \{ 1,0,-1\}$. 
		Suppose that $M_0 \cong \p^2$. Then the list of all possible topological fixed point data is given in the Table \ref{table_II_2_1}
		\begin{table}[H]
			\begin{tabular}{|c|c|c|c|c|c|c|c|c|}
				\hline
				    & $(M_0, [\omega_0])$ & $e(P_{-1}^+)$ & $Z_{-1}$  & $Z_0$ &  $Z_1$ & $b_2(M)$ & $c_1^3(M)$ \\ \hline \hline
				    {\bf (II-2-1.1)} & $(\p^2, 3u)$ & $0$  & $\p^2$  & \makecell{$Z_0 \cong S^2$, \\ $\mathrm{PD}(Z_0) = u$} & 
				    	$\p^2$
					     & $3$ & $46$ \\ \hline    
				    {\bf (II-2-1.2)} & $(\p^2, 3u)$ & $u$  & $\p^2$  & \makecell{$Z_0 \cong S^2$, \\ $\mathrm{PD}(Z_0) = u$} & 
				    	$\p^2$
					     & $3$ & $50$ \\ \hline    
				    {\bf (II-2-1.3)} & $(\p^2, 3u)$ & $-u$  & $\p^2$  & \makecell{$Z_0 \cong S^2$, \\ $\mathrm{PD}(Z_0) = 2u$} & 
				    	$\p^2$
					     & $3$ & $38$ \\ \hline    
					     					     					     
				    {\bf (II-2-1.4)} & $(\p^2, 3u)$ & $0$  & $\p^2$  & \makecell{$Z_0 \cong S^2$, \\ $\mathrm{PD}(Z_0) = 2u$} & 
				    	$\p^2$
					     & $3$ & $40$ \\ \hline    					     
				    {\bf (II-2-1.5)} & $(\p^2, 3u)$ & $-u$  & $\p^2$ & \makecell{$Z_0 \cong T^2$, \\ $\mathrm{PD}(Z_0) = 3u$} & 
				    	$\p^2$
					     & $3$ & $32$ \\ \hline    					     
				    {\bf (II-2-1.6)} & $(\p^2, 3u)$ & $-2u$  & $\p^2$ & \makecell{$Z_0 \cong \Sigma_3$, \\ $\mathrm{PD}(Z_0) = 4u$} & 
				    	$\p^2$
					     & $3$ & $26$ \\ \hline    					     
			\end{tabular}		
			\vs{0.5cm}			
			\caption{\label{table_II_2_1} Topological fixed point data for $\mathrm{Crit} H = \{1,0,-1\}$ with $M_0 = \p^2$}
		\end{table}				   
	\end{theorem}

	\begin{proof}
		Let $e = au$ and $\mathrm{PD}(Z_0) = pu$ for $a, p \in \Z$. Then the Duistermaat-Heckman theorem yields
		\[
			[\omega_{-1}] = (3+a)u, \quad [\omega_1] = (3 - a - p)u \quad \text{so that} \quad 
			\begin{cases}
				3 + a \geq 1 \\
				3 - a - p \geq 1.
			\end{cases}
		\]
		Note that $p$ is positive by $\langle [\omega_0], [Z_0] \rangle  > 0$ and $3 + a \geq 3 - a - p$ by the assumption \eqref{equation_II_1_assumption}.
		Therefore all possible solutions of $(p,a)$
		are listed as follows.
		\begin{table}[H]
			\begin{tabular}{|c|c|c|c|c|c|c|}
				\hline
					& {\bf (II-2-1.1)} & {\bf (II-2-1.2)} & {\bf (II-2-1.3)} & {\bf (II-2-1.4)} & {\bf (II-2-1.5)} & {\bf (II-2-1.6)} \\ \hline \hline
					$(p,a)$ & $(1,0)$ & $(1,1)$  & $(2,-1)$ & $(2,0)$ & $(3,-1)$ & $(4,-2)$ \\ \hline
					$\mathrm{Vol}(Z_{\min})$ & $9$ & $16$ & $4$ & $9$ & $4$ & $1$\\ \hline
					$\mathrm{Vol}(Z_{\max})$ & $4$ & $1$ & $4$ & $1$ & $1$& $1$\\ \hline
					\makecell{$\langle c_1(TM_0), [Z_0] \rangle$ \\ $= \mathrm{Vol}(Z_0)$} & $3$ & $3$ & $6$ & $6$ & $9$& $12$\\ \hline
					$[Z_0] \cdot [Z_0]$ & $1$ & $1$ & $4$ & $4$ & $9$& $16$\\ \hline
					$g$ : genus of $Z_0$ & $0$ & $0$ & $0$ & $0$ & $1$& $3$\\ \hline					
			\end{tabular}		
		\end{table}
		\noindent
		Thus we obtain Table \ref{table_II_2_1} using the adjunction formula and Lemma \ref{lemma_II_Chern_number}.
		\vs{0.3cm}		
	\end{proof}

	\begin{example}[Fano varieties of type {\bf (II-2-1)}]\label{example_II_2_1}
	  We describe Fano varieties of type {\bf (II-2-1)} in Theorem \ref{theorem_II_2_1} as follows. \vs{0.3cm}
	  
	  \begin{itemize}
	  	\item {\bf (II-2-1.1), (II-2-1.4)} \cite[No. 26, 22 in Section 12.4]{IP} : 
	  	Let $M = \p^1 \times \p^2$ with the standard $T^3$-action and we consider the circle subgroup generated by $\xi = (0,0,-1)$.
	  	Then the maximal and minimal fixed component are both $\p^2$. Take (any) line $C_1$ and a conic $C_2$ in $Z_{\max}$ and denote the $S^1$-equivariant 
	  	blow-up of $M$ along $C_i$ by $M_i$ for $i=1,2$. 
	  	
	  		 \begin{figure}[H]
	           	 		\scalebox{0.8}{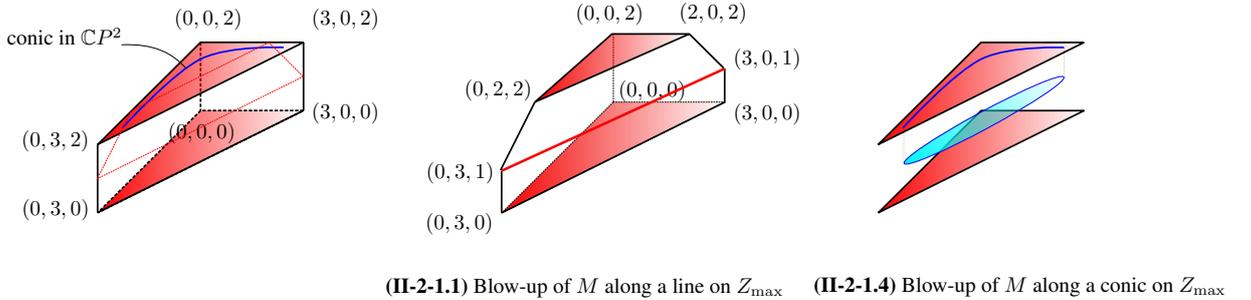}
		           	 	\caption{\label{figure_II_2_1_1} Fano varieties of type {\bf (II-2-1.1)} and {\bf (II-2-1.4)}}
		           \end{figure}
	  	\noindent
	  	Then the induced $S^1$-action on each $M_i$ is semifree and the fixed point data is described in Figure \ref{figure_II_2_1_1}. (The first one is a moment polytope of 
	  	$\p^1 \times \p^2$, the second one is for $M_1$ where $C_1$ is taken to be $T^3$-equivariant, and  the third one is a {\em conceptual images} of the fixed point set
	  	where the {\em blue ellipse} indicates the fixed component on the zero level set in $M_2$. If we denote by $Z_{\min}^i$ and $Z_0^i$ 
	  	the minimal fixed component and the interior fixed component of $M_i$ respectively, then one can 
	  	immediately check that $\mathrm{Vol}(Z_{\min}^1) = \mathrm{Vol}(Z_{\min}^2) = 9$ and $\mathrm{Vol}(Z_0^1) = 3$ and $\mathrm{Vol}(Z_0^2) = 6$ (since $Z_0^2$ is a 
	  	conic in $M_0$).
	  	Thus the corresponding fixed point data coincide with {\bf (II-2-1.1)} and {\bf (II-2-1.4)} in Table \ref{table_II_2_1}, respectively. 
	  	
	  	See also Example \ref{example_I_2} {\bf (I-2)}, Example \ref{example_I_3_2} {\bf (I-3-2.5)}, \cite[Example 8.13]{Cho1}, and \cite[Example 8.5]{Cho2} for {\bf (II-2-1.1)}.
	  	\vs{0.3cm}
	  		  	
	  	\item {\bf (II-2-1.2), (II-2-1.3), (II-2-1.5)} \cite[No. 29, 19, 14 in Section 12.4]{IP} : 
	  	Consider $V_7$, the $T^3$-equivariant blow-up of $\p^3$ at a fixed point, and the $S^1$-action on $V_7$
	  	generated by $\xi = (0,0,1)$ where the fixed components for the action are colored by red in the first of Figure \ref{figure_II_2_1_2}
	  	where $\mathrm{Vol}(Z_{\min}) = 4$ and $\mathrm{Vol}(Z_{\min}) = 16$.
	  	
	  		 \begin{figure}[H]
	           	 		\scalebox{0.8}{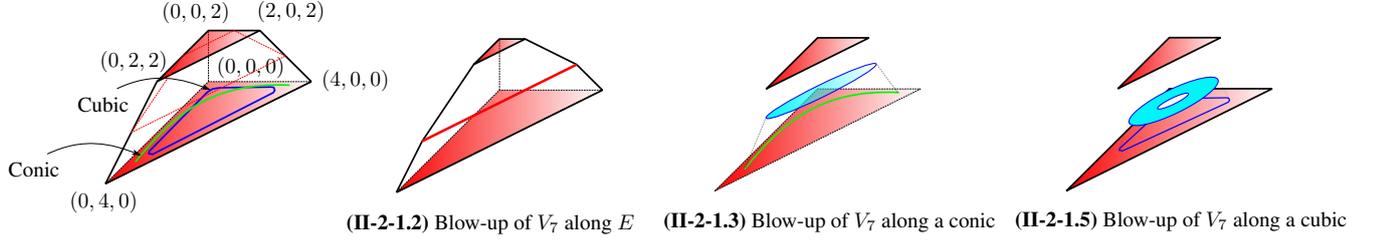}
		           	 	\caption{\label{figure_II_2_1_2} Fano varieties of type {\bf (II-2-1.2)}, {\bf (II-2-1.3)} and {\bf (II-2-1.5)}}
		           \end{figure}
	  	\noindent
	  	There are three ways of obtaining a semifree $S^1$-Fano variety by blowing-up $V_7$ as follows. 
		Let $C_1$ be degree-one curve lying on the exceptional divisor $Z_{\max} \cong \p^2$. Also, we let $C_2$ be a conic, and $C_3$ be a cubic in $Z_{\min}$. 
		Denote by $M_i$ the $S^1$-equivariant blow-up of $V_7$ along $C_i$ for each $i=1,2,3$. If we denote by $Z_{\min}^i$ and $Z_{\max}^i$ the minimal and maximal
		fixed components on $M_i$, respectively, then we can easily see that 
		\begin{itemize}
			\item $\mathrm{Vol}(Z_{\min}^1) = 16$, $\mathrm{Vol}(Z_{\max}^1) = 1$,
			\item $\mathrm{Vol}(Z_{\min}^2) = 9$, $\mathrm{Vol}(Z_{\max}^2) = 4$,
			\item $\mathrm{Vol}(Z_{\min}^3) = 4$, $\mathrm{Vol}(Z_{\max}^3) = 4$.
		\end{itemize}
		Therefore, the fixed point data for each case coincides with one described in Table \ref{table_II_2_1}. See also \cite[Example 8.13]{Cho1} and
		\cite[Example 8.5]{Cho2} for {\bf (II-2-1.2)}.
		
		\begin{remark}
			If we consider a line $C$ in $Z_{\min}$ in $V_7$ and take the $S^1$-equivariant blow-up of $V_7$ along $C$, then the resulting manifold $\widetilde{M}$ is also Fano 
			and $\mathrm{Vol}(\widetilde{Z}_{\min}) = 9$, $\mathrm{Vol}(\widetilde{Z}_{\max}) = 4$. This model can be taken as a toric model and so one can check that 
			$\widetilde{M}$ is isomorphic to the blow-up of $\p^1 \times \p^2$ along a line, see {\bf (II-2-1.1)}. 
			
			We also note that, in Mori-Mukai's classification, the manifold $M_2$ is described as the $T^2$-equivariant blow-up of a smooth quadric $Q$ in $\p^4$ at two isolated
			extrema. See in Figure \ref{figure_II_2_1_2_1}.
	  	
  		 \begin{figure}[H]
	           	 \scalebox{0.8}{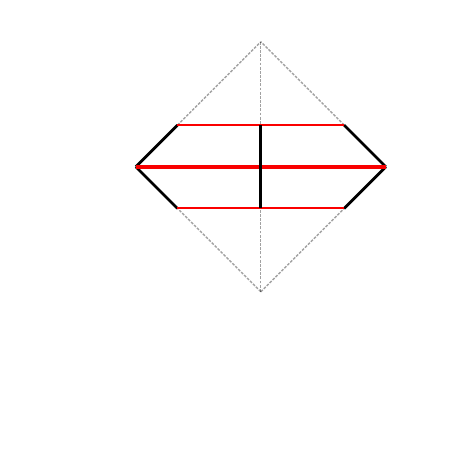}
		           \caption{\label{figure_II_2_1_2_1} Another description of {\bf (II-2-1.3)}}
	           \end{figure}
		\end{remark}

		\vs{0.3cm}
	  	
	  	\item {\bf (II-2-1.6)} \cite[No. 9 in Section 12.4]{IP} : Let $M = \mathbb{P}(\mcal{O} \oplus \mcal{O}(1,1))$ with the semifree $S^1$-action 
	  	as we have seen in Example \ref{example_II_1} {\bf (II-1-1.3)}.
	  	It is a toric variety whose moment polytope is given in the first of Figure \ref{figure_II_2_1_3}. Take $C$ a quartic in $Z_{\min} \cong \p^2$ and denote the $S^1$-equivariant 
	  	blow-up of $M$ along $C$ by $\widetilde{M}$. Then we may easily check that $\mathrm{Vol}(\widetilde{Z}_{\min}) = \mathrm{Vol}(\widetilde{Z}_{\max}) = 1$ and this 
	  	coincides with {\bf (II-2-1.6)} in Table \ref{table_II_2_1}.

	           		 \begin{figure}[H]
	           	 		\scalebox{0.8}{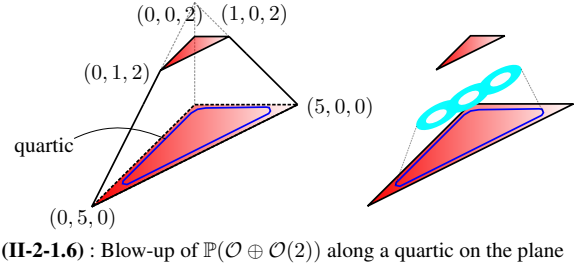}
		           	 	\caption{\label{figure_II_2_1_3} Fano varieties of type {\bf (II-2-1.6)}}
		           	 \end{figure}
		           	 
		\begin{remark}
			It is also possible that we take any curve $C$ of degree less than four in $Z_{\min}$ and take the $S^1$-equivariant blow-up of $M$ along $C$. Namely, if we let
			$C_i$ be any smooth curve of degree $i$ in $Z_{\min}$, then the $S^1$-equivariant blow-up, denoted by $M_i$, is also a semifree $S^1$-Fano variety where the fixed 
			point data
			is given by 
		\begin{itemize}
			\item $\mathrm{Vol}(Z_{\min}^1) = 16$, $\mathrm{Vol}(Z_{\max}^1) = 1$,
			\item $\mathrm{Vol}(Z_{\min}^2) = 9$, $\mathrm{Vol}(Z_{\max}^2) = 1$,
			\item $\mathrm{Vol}(Z_{\min}^3) = 4$, $\mathrm{Vol}(Z_{\max}^3) = 1$.
		\end{itemize}
		Those fixed point data exactly match up with {\bf (II-2-1.2), (II-2-1.4), (II-2-1.5)}, respectively. We will see later that 
		each $M_i$' are $S^1$-equivariantly symplectomorphic to the Fano varieties of type {\bf (II-2-1.2), (II-2-1.4), (II-2-1.5)} given above.

		\end{remark}
		  \end{itemize}
	\end{example}
				
	\begin{theorem}[Case {\bf (II-2-2)}]\label{theorem_II_2_2}
		Let $(M,\omega)$ be a six-dimensional closed monotone semifree Hamiltonian $S^1$-manifold such that $\mathrm{Crit} H = \{ 1,0,-1\}$. 
		Suppose that $M_0 \cong S^2 \times S^2$. Then the list of all possible topological fixed point data is given in the Table \ref{table_II_2_2}
		\begin{table}[H]
			\begin{tabular}{|c|c|c|c|c|c|c|c|c|}
				\hline
				    & $(M_0, [\omega_0])$ & $e(P_{-1}^+)$ & $Z_{-1}$  & $Z_0$ &  $Z_1$ & $b_2(M)$ & $c_1^3(M)$ \\ \hline \hline
				    {\bf (II-2-2.1)} & $(S^2 \times S^2, 2x + 2y)$ & $-x -y$  & $S^2 \times S^2$  & \makecell{$Z_0 \cong T^2$, \\ $\mathrm{PD}(Z_0) = 2x+2y$} & 
				    	$S^2 \times S^2$
					     & $4$ & $28$ \\ \hline    
				    {\bf (II-2-2.2)} & $(S^2 \times S^2, 2x + 2y)$ & $-x$  & $S^2 \times S^2$  & \makecell{$Z_0 \cong S^2$, \\ $\mathrm{PD}(Z_0) = 2x + y$} & 
				    	$S^2 \times S^2$
					     & $4$ & $32$ \\ \hline    
				    {\bf (II-2-2.3)} & $(S^2 \times S^2, 2x + 2y)$ & $-x$  & $S^2 \times S^2$  & \makecell{$Z_0 \cong S^2 ~\dot \cup ~S^2$, \\ 
				    $\mathrm{PD}(Z_0^1) = \mathrm{PD}(Z_0^2) = x$} & 
				    	$S^2 \times S^2$
					     & $5$ & $36$ \\ \hline    			     					     
				    {\bf (II-2-2.4)} & $(S^2 \times S^2, 2x + 2y)$ & $-x$  & $S^2 \times S^2$  & \makecell{$Z_0 \cong S^2$, \\ $\mathrm{PD}(Z_0) = x+y$} & 
				    	$S^2 \times S^2$
					     & $4$ & $36$ \\ \hline    					     
				    {\bf (II-2-2.5)} & $(S^2 \times S^2, 2x + 2y)$ & $-x + y$  & $S^2 \times S^2$  & \makecell{$Z_0 \cong S^2 ~\dot \cup ~S^2$, \\ 
				    $\mathrm{PD}(Z_0^1) = \mathrm{PD}(Z_0^2) = x$} & 
				    	$S^2 \times S^2$
					     & $5$ & $36$ \\ \hline    			     					     
				    {\bf (II-2-2.6)} & $(S^2 \times S^2, 2x + 2y)$ & $-x + y$  & $S^2 \times S^2$  & \makecell{$Z_0 \cong S^2$, \\ $\mathrm{PD}(Z_0) = x$} & 
				    	$S^2 \times S^2$
					     & $4$ & $40$ \\ \hline    					     					     
				    {\bf (II-2-2.7)} & $(S^2 \times S^2, 2x + 2y)$ & $0$  & $S^2 \times S^2$ & \makecell{$Z_0 \cong S^2$, \\ $\mathrm{PD}(Z_0) = x+y$} & 
				    	$S^2 \times S^2$
					     & $4$ & $38$ \\ \hline    					     
				    {\bf (II-2-2.8)} & $(S^2 \times S^2, 2x + 2y)$ & $0$  & $S^2 \times S^2$ & \makecell{$Z_0 \cong S^2$, \\ $\mathrm{PD}(Z_0) = x$} & 
				    	$S^2 \times S^2$
					     & $4$ & $42$ \\ \hline    					     
				    {\bf (II-2-2.9)} & $(S^2 \times S^2, 2x + 2y)$ & $y$  & $S^2 \times S^2$ & \makecell{$Z_0 \cong S^2$, \\ $\mathrm{PD}(Z_0) = x$} & 
				    	$S^2 \times S^2$
					     & $4$ & $44$ \\ \hline    					     					     					     
			\end{tabular}		
			\vs{0.5cm}			
			\caption{\label{table_II_2_2} Topological fixed point data for $\mathrm{Crit} H = \{1,0,-1\}$ with $M_0 = S^2 \times S^2$}
		\end{table}				   
	\end{theorem}
				
	\begin{proof}
		Let $e = ax + by$ and $\mathrm{PD}(Z_0) = px + qy$ for $a,b,p,q \in \Z$. Then we have 
		\[
			[\omega_{-1}] = (2+a)x + (2 + b)y \quad \text{and} \quad [\omega_1] = (2 - (a+p))x + (2 - (b + q))y
		\]
		and so 
		\[
			2 + a \geq 1, \quad 2 + b \geq 1, \quad 2 - a - p \geq 1, \quad 2 - b - q \geq 1
		\]
		by $\int_{M_t} \omega_t^2 > 0$. We also have $\mathrm{Vol}(Z_0) = 2p + 2q \geq 2$ and $(2+a)(2+b) \geq (2 - a - p)(2 - b - q)$ by \eqref{equation_II_1_assumption}. 
		Furthermore, we may assume that 
		\begin{equation}\label{equation_II_2_2}
			\begin{cases}
				p > q, ~\text{or}  \\
				p = q ~\text{and}~ b \geq a
			\end{cases}
		\end{equation}
		as there are two choices of $\{x,y\}$, i.e., other cases ($p < q$ or $p=q$ and $a > b$) can be recovered by swapping $x$ and $y$.
		Summing up, we get
		\[
			-2 \leq a + b \leq 1, \quad -1 \leq a,b, \quad a+p, b+q \leq 1 \quad p + q \geq 1, \quad p \geq q, \quad (2+a)(2+b) \geq (2 - a - p)(2 - b - q).
		\]
		Using the adjunction formula, we obtain the following table. (Note that, two solutions $(a,b,p,q) = (0,-1,1,1)$ and $(-1,1,2,-1)$ are dropped from the table below
		due to \eqref{equation_II_2_2} and the adjunction formula.)
		
		\begin{table}[H]
		\begin{adjustbox}{max width=\textwidth}
			\begin{tabular}{|c|c|c|c|c|c|c|c|c|c|c|}
				\hline
					& {\bf (II-2-2.1)} & {\bf (II-2-2.2)} & {\bf (II-2-2.3)} & {\bf (II-2-2.4)} & {\bf (II-2-2.5)} & {\bf (II-2-2.6)} & {\bf (II-2-2.7)} & 
					{\bf (II-2-2.8)} & {\bf (II-2-2.9)} \\ \hline \hline
					$e = (a,b)$ & $(-1,-1)$ & $(-1,0)$  & $(-1,0)$ & $(-1,0)$ & $(-1,1)$ & $(-1,1)$ &$(0,0)$ & $(0,0)$  & $(0,1)$\\ \hline
					$\mathrm{PD}(Z_0)= (p,q)$ & $(2,2)$ & $(2,1)$  & $(2,0)$ & $(1,1)$ & $(2,0)$  & $(1,0)$ &$(1,1)$ & $(1,0)$ & (1, 0) \\ \hline
					$\mathrm{Vol}(Z_{\min})$ & $2$ & $4$ & $4$ & $4$ & $6$ & $6$ &$8$ & $8$ & $12$ \\ \hline
					$\mathrm{Vol}(Z_{\max})$ & $2$ & $2$ & $4$ & $4$ & $2$ & $4$ &$2$ & $4$ & $2$\\ \hline
					\makecell{$\langle c_1(TM_0), [Z_0] \rangle$ \\ $= \mathrm{Vol}(Z_0)$} & $8$ & $6$ & $4$ & $4$ & $4$ & $2$ &$4$ & $2$ & $2$\\ \hline
					$[Z_0] \cdot [Z_0]$ & $8$ & $4$ & $0$ & $2$ & $0$ & $0$ &$2$ & $0$ & $0$\\ \hline
					\# components of $Z_0$ & $1$ & $1$ & $2$ & $1$ & $2$ & $1$ &$1$ & $1$ & $1$\\ \hline
					$g$ : genus of $Z_0$ & $1$ & $0$ & $(0,0)$ & $0$ & $(0,0)$ &$0$ &$0$ & $0$  & $0$\\ \hline					
			\end{tabular}		
		\end{adjustbox}
		\end{table}
		\noindent
		The Chern numbers in Table \ref{table_II_2_2} directly follows from simple computation using Lemma \ref{lemma_II_Chern_number}.
	\end{proof}				
				
	\begin{example}[Fano varieties of type {\bf (II-2-2)}]\label{example_II_2_2}\cite[No. 2,5,7,8,9,10,11 in Section 12.5, No. 2,3 in Section 12.6]{IP} 
	  We describe Fano varieties of type {\bf (II-2-2)} in Theorem \ref{theorem_II_2_2} as follows. 
			           	 
	\begin{itemize}
		\item {\bf (II-2-2.1), (II-2-2.2), (II-2-2.7)} \cite[No. 2,5,11 in Section 12.5]{IP} : Consider the toric variety $M = \pp(\mcal{O} \oplus \mcal{O}(1,1))$ whose moment polytope 
		is described in the first of Figure \ref{figure_II_2_2_1}. For the $S^1$-action generated by $\xi = (0,0,1)$, the minimal fixed component $Z_{\min}$ is diffeomorphic to 
		$\p^1 \times \p^1$ and it has volume 18, see also 
		Example \ref{example_II_1} {\bf (II-1-2.3)}.
		
	           		 \begin{figure}[H]
	           	 		\scalebox{0.8}{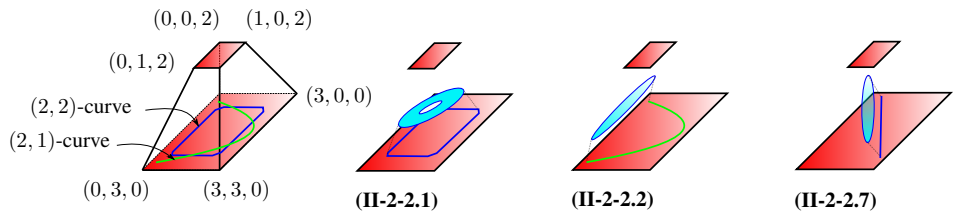}
		           	 	\caption{\label{figure_II_2_2_1} Fano varieties of types {\bf (II-2-2.1), (II-2-2.2), (II-2-2.7)}}
		           	 \end{figure}
		\noindent
		Let $C_{(a,b)}$ be a smooth curve of bidegree $(a,b)$ on $Z_{\min}$. Since $[\omega_1] = 3x + 3y$ as in Example \ref{example_II_1} {\bf (II-1-2.3)}, we may blow-up 
		of $M$ along $C_{(a,b)}$ where $(a,b) = (1,0), (2,0), (1,1),(2,1)$ and $(2,2)$. Denote by $M_{(a,b)}$ the $S^1$-equivariant blow-up of $M$ along $C_{(a,b)}$. Then
		each $M_{(a,b)}$ has three fixed components and one can check that 
		
		\begin{table}[H]
		\begin{adjustbox}{max width=\textwidth}
			\begin{tabular}{|c|c|c|c|c|c|}
				\hline
					$(a,b)$ & $(1,0)$ & $(1,1)$ & $(2,0)$ & $(2,1)$ & $(2,2)$ \\ \hline \hline
					$\mathrm{Vol}(Z_{\min})$ & $12$ & $8$ & $6$ & $4$ & $2$ \\ \hline
					$\mathrm{Vol}(Z_{\max})$ & $2$ & $2$ & $2$ & $2$ & $2$ \\ \hline
					$\mathrm{Vol}(Z_0)$ & $2$ & $4$ & $4$ & $6$ & $8$ \\ \hline
					TFD type & {\bf (II-2-2.9)} & {\bf (II-2-2.7)} & {\bf (II-2-2.5)} & {\bf (II-2-2.2)} & {\bf (II-2-2.1)} \\ \hline
			\end{tabular}		
		\end{adjustbox}
		\end{table}
		\noindent
		We visualize each cases in Figure \ref{figure_II_2_2_1} except for the cases {\bf (II-2-2.5)} and {\bf (II-2-2.9)}, where the latter two types can be 
		represented by toric Fano varieties as we will see below.
		
		\begin{remark}
			Let $M = \p^1 \times \p^1 \times \p^1$ with the standard $T^3$-action and consider the $S^1$-action generated by $\xi = (0,0,1)$. 
			Then there are two fixed components that are both diffeomorphic to $S^2 \times S^2$ with the same volume 8. If $C$ is a curve of bidegree $(1,1)$
			on $Z_{\max}$, then one can take the $S^1$-equivariant blow-up of $M$ along $C$ where we denote the resulting manifold by $\widetilde{M}$. 
			Then we have 
			\[
				\mathrm{Vol}(\widetilde{Z}_{\max}) = 2, \quad \mathrm{Vol}(\widetilde{Z}_{\min}) = 8, \quad \mathrm{Vol}(\widetilde{Z}_{0}) = 4
			\]
			and so the fixed point data is of type {\bf (II-2-2.7)}. See Figure \ref{figure_II_2_2_1_1}. We will see in Section \ref{secMainTheorem} that 
			$\widetilde{M}$ is $S^1$-equivariantly symplectomorphic to the blow-up of $\pp(\mcal{O} \oplus \mcal{O}(1,1))$ along a curve of bidegree $(1,1)$
			described in Figure \ref{figure_II_2_2_1}.
			
	           		 \begin{figure}[H]
	           	 		\scalebox{0.8}{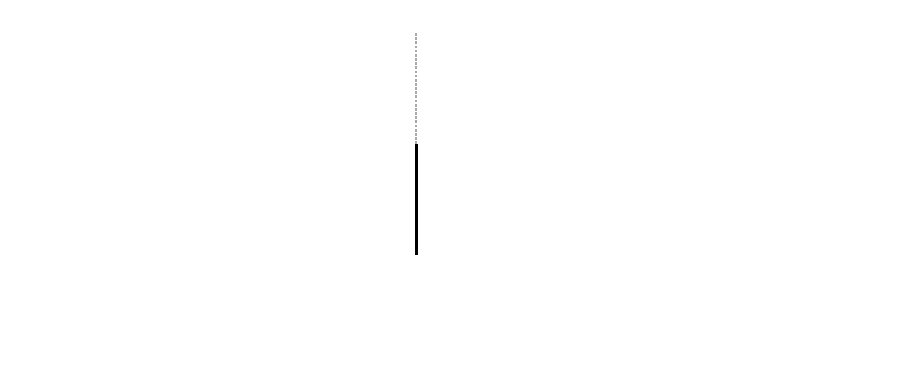}
		           	 	\caption{\label{figure_II_2_2_1_1} Fano varieties of types {\bf (II-2-2.7)}}
		           	 \end{figure}
			
		\end{remark}
		\vs{0.3cm}
		
		\item {\bf (II-2-2.3), (II-2-2.4), (II-2-2.8), (II-2-2.9)} \cite[No. 7,8 in Section 12.5, No. 2,3 in Section 12.6]{IP} : 			
		Now we let $M = \p^1 \times X_1$ admitting the standard $T^3$-action where the moment polytope is given in Figure \ref{figure_II_2_2_2}. With respect to 
		the $S^1$-action generated by $\xi = (0,0,1)$, we have two fixed components such that $\mathrm{Vol}(Z_{\min}) = 12$ and $\mathrm{Vol}(Z_{\max}) = 4$.
		
	           		 \begin{figure}[H]
	           	 		\scalebox{0.8}{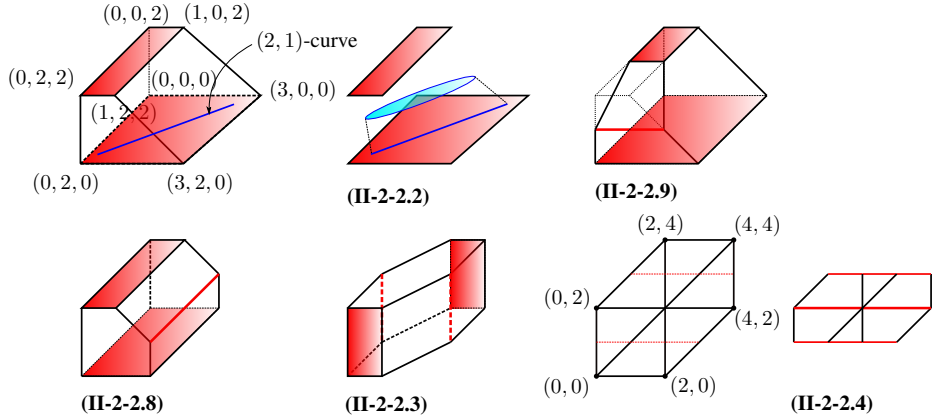}
		           	 	\caption{\label{figure_II_2_2_2} Fano varieties of types {\bf (II-2-2.3), (II-2-2.4), (II-2-2.8), (II-2-2.9)}}
		           	 \end{figure}
		
		\noindent
		Note that $[\omega_{-1}] = x + 2y$ and $[\omega_1] = 3x + 2y$. So, we may obtain a Fano variety by blowing up $M$ along 
		\begin{itemize}
			\item a curve of bidegree $(a,b)$ on $Z_{\min}$ where $(a,b) = (1,0), (2,0), (1,1),$ or $(2,1)$, or 
			\item a curve of bidegree $(0,1)$ on $Z_{\max}$.
		\end{itemize}
		and each of them has fixed point set as follows.
		\begin{table}[H]
		\begin{adjustbox}{max width=\textwidth}
			\begin{tabular}{|c|c|c|c|c|c|}
				\hline
					$(a,b)$ & $(1,0)$ & $(2,0)$ & $(1,1)$ & $(2,1)$ & $(0,1)$ \\ \hline \hline
					$\mathrm{Vol}(Z_{\min})$ & $8$ & $4$ & $4$ & $2$ & $12$ \\ \hline
					$\mathrm{Vol}(Z_{\max})$ & $4$ & $4$ & $4$ & $4$ & $2$ \\ \hline
					$\mathrm{Vol}(Z_0)$ & $2$ & $4$ & $4$ & $6$ & $2$ \\ \hline
					TFD type & {\bf (II-2-2.8)} & {\bf (II-2-2.3)} & {\bf (II-2-2.4)} & {\bf (II-2-2.2)} & {\bf (II-2-2.9)} \\ \hline
			\end{tabular}		
		\end{adjustbox}
		\end{table}
		\noindent
		
		\begin{remark}
			Note that 
			\begin{itemize}
				\item the blow-up of $\pp(\mcal{O} \oplus \mcal{O}(1,1))$ along a curve of bidegree $(2,1)$ and
				\item the blow-up of $\p^1 \times X_1$ along a curve of bidegree $(2,1)$
			\end{itemize}
			have the same fixed point data {\bf (II-2-2.2)}. We will see in Section \ref{secMainTheorem} 
			that they are $S^1$-equivariantly symplectormophic to each other.
		\end{remark}

		\begin{remark}
			Consider the complete flag variety $\mcal{F}(3)$ equipped with the $T^2$-action where the corresponding moment polytope is given 
			on the bottom-right in Figure \ref{figure_II_2_2_2}. By taking $T^2$-equivariant blow-up along two $T^2$-invariant spheres $C_{\min}$ and $C_{\max}$ corresponding
			the edges $\overline{(0,0) ~(2,0)}$ and $\overline{(2,4) ~(4,4)}$, we obtain a Fano variety and it has three fixed components such that 
			\begin{itemize}
				\item $Z_{\min} \cong Z_{\max} \cong S^2 \times S^2$ (since $\mathrm{Vol}(C_{\min}) = \mathrm{Vol}(C_{\max})  = 2$ and by Lemma \ref{lemma_Euler_extremum})
				\item $\mathrm{Vol}(Z_0) = 4$ and $\mathrm{Vol}(Z_{\min}) = \mathrm{Vol}(Z_{\max}) = 4$. 
			\end{itemize}
			So, the fixed point data is of type {\bf (II-2-2.4)}. Consequently, two semifree $S^1$-Fano varieties 
			\begin{itemize}
				\item the blow-up of $\p^1 \times X_1$ along a curve of bidegree $(1,1)$ on the mnimum, and
				\item the blow-up of $\mcal{F}(3)$ along the disjoint union of two spheres
			\end{itemize}
			have the same fixed point data. We will see in Section \ref{secMainTheorem} that they are $S^1$-equivariantly symplectomorphic. 
		\end{remark}

		\item {\bf (II-2-2.5), (II-2-2.6)} \cite[No. 9,10 in Section 12.5]{IP} : In this case, we consider $M$, the $T^3$-equivariant blow-up of $\p^3$ along two disjoint $T^3$-invariant curves
		which we described in Example \ref{example_I_3_1} {\bf (I-3-1.4)}. We fix one exceptional divisor $D$ on $M$ and let $C_1$, $C_2$ be two disjoint $T^3$-invariant curves. Then, 
		we obtain two Fano varieties
		\begin{itemize}
			\item $M_1$ : the $T^3$-equivariant blow-up of $M$ along $C_1 ~\dot \sqcup ~C_2$.
			\item $M_2$ : the $T^3$-equivariant blow-up of $M$ along $C_1$,
		\end{itemize}
		One can check that the $S^1$-action on each $M_i$ generated by $\xi = (-1,-1,0)$ is semifree. Then it immediately follows from Figure \ref{figure_II_2_2_3} that   
		\begin{table}[H]
		\begin{adjustbox}{max width=\textwidth}
			\begin{tabular}{|c|c|c|}
				\hline
							& $M_1$ & $M_2$ \\ \hline \hline
					$\mathrm{Vol}(Z_{\min})$ & $6$ & $6$ \\ \hline
					$\mathrm{Vol}(Z_{\max})$ & $2$ & $4$ \\ \hline
					$\mathrm{Vol}(Z_0)$ & $(2,2)$ & $2$ \\ \hline
					TFD type & {\bf (II-2-2.5)} & {\bf (II-2-2.6)} \\ \hline
			\end{tabular}		
		\end{adjustbox}
		\end{table}
		\noindent
		where the fixed component $Z_0$ in $M_1$ consists of two spheres corresponding to the edges $\overline{(0,1,1) ~(1,0,1)}$ and $\overline{(0,2,0) ~(2,0,0)}$ in the middle of 
		Figure \ref{figure_II_2_2_3} whose affine distances are both two. See also {\bf (I-3-1.5)} and \cite[Example 8.3]{Cho2}.
	\end{itemize}					           	 
	
	           		 \begin{figure}[H]
	           	 		\scalebox{0.8}{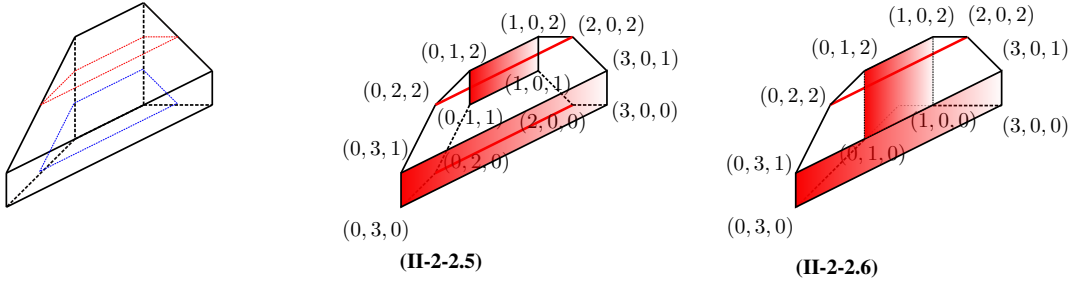}
		           	 	\caption{\label{figure_II_2_2_3} Fano varieties of type {\bf (II-2-2.5) and (II-2-2.6)}}
		           	 \end{figure}

	\end{example}
								
	\begin{theorem}[Case {\bf (II-2-3)}]\label{theorem_II_2_3}
		Let $(M,\omega)$ be a six-dimensional closed monotone semifree Hamiltonian $S^1$-manifold such that $\mathrm{Crit} H = \{ 1,0,-1\}$. 
		Suppose that $M_0 \cong X_1$. Then the list of all possible topological fixed point data is given in the Table \ref{table_II_2_3}
		\begin{table}[H]
			\begin{tabular}{|c|c|c|c|c|c|c|c|c|}
				\hline
				    & $(M_0, [\omega_0])$ & $e(P_{-1}^+)$ & $Z_{-1}$  & $Z_0$ &  $Z_1$ & $b_2(M)$ & $c_1^3(M)$ \\ \hline \hline
				    {\bf (II-2-3.1)} & $(X_1, 3u - E_1)$ & $-u$  & $X_1$  & \makecell{$Z_0 \cong S^2$, \\ $\mathrm{PD}(Z_0) = 2u$} & 
				    	$X_1$
					     & $4$ & $32$ \\ \hline    
				    {\bf (II-2-3.2)} & $(X_1, 3u - E_1)$ & $-E_1$  & $X_1$  & \makecell{$Z_0 \cong S^2 ~\dot \cup ~ S^2$, \\ $\mathrm{PD}(Z_0^1) = u$ \\ $\mathrm{PD}(Z_0^2) = E_1$} & 
				    	$X_1$
					     & $5$ & $36$ \\ \hline    
				    {\bf (II-2-3.3)} & $(X_1, 3u - E_1)$ & $0$  & $X_1$  & \makecell{$Z_0 \cong S^2 $, \\ 
				    $\mathrm{PD}(Z_0) = u$} & 
				    	$X_1$
					     & $4$ & $40$ \\ \hline    			     					     
				    {\bf (II-2-3.4)} & $(X_1, 3u - E_1)$ & $u - E_1$  & $X_1$  & \makecell{$Z_0 \cong S^2$, \\ $\mathrm{PD}(Z_0) = E_1$} & 
				    	$X_1$
					     & $4$ & $46$ \\ \hline    					     
				    {\bf (II-2-3.5)} & $(X_1, 3u - E_1)$ & $0$  & $X_1$  & \makecell{$Z_0 \cong S^2 $, \\ 
				    $\mathrm{PD}(Z_0) = E_1$} & 
				    	$X_1$
					     & $4$ & $44$ \\ \hline    			     					     
			\end{tabular}		
			\vs{0.5cm}			
			\caption{\label{table_II_2_3} Topological fixed point data for $\mathrm{Crit} H = \{1,0,-1\}$ with $M_0 = X_1$}
		\end{table}				   
	\end{theorem}

	\begin{proof}
		Let $e = au + bE_1$ and $\mathrm{PD}(Z_0) = pu + qE_1$ for some $a,b,p,q \in \Z$. By the Dustermaat-Heckman theorem, we get 
		\[
			[\omega_{-1}] = (3+a)u + (b-1)E_1 \quad \text{and} \quad [\omega_1] = (3 - (a+p))u + (-1 - (b + q))E_1. 
		\]
		By the following inequalities $\int_{M_t} \omega_t^2 > 0$, $\langle [\omega_t], E_1 \rangle > 0$, and $\mathrm{Vol}(Z_0) > 0$, we obtain
		\[
			3+ a > 1 - b \geq 1, \quad 3- a - p > 1 + b + q \geq 1, \quad 3p + q \geq 1.
		\]
		Moreover, the condition \eqref{equation_II_1_assumption} allows us to assume that
		\[
			(3+a)^2 - (b-1)^2 \geq (3 - a - p)^2 - (1 + b + q)^2.
		\]
		Combining those equations, we obtain 
		\[
			p+q \leq 2, \quad 3p + q \geq 1, \quad q \geq 0
		\]
		whose integral solutions are $(2,0), (1,0), (0,1), (0,2), (1,1)$. Using the adjunction formula, we obtain the following table. 
		which coincides with Table \ref{table_II_2_3}. 
		\begin{table}[H]
		\begin{adjustbox}{max width=\textwidth}
			\begin{tabular}{|c|c|c|c|c|c|c|c|}
				\hline
					& {\bf (II-2-2.1)} & {\bf (II-2-2.2)} & {\bf (II-2-2.3)} & {\bf (II-2-2.4)} & {\bf (II-2-2.5)} & {\bf (II-2-2.6)} & {\bf (II-2-2.7)}\\ \hline \hline
					$\mathrm{PD}(Z_0)= (p,q)$ & $(2,0)$ & $(1,1)$ & $(1,0)$ & $(0,1)$  & $(0,1)$ & $(0,2)$ & $(0,2)$  \\ \hline					
					$e = (a,b)$ & $(-1,0)$ & $(0, -1)$ & $(0,0)$ & $(1,-1)$ & $(0,0)$ & $(1,-2)$ & $(0,-1)$ \\ \hline
					$\mathrm{Vol}(Z_{\min})$ & $3$ & $5$ & $8$ & $12$ & $8$ & $7$ & $5$ \\ \hline
					$\mathrm{Vol}(Z_{\max})$ & $3$ & $3$ & $3$ & $3$ & $5$ & $3$ & $5$ \\ \hline
					\makecell{$\langle c_1(TM_0), [Z_0] \rangle$ \\ $= \mathrm{Vol}(Z_0)$} & $6$ & $4$ & $3$ & $1$ & $1$ & $2$ & $2$ \\ \hline
					$[Z_0] \cdot [Z_0]$ & $4$ & $0$ & $1$ & $-1$ & $-1$ & $-4$ & $-4$ \\ \hline
					\# components & $1$ & $2$ & $1$ & $1$ & $1$ & $\times$ & $\times$ \\ \hline
					$g$ : genus of $Z_0$ & $0$ & $(0,0)$ & $0$ & $0$ & $0$ & $\times$ & $\times$ \\ \hline					
			\end{tabular}		
		\end{adjustbox}
		\end{table}
		\noindent
		Note that the Chern numbers in Table \ref{table_II_2_3} can be obtained from Lemma \ref{lemma_II_Chern_number} by direct computation.
	\end{proof}
				
	\begin{example}[Fano varieties of type {\bf (II-2-3)}]\label{example_II_2_3}\cite[No. 4,9,11,12 in Section 12.5, No. 2 in Section 12.6]{IP} 
	  We describe Fano varieties of type {\bf (II-2-3)} in Theorem \ref{theorem_II_2_3} as follows. 
			           	 
	\begin{itemize}
		\item {\bf (II-2-3.1), (II-2-3.2), (II-2-3.3), (II-2-3.4)} \cite[No. 4,9,12 in Section 12.5, No. 2 in Section 12.6]{IP} : 		
		Let $M$ be the $T^3$-equivariant blow-up of $V_7$ along a $T^3$-invariant line passing through the exceptional divisor of $V_7 \rightarrow \p^3$
		where the moment polytope of $M$ is depicted on the left of Figure \ref{figure_II_2_3_1}. 
		Then the $S^1$-subgroup of $T^3$ generated by $\xi = (0,0,1)$ acts on $M$ semifreely such that $\mathrm{Vol}(Z_{\min}) = 15$ and $\mathrm{Vol}(Z_{\max}) = 3$.		
		See also Example \ref{example_II_1} {\bf (II-1-3.2)}. 			
		
			 \begin{figure}[H]
	           	 		\scalebox{0.8}{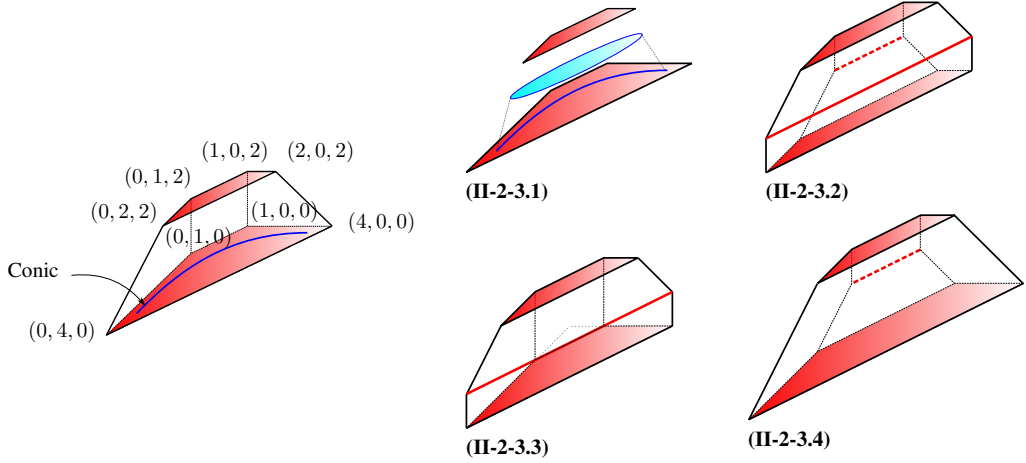}
		           	 	\caption{\label{figure_II_2_3_1} {\bf (II-2-3.1)} : Blow-up of $M$ along the conic}
	           	 \end{figure}
	           	 
		\noindent
		There are four possible ways of blowing $M$ up to obtain a semifree $S^1$-Fano variety. Let $C_{(a,b)}$ be a smooth (possibly disconnected) curve 
		of bidegree $(a,b)$ in $Z_{\min} \cong X_1$. Since $[\omega_{-1}] = 4u - E_1$ and by the adjunction formula, we can easily check that all possible $(a,b)$'s are $(1,0), (2,0), (0,1), (1,1)$.
		(Note that $(4u - E_1) - (au  + bE_1)$ becomes the symplectic form on the minimal fixed component of the blow-up of $M$ along $C_{(a,b)}$.) Denote by $M_{(a,b)}$ the $S^1$-equivariant 
		blow-up of $M$ along $C_{(a,b)}$. Then
		\begin{itemize}
			\item $M_{(1,0)}$ is the blow-up of $M$ along a line $H$ (corresponding to the edge $\overline{(4,0,0)~(0,4,0)}$),
			\item $M_{(2,0)}$ is the blow-up of $M$ along a conic,
			\item $M_{(0,1)}$ is the blow-up of $M$ along an exceptional line $E$ (corresponding to the edge $\overline{(1,0,0)~(0,1,0)}$),
			\item $M_{(1,1)}$ is the blow-up of $M$ along a disjoint union of an exceptional line $E$ and a line $H$.
		\end{itemize}
		Moreover, the volume of the extremal fixed component for each $M_{(a,b)}$ can be computed directly from Figure \ref{figure_II_2_3_1} and therefore we obtain the following.
		\vs{0.3cm}	      
		\begin{table}[H]
		\begin{adjustbox}{max width=\textwidth}
			\begin{tabular}{|c|c|c|c|c|}
				\hline
							& $M_{(1,0)}$ & $M_{(2,0)}$ &  $M_{(0,1)}$ & $M_{(1,1)}$ \\ \hline \hline
					$\mathrm{Vol}(Z_{\min})$ & $8$ & $3$ & $12$ & $5$ \\ \hline
					$\mathrm{Vol}(Z_{\max})$ & $3$ & $3$ & $3$ & $3$  \\ \hline
					$\mathrm{Vol}(Z_0)$ & $3$ & $(3,3)$ & $1$ & $4$ \\ \hline
					TFD type & {\bf (II-2-3.3)} & {\bf (II-2-3.1)} & {\bf (II-2-3.4)} & {\bf (II-2-3.2)} \\ \hline
			\end{tabular}		
		\end{adjustbox}
		\end{table}
		\noindent
	           	 
		\item {\bf (II-2-3.5)} \cite[No. 11 in Section 12.5]{IP} : Consider $M = \p^1 \times X_1$ with the standard $T^3$-action and let $\widetilde{M}$ be the $T^3$-equivariant blow-up 
		along the $T^3$-invariant curve 	$t \times E$ where $t$ is a fixed point of the $S^1$-action on $\p^1$ and $E$ is the exceptional divisor of $X_1$. (Note that the curve $t \times E$
		corresponds to the edge $\overline{(0,2,2) ~(1,2,2)}$ in Figure \ref{figure_II_2_3_2}.)
		
			 \begin{figure}[H]
	           	 		\scalebox{0.8}{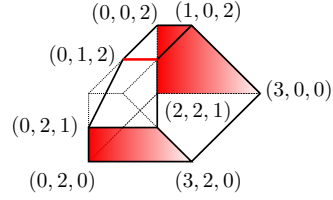}
		           	 	\caption{\label{figure_II_2_3_2} {\bf (II-2-3.5)} : Blow-up of $\p^1 \times X_1$ along $t \times E$}
	           	 \end{figure}		
	           \noindent
	           The $S^1$ subgroup of $T^3$ generated by $\xi = (0,1,0)$ is semifree and the fixed point set satisfies 
	           \[
	           	\mathrm{Vol}(Z_{\min}) = 8, \quad \mathrm{Vol}(Z_{0}) = 1, \quad \mathrm{Vol}(Z_{\max}) = 5
	           \]
	           and therefore the fixed point data coincides with {\bf (II-2-3.5)} in Table \ref{table_II_2_3}.
	           	 
	\end{itemize}			           	 

	\end{example}

\section{Main Theorem}
\label{secMainTheorem}

	In this section, we prove our main theorem as follows.
	
	\begin{theorem}[Theorem \ref{theorem_main}]
	Let $(M,\omega)$ be a six-dimensional closed monotone symplectic manifold equipped with a semifree Hamiltonian circle action such that $c_1(TM) = [\omega]$.
	Then $(M,\omega)$ is $S^1$-equivariantly symplectomorphic to some K\"{a}hler Fano manifold with a certain holomorphic Hamiltonian circle action. 
	Indeed, there are (18 + 21 + 56) types of such manifolds up to $S^1$-equivariant symplectomorphism.
	\end{theorem}

	Notice that we have already proved Theorem \ref{theorem_main} for the following cases:
	\begin{enumerate}
		\item \cite{Cho1} $Z_{\min}$ is isolated, and 
		\item \cite{Cho2} $\dim Z_{\min} = \dim Z_{\max} = 2$. 
	\end{enumerate} 
	In addition, we have classified all such manifolds up to $S^1$-equivariant symplectomorphism and seen that there are exactly 18 types and 21 types for each cases.
	See \cite[Table 9.1]{Cho1} and \cite[Table 9.1]{Cho2}. Therefore we only need to prove the case where $\dim Z_{\max} = 4$ and $\dim Z_{\min} \geq 2$.
	
	In Section \ref{secTopologicalFixedPointDataDimZMin2} and Section
	 \ref{secTopologicalFixedPointDataDimZMax4} we have classified all possible topological fixed point data (56 types) listed in Table \ref{table_list_1} and Table \ref{table_list_2}.
	 An important consequence of our classification is that any reduced space of $(M,\omega)$ in Theorem \ref{theorem_main} is diffeomorphic to one of the followings:
	\[
		\begin{cases}
			X_k \quad (2 \leq k \leq 8) & \text{\bf (II-1-4.k)} \\
			\p^2, S^2 \times S^2, ~\text{or} ~X_k (k \leq 3) & \text{otherwise}. 
		\end{cases}
	\]	
	Therefore in case that $(M,\omega)$ is not of type {\bf (II-1-4.k)}, every reduced of $(M,\omega)$ is symplectically rigid by Corollary \ref{corollary_symplectically_rigid}.
	Thus Gonzalez's theorem \ref{theorem_G} would be applied so that we obtain the following.
	
	\begin{lemma}\label{lemma_FD_M}
		Suppose that $(M,\omega)$ is not of type {\bf (II-1-4.k)}. Then the fixed point data determines $(M,\omega)$ uniquely. 
	\end{lemma}

	\begin{proposition}\label{proposition_TFD_FD_1}
		Suppose that $(M,\omega)$ is not of type {\bf (II-1-4.k)}. Then the topological fixed point data of $(M,\omega)$ determines the fixed point data $(M,\omega)$ uniquely.
	\end{proposition}
	
	\begin{proof}
		It is obvious that TFD determines FD on an extremal level (since the reduced space is itself a fixed component). 
		
		For a non-extremal critical value $-1$, the fixed point data consists of the reduced space $M_{-1}$, the symplectic class $[\omega_{-1}]$, 
		and homology class of $Z_{-1}$, that is, isolated fixed points of index two. Suppose that two possible fixed point data are given as
		\[
		(M_{-1}, \omega_{-1}, p_1, \cdots, p_r) \quad \text{and} \quad (M_{-1}, \omega_{-1}' , q_1, \cdots, q_r), \quad \quad p_i, q_j : \text{points},\quad  [\omega_c] = [\omega_c'].
		\]	
		By the symplectic rigidity of $M_{-1}$, we have a symplectomorphism $\phi : (M_c, \omega_c) \rightarrow (M_c, \omega_c')$. Moreover, it can be chosen to be
		sending $p_i$ to $q_i$ for $i=1,\cdots,r$ by \cite[Proposition 0.3]{ST}. Thus the TFD determines the FD in case of level $-1$. 
		
		For level zero, the TFD is of the form
		\[
			(M_0, [\omega_0], [Z_0^1], \cdots, [Z_0^{k}]), \quad \quad k = \text{number of components of $Z_0$}.
		\]
		So, we need to prove that a homology class of each $Z_0^i$ determines a symplectic submanifold representing $[Z_0^1]$ uniquely (up to symplectomorphism).
		On the other hand, one can immediately check that $[Z_0^i]\cdot [Z_0^i] \geq -1$ and hence 
		\begin{itemize}
			\item $Z_0^i$ is algebraic (by Theorem \ref{theorem_ST} and Theorem \ref{theorem_Z}), 
			\item any two smooth algebraic curves in $M_0$ (where $H^1(M_0, \mcal{O}_{M_0})$) are symplectically isotopic to each other (by \cite[Lemma 9.6]{Cho1}).
		\end{itemize}
		Therefore, any two possible FD's representing the same TFD are symplectomorphic to each other and this finishes the proof.
	\end{proof}
	
	We state two theorems (due to Seibert-Tian,  Li-Wu, and Zhang) used in the proof of Proposition \ref{proposition_TFD_FD_1} as below.
	
	\begin{theorem}\cite[Theorem C]{ST}\label{theorem_ST}
		Any symplectic surface in $\p^2$ of degree $d \leq17$ is symplectically isotopic to an algebraic curve.
	\end{theorem}
	
	\begin{theorem}\cite[Proposition 3.2]{LW}\cite[Theorem 6.9]{Z}\label{theorem_Z}
		Any symplectic sphere $S$ with self-intersection $[S]\cdot[S] \geq 0$ in a symplectic four manifold $(M,\omega)$ is symplectically isotopic to an (algebraic) rational curve.
		Any two homologous spheres with self-intersection $-1$ are symplectically isotopic to each other.
	\end{theorem}

	Now we are ready to prove Theorem \ref{theorem_main}
		
	\begin{proof}[Proof of Theorem \ref{theorem_main}]
		If $(M,\omega)$ is given in Theorem \ref{theorem_main} not of type {\bf (II-1-4.k)}, then $M$ is (algebraic) Fano since there exists 
		a Fano variety whose TFD equals the one of $(M,\omega)$ as we have seen in Section \ref{secTopologicalFixedPointDataDimZMin2} and
		 \ref{secTopologicalFixedPointDataDimZMax4} and so the Fano variety should be equivalent to $(M,\omega)$ by Proposition \ref{proposition_TFD_FD_1}.
		 
		On the other hand, if $(M,\omega)$ is of type {\bf (II-1-4.k)}, then the TFD has no extremal fixed component and the Euler class of each level set (as a principal $S^1$-bundle)
		vanishes, see Table \ref{table_list_2}. Moreover, since any reduced space of $M$ is $X_k$, which is in particular a rational surface, any two deformation equivalent cohomologous 
		symplectic forms are isotopic by Theorem \ref{theorem_uniqueness}.
		
		 If $(M', \omega')$ is another semifree Hamiltonian $S^1$-manifold having the same TFD of $(M,\omega)$'s, 
		we obtain cohomologous families of reduced symplectic forms $\{\omega_t\}$ from $M$ and $\{\omega_t'\}$ from $M'$ such that $[\omega_t] = [\omega_t']$ for every $t \in [-1,1]$. 
		Note that both $\{\omega_t\}$ and $\{\omega_t'\}$ can be regarded as families of symplectic forms on $M_0$ since $M_0 \cong M_t$ for every $t \in [-1,1]$. Moreover, 
		by the uniqueness of symplectic structure on $M_{-1}$ (Theorem \ref{theorem_uniqueness}), we may assume that $\omega_{-1} = \omega_{-1}'$. 
		
		Define a two-parameter family of symplectic forms on $M_0$:
		\[
			\omega_{s,t} := \begin{cases}
				\omega_{(1-2s)t} & \text{if $0 \leq s \leq \frac{1}{2}$} \\
				\omega_{(2s - 1)t}' & \text{if $\frac{1}{2} \leq s \leq 1$} \\
			\end{cases}
		\]
		Then $[\omega_{s,t}]$ is a constant class in $H^2(M_0)$ by the vanishing of the Euler class $e(P_{-1}^+)$, see Table \ref{table_list_2} {\bf (II-1-4.k)}. Then 
		$\{ \omega_{s,t} \}$ satisfies 
		\begin{itemize}
			\item $\omega_{0, t} = \omega_t$ and $\omega_{1,t} = \omega_t'$,
			\item $\frac{d}{ds} [\omega_{s,t}] = 0$ 
		\end{itemize}
		for any $0 \leq s \leq 1$ and $-1 \leq t \leq 1$. (Equivalently, $\{\omega_t\}$ and $\{\omega_t'\}$ are equivalent in the sense of \cite[Definition 3.3]{G}.)
		Therefore, we can apply the Moser type argument to show that there exists a family of symplectomorphisms 
		\[
			\varphi_t : (M_0, \omega_t) \rightarrow (M_0, \omega_t'), \quad \varphi_t^*  \omega_t' =  \omega_t
		\]
		for every $t \in [-1,1]$. This family $\{\varphi_t\}$ induces an $S^1$-equivariant symplectomorphism
		\[
			\varphi : P \times [-1,1]  \rightarrow P \times [-1,1], \quad \varphi^*\widetilde{\omega}' = \widetilde{\omega}
		\]
		where 
		\begin{itemize}
			\item $P$ is the principal $S^1$-bundle $S^1 \times M_0 \rightarrow M_0$,
			\item $\widetilde{\omega}$ and $\widetilde{\omega}'$ are two $S^1$-invariant symplectic forms on $P \times I$ determined by $\{ \omega_t\}$ and $\{\omega_t'\}$, respectively.
		\end{itemize}
		See \cite[Proposition 5.8]{McS1}.
		Then $\varphi$ implies the symplectomorphism from $(M,\omega)$ to $(M,\omega')$ induced by the symplectic cut of
		$(P \times I, \widetilde{\omega})$ and $(P \times I, \widetilde{\omega}')$ respectively. 
		This completes the proof.
	\end{proof}
	
	\newpage
	
	\begin{table}[h]
		\begin{adjustbox}{max width=\textwidth}	
		\begin{tabular}{|c|c|c|c|c|c|c|c|c|}
			\hline
			& $(M_0, [\omega_0])$ & $e(P_{-2}^+)$ & $Z_{-2}$ & $Z_{-1}$ &  $Z_0$ & $Z_1$ & $b_2$ & $c_1^3$ \\ \hline \hline
		{\bf (I-1-1.1)} & $(E_{S^2}, 3x + 2y)$  & $-x-y$  & $S^2$ & & & $E_{S^2}$ & $2$ & $54$ \\ \hline    
		{\bf (I-1-2.1)} & $(S^2 \times S^2, 2x + 2y)$ & $-y$  & $S^2$ & & & $S^2 \times S^2$ & $2$ & $54$ \\ \hline    
		{\bf (I-1-2.2)} & $(S^2 \times S^2, 2x + 2y)$ & $x-y$  & $S^2$ & & & $S^2 \times S^2$ & $2$ & $54$ \\ \hline    
		{\bf (I-2)} & $(E_{S^2}, 3x + 2y - E_1)$ & $-x-y$  & $S^2$ & $\mathrm{pt}$ &  & $E_{S^2} \# ~\overline{\p^2}$ & $3$ & $46$ \\ \hline    
		{\bf (I-3-1.1)} & $(E_{S^2}, 3x + 2y)$ & $-y$ & & $S^2$ & \makecell{$Z_0 \cong S^2$ \\ $\mathrm{PD}(Z_0) = y$} & $E_{S^2}$ & $3$ & $52$ \\ \hline    
		{\bf (I-3-1.2)} & $(E_{S^2}, 3x + 2y)$ & $-x-y$  & $S^2$ & & \makecell{$Z_0 \cong S^2$ \\ $\mathrm{PD}(Z_0) = y$} & $E_{S^2}$ & $3$ & $50$ \\ \hline    
		{\bf (I-3-1.3)} & $(E_{S^2}, 3x + 2y)$ & $-y$  & $S^2$ &  &
				    				\makecell{$Z_0 \cong S^2 ~\dot\cup ~S^2$ \\ $\mathrm{PD}(Z_0^1) = x+y$ \\ $\mathrm{PD}(Z_0^2) = y$} 				    
				    															& $E_{S^2}$ & $4$ & $44$ \\ \hline    
		{\bf (I-3-1.4)} & $(E_{S^2}, 3x + 2y)$ & $-x-y$  & $S^2$ & & \makecell{$Z_0 \cong S^2$ \\ $\mathrm{PD}(Z_0) = x + y$} 
				    								& $E_{S^2}$ & $3$ & $44$ \\ \hline    
		{\bf (I-3-1.5)} & $(E_{S^2}, 3x + 2y)$ & $-x-y$  & $S^2$ & &
				    							\makecell{$Z_0 \cong S^2 ~\dot\cup ~S^2$ \\ $\mathrm{PD}(Z_0^1) = x+y$ \\ $\mathrm{PD}(Z_0^2) = y$} 
				    								& $E_{S^2}$ & $4$ & $40$ \\ \hline    
		{\bf (I-3-1.6)} & $(E_{S^2}, 3x + 2y)$ & $-x-y$  & $S^2$ & &\makecell{$Z_0 \cong S^2$ \\ $\mathrm{PD}(Z_0) = 2x + 2y$} & $E_{S^2}$ & $3$ & $36$ \\ \hline    
		{\bf (I-3-2.1)} & $(S^2 \times S^2, 2x + 2y)$ & $-y$  & $S^2$ & &\makecell{$Z_0 \cong S^2$ \\ $\mathrm{PD}(Z_0) = y$} & $S^2 \times S^2$ & $3$ & $48$ \\ \hline    
		{\bf (I-3-2.2)} & $(S^2 \times S^2, 2x + 2y)$ & $x-y$  & $S^2$ & &\makecell{$Z_0 \cong S^2$ \\ $\mathrm{PD}(Z_0) = y$} 
				    															&$S^2 \times S^2$ & $3$ & $50$ \\ \hline    
		{\bf (I-3-2.3)} & $(S^2 \times S^2, 2x + 2y)$ & $-y$  & $S^2$ &  &
				    				\makecell{$Z_0 \cong S^2 ~\dot\cup ~S^2$ \\ $\mathrm{PD}(Z_0^1) = y$ \\ $\mathrm{PD}(Z_0^2) = y$} 				    
				    															&$S^2 \times S^2$ & $4$ & $42$ \\ \hline    
		{\bf (I-3-2.4)} & $(S^2 \times S^2, 2x + 2y)$ & $x-y$  & $S^2$ & &
				    								\makecell{$Z_0 \cong S^2 ~\dot\cup ~S^2$ \\ $\mathrm{PD}(Z_0^1) = y$ \\ $\mathrm{PD}(Z_0^2) = y$} 
				    								& $S^2 \times S^2$ & $4$ & $46$ \\ \hline    
		{\bf (I-3-2.5)} & $(S^2 \times S^2, 2x + 2y)$ & $-y$  & $S^2$ & &
				    							\makecell{$Z_0 \cong S^2$ \\ $\mathrm{PD}(Z_0) = x$}
				    								& $S^2 \times S^2$ & $3$ & $46$ \\ \hline    
		{\bf (I-3-2.6)} & $(S^2 \times S^2, 2x + 2y)$ & $-y$  & $S^2$ & &
				    							\makecell{$Z_0 \cong S^2$ \\ $\mathrm{PD}(Z_0) = x + y$}
				    								& $S^2 \times S^2$ & $3$ & $42$ \\ \hline    
		{\bf (I-3-2.7)} & $(S^2 \times S^2, 2x + 2y)$ & $-y$  & $S^2$ & &
				    							\makecell{$Z_0 \cong S^2$ \\ $\mathrm{PD}(Z_0) = x + 2y$}
				    								& $S^2 \times S^2$ & $3$ & $38$ \\ \hline    				    									{\bf (I-4-1.1)} & \makecell{$(E_{S^2} \# \overline{\p^2},$ \\ $3x + 2y - E_1)$} & $-x-y$  & $S^2$ & $\mathrm{pt}$ & 
				    								\makecell{ $Z_0 \cong S^2$  \\ $\mathrm{PD}(Z_0) = x + y - E_1$} & $X_2$ & $4$ & $40$ \\ \hline    
		{\bf (I-4-1.2)} & \makecell{$(E_{S^2} \# 2\overline{\p^2},$ \\ $3x + 2y - E_1 - E_2)$} & $-x-y$  & $S^2$ & $\mathrm{2 ~pts}$ & 
			\makecell{ $Z_0 \cong S^2$  \\  $\mathrm{PD}(Z_0) = x + y - E_1 - E_2$} & $X_3$ & $5$ & $36$ \\ \hline    					     
		{\bf (I-4.2)} & \makecell{$(S^2 \times S^2 \# \overline{\p^2},$ \\  $2x + 2y - E_1)$} & $-y$  & $S^2$ & $\mathrm{pt}$ & 
		\makecell{ $Z_0 \cong S^2$  \\ $\mathrm{PD}(Z_0) = y - E_1$} & $X_2$ & $4$ & $44$ \\ \hline    					     					     		
		\end{tabular}
		\end{adjustbox}		
		\vs{0.3cm}
		\caption {List of topological fixed point data for $\dim Z_{\min} = 2$ and $\dim Z_{\max} = 4$} \label{table_list_1} 
	\end{table}
	\newpage

	\begin{table}[H]
		\begin{adjustbox}{max width=.9\textwidth}
		\begin{tabular}{|c|c|c|c|c|c|c|c|}
			\hline
			& $(M_0, [\omega_0])$ & $e(P_{-1}^+)$  & $Z_{-1}$ &  $Z_0$ & $Z_1$ & $b_2$ & $c_1^3$ \\ \hline \hline
				    {\bf (II-1-1.1)} & $(\p^2, 3u)$ & $0$  & $\p^2$  & & 
				    	$\p^2$
					     & $2$ & $54$ \\ \hline    
				    {\bf (II-1-1.2)} & $(\p^2, 3u)$ & $u$  & $\p^2$  && 
				    	$\p^2$
					     & $2$ & $56$ \\ \hline    
				    {\bf (II-1-1.3)} & $(\p^2, 3u)$ & $2u$  & $\p^2$  && 
				    	$\p^2$
					     & $2$ & $62$ \\ \hline    
					     					     					     
				    {\bf (II-1-2.1)} & $(S^2 \times S^2, 2x + 2y)$ & $0$  & $S^2 \times S^2$  && 
				    	$S^2 \times S^2$
					     & $3$ & $48$ \\ \hline    					     
				    {\bf (II-1-2.2)} & $(S^2 \times S^2, 2x + 2y)$ & $x$  & $S^2 \times S^2$ && 
				    	$S^2 \times S^2$
					     & $3$ & $48$ \\ \hline    					     
				    {\bf (II-1-2.3)} & $(S^2 \times S^2, 2x + 2y)$ & $x+y$  & $S^2 \times S^2$ && 
				    	$S^2 \times S^2$
					     & $3$ & $52$ \\ \hline    					     
				    {\bf (II-1-2.4)} & $(S^2 \times S^2, 2x + 2y)$ & $x-y$  & $S^2 \times S^2$ && 
				    	$S^2 \times S^2$
					     & $3$ & $44$ \\ \hline    					     
					     					     					     					     
				    {\bf (II-1-3.1)} & $(E_{S^2}, 3x + 2y)$ & $0$  & $E_{S^2}$ && 
				    	$E_{S^2}$
					     & $3$ & $48$ \\ \hline    					     					     
				    {\bf (II-1-3.2)} & $(E_{S^2}, 3x + 2y)$ & $x + y$  & $E_{S^2}$ && 
				    	$E_{S^2}$
					     & $3$ & $50$ \\ \hline    					     					     
					     					     
				    \makecell{{\bf (II-1-4.k)} \\ {\bf k = 2$\sim$8}} & $(X_k, 3u - \sum_{i=1}^k E_i)$ & $0$  & $X_k$ && 
				    	$X_k$
					     & $k+2$ & $54-6k$ \\ \hline    					     					     					     	
				    {\bf (II-2-1.1)} & $(\p^2, 3u)$ & $0$  & $\p^2$  & \makecell{$Z_0 \cong S^2$, \\ $\mathrm{PD}(Z_0) = u$} & 
				    	$\p^2$
					     & $3$ & $46$ \\ \hline    
				    {\bf (II-2-1.2)} & $(\p^2, 3u)$ & $u$  & $\p^2$  & \makecell{$Z_0 \cong S^2$, \\ $\mathrm{PD}(Z_0) = u$} & 
				    	$\p^2$
					     & $3$ & $50$ \\ \hline    
				    {\bf (II-2-1.3)} & $(\p^2, 3u)$ & $-u$  & $\p^2$  & \makecell{$Z_0 \cong S^2$, \\ $\mathrm{PD}(Z_0) = 2u$} & 
				    	$\p^2$
					     & $3$ & $38$ \\ \hline    
					     					     					     
				    {\bf (II-2-1.4)} & $(\p^2, 3u)$ & $0$  & $\p^2$  & \makecell{$Z_0 \cong S^2$, \\ $\mathrm{PD}(Z_0) = 2u$} & 
				    	$\p^2$
					     & $3$ & $40$ \\ \hline    					     
				    {\bf (II-2-1.5)} & $(\p^2, 3u)$ & $-u$  & $\p^2$ & \makecell{$Z_0 \cong T^2$, \\ $\mathrm{PD}(Z_0) = 3u$} & 
				    	$\p^2$
					     & $3$ & $32$ \\ \hline    					     
				    {\bf (II-2-1.6)} & $(\p^2, 3u)$ & $-2u$  & $\p^2$ & \makecell{$Z_0 \cong \Sigma_3$, \\ $\mathrm{PD}(Z_0) = 4u$} & 
				    	$\p^2$
					     & $3$ & $26$ \\ \hline    					     
				    {\bf (II-2-2.1)} & $(S^2 \times S^2, 2x + 2y)$ & $-x -y$  & $S^2 \times S^2$  & \makecell{$Z_0 \cong T^2$, \\ $\mathrm{PD}(Z_0) = 2x+2y$} & 
				    	$S^2 \times S^2$
					     & $4$ & $28$ \\ \hline    
				    {\bf (II-2-2.2)} & $(S^2 \times S^2, 2x + 2y)$ & $-x$  & $S^2 \times S^2$  & \makecell{$Z_0 \cong S^2$, \\ $\mathrm{PD}(Z_0) = 2x + y$} & 
				    	$S^2 \times S^2$
					     & $4$ & $32$ \\ \hline    
				    {\bf (II-2-2.3)} & $(S^2 \times S^2, 2x + 2y)$ & $-x$  & $S^2 \times S^2$  & \makecell{$Z_0 \cong S^2 ~\dot \cup ~S^2$, \\ 
				    $\mathrm{PD}(Z_0^1) = \mathrm{PD}(Z_0^2) = x$} & 
				    	$S^2 \times S^2$
					     & $5$ & $36$ \\ \hline    			     					     
				    {\bf (II-2-2.4)} & $(S^2 \times S^2, 2x + 2y)$ & $-x$  & $S^2 \times S^2$  & \makecell{$Z_0 \cong S^2$, \\ $\mathrm{PD}(Z_0) = x+y$} & 
				    	$S^2 \times S^2$
					     & $4$ & $36$ \\ \hline    					     
				    {\bf (II-2-2.5)} & $(S^2 \times S^2, 2x + 2y)$ & $-x + y$  & $S^2 \times S^2$  & \makecell{$Z_0 \cong S^2 ~\dot \cup ~S^2$, \\ 
				    $\mathrm{PD}(Z_0^1) = \mathrm{PD}(Z_0^2) = x$} & 
				    	$S^2 \times S^2$
					     & $5$ & $36$ \\ \hline    			     					     
				    {\bf (II-2-2.6)} & $(S^2 \times S^2, 2x + 2y)$ & $-x + y$  & $S^2 \times S^2$  & \makecell{$Z_0 \cong S^2$, \\ $\mathrm{PD}(Z_0) = x$} & 
				    	$S^2 \times S^2$
					     & $4$ & $40$ \\ \hline    					     					     
				    {\bf (II-2-2.7)} & $(S^2 \times S^2, 2x + 2y)$ & $0$  & $S^2 \times S^2$ & \makecell{$Z_0 \cong S^2$, \\ $\mathrm{PD}(Z_0) = x+y$} & 
				    	$S^2 \times S^2$
					     & $4$ & $38$ \\ \hline    					     
				    {\bf (II-2-2.8)} & $(S^2 \times S^2, 2x + 2y)$ & $0$  & $S^2 \times S^2$ & \makecell{$Z_0 \cong S^2$, \\ $\mathrm{PD}(Z_0) = x$} & 
				    	$S^2 \times S^2$
					     & $4$ & $42$ \\ \hline    					     
				    {\bf (II-2-2.9)} & $(S^2 \times S^2, 2x + 2y)$ & $y$  & $S^2 \times S^2$ & \makecell{$Z_0 \cong S^2$, \\ $\mathrm{PD}(Z_0) = x$} & 
				    	$S^2 \times S^2$
					     & $4$ & $44$ \\ \hline    					     					     					     
				    {\bf (II-2-3.1)} & $(X_1, 3u - E_1)$ & $-u$  & $X_1$  & \makecell{$Z_0 \cong S^2$, \\ $\mathrm{PD}(Z_0) = 2u$} & 
				    	$X_1$
					     & $4$ & $32$ \\ \hline    
				    {\bf (II-2-3.2)} & $(X_1, 3u - E_1)$ & $-E_1$  & $X_1$  & 
				    \makecell{$Z_0 \cong S^2 ~\dot \cup ~ S^2$, \\ $\mathrm{PD}(Z_0^1) = u$ \\ $\mathrm{PD}(Z_0^2) = E_1$} & 
				    	$X_1$
					     & $5$ & $36$ \\ \hline    
				    {\bf (II-2-3.3)} & $(X_1, 3u - E_1)$ & $0$  & $X_1$  & \makecell{$Z_0 \cong S^2 $, \\ 
				    $\mathrm{PD}(Z_0) = u$} & 
				    	$X_1$
					     & $4$ & $40$ \\ \hline    			     					     
				    {\bf (II-2-3.4)} & $(X_1, 3u - E_1)$ & $u - E_1$  & $X_1$  & \makecell{$Z_0 \cong S^2$, \\ $\mathrm{PD}(Z_0) = E_1$} & 
				    	$X_1$
					     & $4$ & $46$ \\ \hline    					     
				    {\bf (II-2-3.5)} & $(X_1, 3u - E_1)$ & $0$  & $X_1$  & \makecell{$Z_0 \cong S^2 $, \\ 
				    $\mathrm{PD}(Z_0) = E_1$} & 
				    	$X_1$
					     & $4$ & $44$ \\ \hline    			     					     					     
	\end{tabular}
	\end{adjustbox}
		\vs{0.3cm}
		\caption {List of topological fixed point data for $\dim Z_{\min} = \dim Z_{\max} = 4$} \label{table_list_2} 
	\end{table}

\bibliographystyle{annotation}

\end{document}